\documentclass{amsart}
\usepackage{amsthm,amsfonts,amssymb,amsmath,color}
\usepackage{amsbsy}
\usepackage{wasysym,verbatim}
\numberwithin{equation}{section}

\renewcommand\d{\partial}
\renewcommand\a{\alpha}
\renewcommand\b{\beta}

\newcommand\s{\sigma}

\newcommand\R{\mathbb R}\newcommand\N{\mathbb N}

\def\g{\gamma}
\def\de{\mathfrak{z}}

\def\O{\Omega}
\def\th{\theta}

\def\l{\lambda}
\def\vp{\varphi}
\def\epsilon{\varepsilon}
\def\e{\varepsilon}

\newcommand\br{\begin{rem}}
\newcommand\er{\end{rem}}
\newcommand\bp{\begin{pmatrix}}
\newcommand\ep{\end{pmatrix}}
\newcommand\be{\begin{equation}}
\newcommand\ee{\end{equation}}
\newcommand\ba{\begin{equation}\begin{aligned}}
\newcommand\ea{\end{aligned}\end{equation}}

\newcommand{\TT}{{\mathbb T}}
\newcommand{\KK}{{\mathbb K}}
\newcommand{\CC}{{\mathbb C}}
\newcommand{\AAA}{{\mathbb A}}

\newcommand{\BB}{{\mathbb B}}
\newcommand{\II}{{\mathbb I}}
\newcommand{\DD}{{\mathbb D}}

\newcommand{\SSS}{{\mathbb S}}

\newcommand{\OO}{{\mathbb O}}
\newcommand{\PP}{{\mathbb P}}
\newcommand{\QQ}{{\mathbb Q}}

\newcommand{\YY}{{\mathbb Y}}

\newcommand{\tr}{{\rm tr }}

\newcommand{\vv}{{\mathbf v}}
\newcommand{\ff}{{\mathbf f}}

\newcommand{\ww}{{\mathbf w}}

\newcommand{\vvarphi}{{\boldsymbol \varphi}}

\newcommand{\Ov}[1]{\overline{#1}}

\newcommand{\DC}{C^\infty_c}

\newcommand{\vr}{\varrho}

\newcommand{\vu}{\vc{u}}

\newcommand{\vc}[1]{{\bf #1}}
\newcommand{\Div}{{\rm div}}
\newcommand{\Grad}{\nabla_x}

\newcommand{\dx}{{\rm d} {x}}
\newcommand{\dt}{{\rm d} t }
\newcommand{\dq}{{\rm d} q }

\newcommand{\intO}[1]{\int_{\O} #1 \, \dx}

\newtheorem{definition}{Definition}[section]
\newtheorem{theorem}[definition]{Theorem}
\newtheorem{proposition}[definition]{Proposition}
\newtheorem{lemma}[definition]{Lemma}

\newtheorem{remark}[definition]{Remark}

\numberwithin{equation}{section}
\usepackage[margin=1.7cm]{geometry}

\def\ocirc#1{\ifmmode\setbox0=\hbox{$#1$}\dimen0=\ht0
    \advance\dimen0 by1pt\rlap{\hbox to\wd0{\hss\raise\dimen0
    \hbox{\hskip.2em$\scriptscriptstyle\circ$}\hss}}#1\else
    {\accent"17 #1}\fi}

\begin{document}

\title[Global weak solutions to a compressible Oldroyd-B
model]{Existence of large-data finite-energy global weak solutions\\to a
compressible Oldroyd-B model}

\author[John W. Barrett]{John W. Barrett}
\address{Department of Mathematics, Imperial College London, London SW7 2AZ, UK}
\email{j.barrett@imperial.ac.uk}

\author[Yong Lu]{Yong Lu}
\address{Faculty of Mathematics and Physics, Charles University, Sokolovsk\'a 83, 186 75 Prague, Czech Republic}
\email{luyong@karlin.mff.cuni.cz}

\author[Endre S\"{u}li]{Endre S\"{u}li}
\address{Mathematical Institute, University of Oxford\\ Andrew Wiles Building, Woodstock Rd., Oxford OX2 6GG, UK}
\email{suli@maths.ox.ac.uk}


\date{}%2015.1.15

\begin{abstract}

A compressible Oldroyd--B type model with stress diffusion is derived from a compressible Navier--Stokes--Fokker--Planck system arising in the kinetic theory of dilute polymeric fluids, where polymer chains immersed in a barotropic, compressible, isothermal, viscous Newtonian solvent, are idealized as pairs of massless beads connected with Hookean springs.
We develop \textit{a priori} bounds for the model, including a logarithmic bound,
which guarantee the nonnegativity of the elastic extra stress tensor, and we prove the existence of large data global-in-time finite-energy weak solutions in two space dimensions.
\end{abstract}

\maketitle

%\tableofcontents

\renewcommand{\refname}{References}

%%%%%%%%%%%%%%%%%%%%%%%%%%%%%%%%%%%%%%%%%%%%%%%%%%%%%%%%%%%%%%%%%%%%%%%%%%%%%%%%%%%%%%%%%%

~\vspace{-4mm}

\section{Introduction}

Micro-macro models of dilute polymeric fluids that arise from statistical physics are based on coupling the Navier--Stokes system to the Fokker--Planck equation. In these models polymer molecules are idealized as chains of massless beads, linearly connected with inextensible rods or elastic springs. In the simplest case of two massless beads connected with a single Hookean spring the elastic spring-force is assumed to be a linear function of the conformation vector $q \in \mathbb{R}^d$, $d \in \{2,3\}$, describing the orientation of the spring, and the model is referred to as the \emph{Hookean dumbbell model}. An interesting aspect of the Hookean dumbbell model is that it has a (formal) macroscopic closure in the sense that the macroscopic evolution equation for the elastic extra stress tensor associated with the classical Oldroyd-B model with stress diffusion can be deduced from it by multiplying the Fokker--Planck equation with the rank-1 matrix $q \otimes q:=q q^{\rm T}$, integrating with respect to $q$ over a ball $B(0,R) \subset \mathbb{R}^d$ of radius $R>0$, and performing (formal) partial integrations, where contour/surface integrals over $\partial B(0,R)$ are set to zero in the limit of $R \rightarrow +\infty$, by postulating that the (nonnegative) probability density function satisfying the Fokker--Planck equation decays to $0$
sufficiently rapidly as $|q| \rightarrow \infty$.

In \cite{Barrett-Suli} and \cite{BS2016},  Barrett \& S\"uli proved the existence of large data global-in-time finite-energy weak solutions to a compressible Navier--Stokes--Fokker--Planck system, where the solvent was assumed to be a barotropic, compressible, isothermal, viscous Newtonian fluid confined to a bounded domain $\O\subset \R^d$, $d \in \{2,3\}$, and where the elastic spring force was, instead of a Hookean spring potential, modelled by a finitely extensible nonlinear elastic (FENE-type) spring potential. In \cite{FLS16} the results of \cite{Barrett-Suli} were extended to compressible
Navier--Stokes--Fokker--Planck systems with viscosity coefficients that depend on the polymer number density.
The aim of the present paper is to explore the existence of weak solutions to a fully macroscopic model, the compressible Oldroyd-B system, which arises, upon the formal macroscopic closure described above, from a compressible Navier--Stokes--Fokker--Planck system, with polymer chains idealized as Hookean dumbbells. The main contribution of the paper is the proof, in the case of two space dimensions ($d=2$), of the existence of large data global-in-time finite-energy weak solutions to this model.
In the case of the \textit{incompressible} Oldroyd-B model with stress diffusion in two space dimensions
the existence of large data global weak solutions was shown by Barrett \& Boyaval \cite{Barrett-Boyaval},
and the existence and uniqueness of a global strong solution,
again in two space dimensions, was proved by Constantin and Kliegl \cite{CK}.
The question of existence of large data global weak solutions to both the \textit{incompressible} and the \textit{compressible} Oldroyd-B model with stress diffusion remains a nontrivial open problem in the case of $d=3$.
We note in passing that in the \textit{incompressible} case the existence of
global weak solutions to the Navier--Stokes--Fokker--Planck system with Hookean dumbbells was recently proved in \cite{Barrett-Suli-in-prep} for $d=2$, as part of the research programme initiated in the series of papers \cite{Barrett-Suli2,Barrett-Suli4,Barrett-Suli1}; for $d=3$ the problem is open, and the problem is also open in the \textit{compressible} case for both $d=2$ and $d=3$. The results of the present paper may however lead one to speculate that in the case of $d=2$ at least the question of existence of large data global-in-time finite-energy weak solutions to the Hookean dumbbell model for the compressible Navier--Stokes--Fokker--Planck system can also be answered positively.

For the moment we shall keep the presentation general, with $\Omega \subset \mathbb{R}^d$ assumed to be a bounded open domain with $C^{2,\beta}$ boundary (briefly, a $C^{2,\beta}$ domain), with $\beta \in (0,1)$, and $d \in \{2,3\}$. In subsequent instances, whenever we are forced to restrict ourselves to the case of $d=2$ this restriction will be clearly stated. We consider the following compressible Oldroyd-B model, posed in the time-space cylinder $(0,T]\times \O$:
\begin{alignat}{2}
\label{01a}
\d_t \vr + \Div_x (\vr \vu) &= 0,
\\
\label{02a}
\d_t (\vr\vu)+ \Div_x (\vr \vu \otimes \vu) +\nabla_x p(\vr)  -
\Div_x\, \SSS(\nabla_x \vu) &=\Div_x \big(\TT - (kL\eta + \de\, \eta^2)\,\II\, \big)  +  \vr\, \ff,
\\
\label{03a}
\d_t \eta + \Div_x (\eta \vu) &= \e \Delta_x \eta,
\\
\label{04a}
\d_t \TT + {\rm Div}_x (\vu\,\TT) - \left(\nabla_x \vu \,\TT + \TT\, \nabla_x^{\rm T} \vu \right) &= \e \Delta_x \TT + \frac{k\,A_0}{2\lambda}\eta  \,\II - \frac{A_0}{2\lambda} \TT,
\end{alignat}
where the pressure $p$ and the density $\vr$ of the solvent are supposed to be related by the typical power law relation:
\be\label{pressure}
 p(\vr)=a \vr^\gamma, \quad a>0, \ \gamma >\frac{d}{2},
\ee
and the {\em Newtonian stress tensor} $\!\!~\SSS(\nabla_x \vu)$ is defined by
\be\label{Newtonian-tensor}
\SSS(\nabla_x \vu) = \mu^S \left( \frac{\nabla_x \vu + \nabla^{\rm T}_x \vu}{2} - \frac{1}{d} (\Div_x \vu) \II \right) + \mu^B (\Div_x \vu) \II,
\ee
with constant shear and bulk viscosity coefficients, respectively, $\mu^S>0$ and $\mu^B\geq 0$. The velocity gradient matrix is defined as
\be\label{def-nabla-u}
( \nabla_x \vu )_{1\leq i,j\leq d}= (\d_{x_j} \vu_i)_{1\leq i,j\leq d}.
\ee
The symmetric matrix function $\TT = (\TT_{\kappa,\iota})$, $1\leq \kappa,\iota\leq d$, defined on $(0,T]\times \O$, is the extra stress tensor and the notation ${\rm Div}_x(\vu\,\TT)$ is defined by
\be\label{def-Div-tau}
\left({\rm Div}_x(\vu\,\TT)\right)_{\kappa,\iota} = \Div_{ x}(\vu\,\TT_{\kappa,\iota}), \quad 1\leq \kappa,\iota\leq d.
\ee
The meaning of the various quantities and parameters appearing in \eqref{01a}--\eqref{04a} will be introduced in the derivation of the model in Section \ref{sec:der-Oldroyd-B}. In particular, the parameters $\e$, $k$, $A_0$, $\l$ are all positive numbers, whereas $\de \geq 0$ and $L\geq 0$ with $\de + L \neq 0$.
We note in passing that, in contrast with the \textit{compressible} Oldroyd-B model considered here, in the case of the \textit{incompressible} Oldroyd-B model the term $\Div_x \big((kL\eta + \de\, \eta^2)\,\II\, \big)$ appearing on the right-hand side of \eqref{02a} plays no particular role in the proof of the existence of global weak solutions and can be absorbed into the pressure
term $\nabla_x p$ on the left-hand side of the equation.

The equations \eqref{01a}--\eqref{04a} are supplemented
by initial conditions for $\vr$, $\vu$, $\eta$ and $\TT$,
and the following boundary conditions:
\begin{alignat}{2}
\label{05a}
\vu&=\mathbf{0}   &&\quad \mbox{on}\ (0,T]\times \partial\Omega,
\\
\label{06a}
\d_{\bf n} \eta  &=0 &&\quad \mbox{on}\  (0,T]\times \d\O,
\\
\label{07a}
\d_{\bf n} \TT &=0 &&\quad \mbox{on}\  (0,T]\times \d\O.
\end{alignat}
Here $\d_{\bf n}:= {\bf n}\cdot \nabla_x$, where ${\bf n}$ is the outer unit normal vector on the boundary $\d\O$,
and the external force $\ff$ is assumed to be an element of the function space $L^\infty((0,T]\times \O;\R^d)$.

Our proof is based on several levels of regularization, the first of which
involves supplementing \eqref{02a} by an additional term, including the regularization parameter $\alpha>0$,
and replacing $\eta$ in \eqref{04a} by $\eta +\alpha$. The procedure results in the following regularized compressible Oldroyd-B model,
posed on $(0,T] \times \O$:
\begin{alignat}{2}
\label{01}
\d_t \vr + \Div_x (\vr \vu) &= 0,
\\
\label{02}
\d_t (\vr\vu)+ \Div_x (\vr \vu \otimes \vu) +\nabla_x p(\vr) + \nabla_x\big(kL\eta+\de\,\eta^2\big) &-
\Div_x\, \SSS(\nabla_x \vu) =\Div_x \TT  - \frac{\alpha}{2}\,\nabla_x  \tr\left(\log \TT \right)+  \vr\, \ff,
\\
\label{03}
\d_t \eta + \Div_x (\eta \vu) &= \e \Delta_x \eta,
\\
\label{04}
\d_t \TT + {\rm Div}_x (\vu\,\TT) - \left(\nabla_x \vu \,\TT + \TT\, \nabla_x^{\rm T} \vu \right) &= \e \Delta_x \TT + \frac{k\,A_0}{2\lambda}(\eta+\alpha)  \,\II - \frac{A_0}{2\lambda} \TT.
\end{alignat}
The equations \eqref{01}--\eqref{04}
are once again supplemented
by initial conditions for $\vr$, $\vu$, $\eta$ and $\TT$,
and the boundary conditions (\ref{05a})--(\ref{07a}).
The regularization term on the right-hand side of (\ref{02a}) presupposes that $\TT$ is
symmetric positive definite, but this will be proved rigorously below in the case
of $d=2$,
provided that $\TT$ is
symmetric positive definite at $t=0$.
In the final step of the proof, we shall pass to the limit $\alpha \rightarrow 0$ with the regularization parameter $\alpha$.

The paper is organized as follows. In Section \ref{sec:der-Oldroyd-B}, we shall derive the compressible Oldroyd-B model \eqref{01a}--\eqref{07a} from the compressible Navier--Stokes--Fokker--Planck system in the Hookean dumbbell setting.  In the cental part of the paper, between Section \ref{sec:a-priori} and the first part of Section \ref{sec:completion-proof}, we shall focus on the regularized model
\eqref{01}--\eqref{04}, with $\a>0,\ \de>0$, and the global-in-time existence of weak solutions in two-dimensional space will be proved in this case. In the second part of Section \ref{sec:completion-proof}, we will show the global-in-time existence of weak solutions in two-dimensional space to the original model \eqref{01a}--\eqref{07a} when $\de>0$, by passing to the limit $\a\to 0$. Finally, the existence result in the case of $\de=0$ will be established in Section \ref{sec:de-to-0}, by passing to the limit $\de\to 0$. We note that the condition $L>0$ is only needed in the passage to the limit $\de\to 0$; in other words, as long as $\de>0$, it suffices to assume that $L \geq 0$.

The mathematical analysis of compressible viscoelastic fluid flow models has been the subject of active research in recent years.
The existence and uniqueness of local strong solutions and the existence of global solutions near
equilibrium for macroscopic models of three-dimensional compressible viscoelastic fluids was considered in \cite{Hu-Wang1,Qian-Zhang,
Qian,Hu-Wang2,Hu-Wang3,Hu-Wu}.
Fang and Zi \cite{Fang-Zi}
proved the existence of a unique local strong solution to a compressible Oldroyd-B model for all initial data satisfying a certain compatibility condition, and established a blow-up criterion for strong solutions.
Lei \cite{Lei} proved the local and global existence of classical solutions to a compressible Oldroyd-B system in a torus with small initial data; he also studied the incompressible limit problem and showed that solutions to the compressible flow model with well-prepared initial data converge to those of the incompressible model when the Mach number converges to zero.
Guillop\'{e}, Salloum, and Talhouk \cite{GST} investigated weakly compressible viscoelastic fluids
satisfying the Oldroyd constitutive law;
they obtained \textit{a priori} estimates that are uniform in the Mach number, which then allowed them to prove that weakly compressible flows with well-prepared initial data converge to incompressible flows when the Mach number converges to zero.
The existence of measure-valued solutions to
non-{N}ewtonian compressible, isothermal, monopolar fluid flow models was studied
by Ne\v{c}asov\'{a} in \cite{Necasova1,Necasova2};
for bipolar isothermal non-{N}ewtonian compressible
fluids related analysis was pursued in \cite{Necasova3}.
In a series of papers (cf. \cite{Mamontov1,Mamontov2,Mamontov3})
Mamontov developed \textit{a priori} estimates for two- and three-dimensional
compressible nonlinear viscoelastic flow problems and studied the existence of solutions.
Zhikov \& Pastukhova \cite{Zhik-Past} proved the existence of global weak solutions to a
class of
compressible viscoelastic flow models with $p$-Laplacian structure. There is also a substantial literature in chemical
engineering on the use of the compressible Oldroyd-B system in modelling bubble dynamics in compressible viscoelastic
liquids (cf., for example, \cite{brujan}).
Bae \& Trivisa \cite{BaeKon} have established the existence of global weak solutions to
Doi's rod-model in three-dimensional bounded domains; the model concerns suspensions of
rod-like molecules in compressible fluids and involves the coupling of a
Fokker--Planck type equation with the compressible Navier--Stokes system. In a related context,
Jiang, Jiang \& Wang \cite{Jiang-Jiang-Wang} have studied the existence of
global weak solutions to the equations of compressible flow of nematic liquid crystals in two dimensions.
For a survey of macroscopic models of compressible viscoelastic flow, the reader is referred to the paper by Bollada \& Phillips \cite{Boll-Phil-2012}. As was noted there, even for isothermal viscoelastic models, the transition from the incompressible to the compressible case is nontrivial; in fact, the precise form of temperature-dependence in compressible viscoelastic models is not yet properly understood, the development of complete, thermodynamically consistent, models being the subject of ongoing research. We shall therefore confine ourselves here to the isothermal setting, with the temperature assumed to be held fixed.

\section{Derivation of the compressible Oldroyd-B model}\label{sec:der-Oldroyd-B}

In this section we recall the compressible Navier--Stokes--Fokker--Planck system considered in \cite{Barrett-Suli1}.
We shall then (formally) derive from it the compressible Oldroyd-B model \eqref{01a}--\eqref{07a}
by considering the special case of the compressible Hookean dumbbell model and formulating its (formal) macroscopic
closure.

\subsection{Compressible Navier--Stokes--Fokker--Planck system}\label{sec:derivation-NSFP}
The solvent density $\vr$ and the solvent velocity field $\vu$ are defined in  $(0,T]\times \O$ and  $(0,T]\times \overline{\O}$, respectively, with $T>0$, and satisfy the compressible Navier--Stokes equations with an elastic extra stress-tensor $\KK$:
\begin{alignat}{2}\label{NS1}
\d_t \vr + \Div_x (\vr \vu) &= 0\qquad &&\mbox{in $(0,T] \times \Omega$},
\\
\label{NS2}
 \d_t (\vr\vu)+ \Div_x (\vr \vu \otimes \vu) +\nabla_x p(\vr) -
\Div_x\, \SSS(\nabla_x \vu) &=\Div_x \KK +\vr\, \ff \qquad
&&\mbox{in $(0,T] \times \Omega$}.
\end{alignat}
The pressure $p(\vr)$ and the {\em Newtonian shear stress tensor} $\SSS$ are defined by \eqref{pressure} and \eqref{Newtonian-tensor}.
We shall impose a no-slip boundary condition on the velocity field; i.e.,
\be\label{NS4}
\vu=\mathbf{0}\quad \mbox{on}\ (0,T]\times \partial\Omega.
\ee

\medskip

In \emph{a bead-spring chain model} consisting of $K+1$ beads coupled with $K$ elastic springs representing a polymer chain,  the non-Newtonian elastic extra stress tensor $\KK$ is defined by a version of the Kramers expression (cf. \eqref{def-TT0} below), depending on the probability density function $\psi$, which, in addition to $t$ and $x$, also depends on the conformation vector $(q_1^{\rm T},\dots, q_K^{\rm T})^{\rm T} \in \R^{dK}$, with $q_i$ representing the $d$-component \emph{conformation/orientation vector} of the $i$th spring in the chain. Let $D:=D_1\times \cdots \times D_K \subset \R^{dK}$ be the domain of admissible conformation vectors. Typically $D_i$ is the whole space $\R^d$ or a bounded open ball centered at the origin $0$ in $\R^d$, for each $i=1,\dots,K$. When $K=1$, the model is referred to as the \emph{dumbbell model}. Here we consider the Hookean bead-spring chain model, where $D_i = \R^d$ for all $i \in \{1,\dots,K\}$, and the \emph{elastic spring-force} $F_i : q_i \in D_i \mapsto  U_i'(\frac{1}{2}|q_i|^2)q_i \in \R^d$ and the \emph{spring potential} $U_i : \R_{\geq 0} \to \R_{\geq 0} $ of the $i$th spring in the chain are defined by
\be\label{def-Fi-Ui}
F_i(q_i)=q_i\ \mbox{for all $q_i\in D_i$},\quad U_i(s)=s \ \mbox{for all $s\geq 0$}, \quad \mbox{$i=1,\ldots,K$}.
\ee

The {\em extra-stress tensor} $\KK$ is defined by the formula:
\be\label{def-TT0}
\KK (\psi)(t,x) := \KK_1 (\psi) (t,x) -\left(\int_{D\times D} \gamma(q,q')\, \psi(t,x,q)\,\psi(t,x,q') \, \dq\,\dq'\right)\II,
\ee
where, similarly to \cite{Barrett-Suli,FLS16}, the interaction kernel $\g$ is assumed to be a nonnegative constant $\gamma(q,q')\equiv \de \geq 0.$
Consequently,
\be\label{def-TT}
\KK (\psi) := \KK_1 (\psi)  -\de \left(\int_{D}  \psi \,\dq\right)^2\II .
\ee
The first part, $\KK_1(\psi)$, of $\KK(\psi)$ is given by the {\em Kramers expression}
\be\label{def-TT1}
\KK_1(\psi) := k \left[\left(\sum_{i=1}^K \CC_i(\psi)\right)- L \left(\int_{D}  \psi \ \dq\right)\II\right],
\ee
where $k>0$ is the product of the Boltzmann constant and the absolute temperature, $L=K+1$ is the number of beads in the polymer chain in the classical Kramers expression (in our setting $L$ can be taken to be any nonnegative real number as long as $\de>0$; in order to cover the case of $\de=0$ in the final step of our proof we pass to the limit $\de\to 0$, and this requires that $L>0$ in this step), and
\be\label{def-Ci}
\CC_{i}(\psi)(t,x) := \int_D \psi(t,x,q)\, U_i'\bigg(\frac{|q_i|^2}{2}\bigg)\, q_i q_i^{\rm T}\,\dq, \quad i=1,\dots,K.
\ee

By noting \eqref{def-Fi-Ui}, one deduces from \eqref{def-TT0}--\eqref{def-Ci} that in the Hookean case
\be\label{def-TT-f}
\KK (\psi) = \TT -\left(k L  \eta  + \de\,  \eta^2\right) \II ,
\ee
where
\be\label{def-tau-eta}
\TT(t,x):=k\sum_{i=1}^K \int_D \psi(t,x,q)\, q_i q_i^{\rm T}\,\dq,\quad \eta(t,x):= \int_{D} \psi(t,x,q)\, \dq,
\ee
with the quantity $\eta$ being called the \emph{polymer number density}. We thus arrive at the momentum equation \eqref{02a}. Since $\psi$ is a probability density function, and
therefore nonnegative a.e. on $[0,T] \times \Omega \times D$ and $\int_{\Omega \times D} \psi(t,x,q)\, \dx\, \dq = 1$ for a.e.
$t \in [0,T]$, it is clear from \eqref{def-tau-eta} that $\TT(t,x)$ is symmetric and
nonnegative definite for a.e. $(t,x) \in [0,T] \times \Omega$.

We introduce the {\em partial Maxwellian} $M_i : D_i \to [0,\infty)$ by
\[\label{def-Mi}
M_i(q_i) := \frac{1}{Z_i} {\exp}\left(- U_i\left(\frac{1}{2}|q_i|^2\right)\right),\quad \mbox{where } Z_i:=\int_{D_i} \exp\bigg(- U_i\left(\frac{1}{2}|p_i|^2\right)\bigg)\,{\rm d}p_i, \qquad i=1,\dots,K.
\]
The {\em Maxwellian} $M : D \to [0,\infty)$ is then defined as the product of the $K$ partial Maxwellians: i.e., for any $q=(q_1^{\rm T},\dots, q_K^{\rm T})^{\rm T}$ contained in $D = D_1\times\cdots \times D_K$, we have that
\[M(q) := \prod_{i=1}^K M_i(q_i).\]
Clearly,
$
\int_D M(q)\,\dq=1.
$

The probability density function $\psi$ satisfies the following {\em Fokker--Planck equation} in $(0,T]\times \O\times D$:
\ba\label{eq-psi}
\d_t \psi + \Div_x (\vu\,\psi) + \sum_{i=1}^K \Div_{q_i} \left(  (\nabla_x \vu)\, q_i\, \psi \right)= \e \Delta_x \psi+\frac{1}{4\lambda} \sum_{i=1}^K\sum_{j=1}^K A_{ij}\, \Div_{q_i}\!\left( M \nabla_{q_j} \left( \frac{\psi}{M} \right)\right).
\ea
A simple calculation reveals that in the case of Hookean springs, when $U_i(\frac{1}{2}|q_i|^2) = \frac{1}{2}|q_i|^2$,
 $i=1,\dots,K$, the expression appearing in the second term on the right-hand side of \eqref{eq-psi} can be rewritten as follows:
\ba\label{eq-psi-lastterm}
 M \nabla_{q_j} \left( \frac{\psi}{M} \right) =\nabla_{q_j} \psi + \psi \, q_j.
\ea

The \emph{centre-of-mass diffusion} term $\e \Delta_x \psi$ is generally of the form
$
\e \Delta_x \left( \frac{\psi}{\zeta(\vr)} \right)$,
which involves the {\em drag coefficient} $\zeta(\cdot)$ depending on the fluid density $\vr$. Here we assume that $\zeta$ is a constant function, which is, for simplicity, taken to be identically $1$. The constant parameter $\e>0$ is the {\em centre-of-mass diffusion coefficient}. The parameter $\l>0$ is called the \emph{Deborah number}; it characterizes the elastic relaxation property of the fluid. The constant matrix $A=(A_{ij})_{1\leq i,j\leq K}$, called the \textit{Rouse matrix}, is symmetric and positive definite. We denote by $A_0$ the smallest eigenvalue of $A$; clearly, $A_0>0$.

The Fokker--Planck equation needs to be supplemented by suitable boundary conditions; in the Hookean case considered here, with $D=\R^{Kd}$, these are:
\ba\label{boundary-psi}
 \psi \,|q_j| &\to 0,\quad \nabla_{q_j}\psi\cdot \frac{q_j}{|q_j|}\to 0, \quad &&\mbox{as}\ |q_j|\to \infty, \quad \mbox{for all \ $(t,x) \in (0,T] \times \O, \quad j=1,\ldots,K$},\\
 \d_{\bf n} \psi &=0 \quad &&\mbox{on}\ (0,T]\times \d\O \times D.
\ea

Finally, by (formally) integrating the partial differential equation \eqref{eq-psi} over $D$ and using the boundary condition in $\eqref{boundary-psi}_1$,
 and by integrating the boundary condition $\eqref{boundary-psi}_2$ over $D$, we deduce the following partial differential equation and boundary condition for the function $\eta$:
\be\label{eq-eta}
\d_t \eta + \Div_x (\vu \,\eta) = \e \Delta_x \eta\quad \mbox{in}\ (0,T] \times \O;\qquad \d_{\bf n} \eta =0 \quad \mbox{on}\  (0,T]\times \d\O.
\ee

The \emph{compressible Navier--Stokes--Fokker--Planck system} in the case of Hookean bead-spring chains consists of \eqref{NS1}, \eqref{NS2}, \eqref{eq-psi}, \eqref{eq-eta}, supplemented with the boundary conditions in \eqref{NS4}, \eqref{boundary-psi}, \eqref{eq-eta}, and suitable initial conditions
for $\vr$, $\vu$, $\psi$ and $\eta$. In the next section we focus on the special case of this model, when $K=1$, and use formal computations to derive the compressible Oldroyd-B model whose analysis is thereafter pursued in the rest of the paper.

\subsection{Compressible Oldroyd-B model}

This section is devoted to the derivation of the model \eqref{01a}--\eqref{07a}
from the Navier--Stokes--Fokker--Planck system stated in Section \ref{sec:derivation-NSFP}, consisting of equations \eqref{NS1}, \eqref{NS2}, \eqref{eq-psi}, \eqref{eq-eta}, and the boundary conditions  \eqref{NS4}, \eqref{boundary-psi}, \eqref{eq-eta}. Since the special case of the dumbbell model, corresponding to $K=1$, is sufficiently representative from the point of view of highlighting the main technical difficulties, for the sake of simplicity of the exposition we shall assume following equation \eqref{eq-psi-tau-5-3} below that $K=1$. Up until that point we shall admit $K\geq 1$ so as to illuminate the connection with the defining expression \eqref{def-tau-eta} for the tensor $\TT$.

The continuity equation \eqref{01a} for the fluid density $\vr$  and the equation \eqref{03a} for the polymer number density $\eta$, as well as the boundary condition \eqref{06a} for $\eta$, follow directly from \eqref{NS1} and \eqref{eq-eta}. From \eqref{NS2} and \eqref{def-TT-f}, we deduce the balance of momentum equation \eqref{02a}; the boundary condition for the velocity field $\vu$ follows from \eqref{NS4}.

The boundary condition \eqref{07a} can be deduced from \eqref{boundary-psi} and \eqref{def-tau-eta}. It remains to derive the evolution equation \eqref{04a} for the elastic extra stress tensor $\TT$. To this end, we note the definition of $\TT$ in \eqref{def-tau-eta}, multiply equation \eqref{eq-psi} by the matrix $k\sum_{i=1}^K q_i q_i^{\rm T}$, and integrate it with respect to $q \in D$. In the following we shall calculate the results, term by term. We will see that in the special case of $K=1$ the resulting evolution equation for the extra stress tensor $\TT$ is precisely \eqref{04a}.

We begin by noting that for the first term in \eqref{eq-psi}, associated with the time derivative, we have that
\be\label{eq-psi-tau-1}
\int_D \d_t \psi \left(k\sum_{i=1}^K q_i q_i^{\rm T}\right)  \dq = \d_t  \int_D \psi \left(k\sum_{i=1}^K q_i q_i^{\rm T}\right)  \dq = \d_t \TT.
\ee
For the second term in \eqref{eq-psi}, we have that
\ba\label{eq-psi-tau-2}
\int_D \Div_x (\vu\,\psi)\left(k\sum_{i=1}^K q_i q_i^{\rm T}\right)  \dq & = \int_D \left[(\Div_x \vu) \,\psi+ (\vu\cdot \nabla_x)\psi\right]\left(k\sum_{i=1}^K q_i q_i^{\rm T}\right)  \dq\\
& =(\Div_x \vu)  \int_D \psi \left(k\sum_{i=1}^K q_i q_i^{\rm T}\right)  \dq +(\vu\cdot \nabla_x) \int_D \psi\left(k\sum_{i=1}^K q_i q_i^{\rm T}\right)  \dq\\
& =(\Div_x \vu)  \TT  +  (\vu\cdot \nabla_x) \TT\\
&=\left(\Div_{ x}(\vu\,\TT_{\kappa,\iota})\right)_{1\leq \kappa,\iota\leq d}=:{\rm Div}_x (\vu\,\TT).
\ea
For the third term in \eqref{eq-psi}, for any $1\leq \kappa,\iota\leq d$, we have,
with $q_j^\kappa$ being the $\kappa^{\rm th}$ component of $q_j$,
 that
\ba\label{eq-psi-tau-3}
&\int_D \sum_{i=1}^K \Div_{q_i} \left(  (\nabla_x \vu)\, q_i\, \psi \right) \left(k\sum_{j=1}^K q_j^\kappa q_j^\iota\right)  \dq \\
 &= -k \int_D \sum_{i=1}^K \sum_{j=1}^K \left(  (\nabla_x \vu)\, q_i\, \psi \right)\cdot \nabla_{q_i} \left( q_j^\kappa q_j^\iota\right)  \dq\\
& = - k \int_D \sum_{i=1}^K \sum_{\a,\b=1}^d\left(  \d_{x_\b} \vu_\a \, q_i^{\beta} \, \psi \right)  \d_{q_i^\a} \left( q_i^\kappa q_i^\iota\right)  \dq\\
& =- k \int_D \sum_{i=1}^K \sum_{\a,\b=1}^d\left(  \d_{x_\b} \vu_\a \, q_i^{\beta} \, \psi \right)   \left( q_i^\kappa \delta_{\a,\iota} + q_i^\iota \delta_{\a,\kappa}\right)  \dq\\
& =- k \int_D \sum_{i=1}^K \sum_{\b=1}^d \left(  \d_{x_\b} \vu_\iota \, q_i^{\beta} \, \psi \,  q_i^\kappa\right) \dq - k \int_D \sum_{i=1}^K \sum_{\b=1}^d\left(  \d_{x_\b} \vu_\kappa \, q_i^{\beta} \, \psi \,q_i^\iota \right)  \dq\\
%& = - \sum_{\b=1}^d (\d_{x_\b} \vu_\kappa)  \left(\int_D k \,\psi \sum_{i=1}^K  q_i^{\beta} q_i^\iota\,\dq\right) %- \left(\sum_{\b=1}^d  \int_D k \,\psi \sum_{i=1}^K  q_i^{\beta} q_i^\kappa\,\dq\right)( \d_{x_\b} \vu_\iota)\\
&= - \left(\nabla_x \vu \,\TT\right)_{\iota,\kappa} - \left(\nabla_x \vu \,\TT\right)_{\kappa,\iota}\\
&= - \left(\nabla_x \vu \,\TT\right)_{\kappa,\iota} - \left(\TT\, \nabla_x^{\rm T} \vu \right)_{\kappa,\iota}.
\ea
For the fourth term in \eqref{eq-psi} we have that
\be\label{eq-psi-tau-4}
\int_D \e \Delta_x \psi \left(k\sum_{i=1}^K q_i q_i^{\rm T}\right)  \dq = \e \Delta_x  \int_D \psi \left(k\sum_{i=1}^K q_i q_i^{\rm T}\right)  \dq = \e \Delta_x \TT.
\ee

It remains to deal with the last term in \eqref{eq-psi}.
By \eqref{eq-psi-lastterm}, for any $1\leq i, j\leq K$, $1\leq \kappa, \iota \leq d$, we have by (formal) integration by parts and ignoring the ``boundary" terms at $|q|=\infty$,
\ba\label{eq-psi-tau-5-0}
\int_D \Div_{q_i}\left( M \nabla_{q_j} \left( \frac{\psi}{M} \right)\right) \left(k\sum_{h=1}^K q_h^\kappa q_h^\iota\right)& =\int_D \Div_{q_i}\left( \nabla_{q_j} \psi + \psi \, q_j \right) \left(k\sum_{h=1}^K q_h^\kappa q_h^\iota\right) \dq\\
&= - k \int_D \left( \nabla_{q_j} \psi + \psi \, q_j \right) \cdot \nabla_{q_i} \left(\sum_{h=1}^K q_h^\kappa q_h^\iota\right) \dq\\
&= - k \int_D \left( \nabla_{q_j} \psi + \psi \, q_j \right) \cdot \nabla_{q_i} \left(q_i^\kappa q_i^\iota\right) \dq.
\ea
We compute the last term in \eqref{eq-psi-tau-5-0}, which can be decomposed into two terms. The first one is, after a second (formal) partial integration,
\ba\label{eq-psi-tau-5-1}
-k \int_D \left( \nabla_{q_j} \psi \right) \cdot \nabla_{q_i} \left(q_i^\kappa q_i^\iota\right) \dq=  k \int_D \psi \Delta_{q_i} \left(q_i^\kappa q_i^\iota\right)  \delta_{i,j}\,\dq = 2k \,\delta_{i,j} \,\delta_{\kappa,\iota}\int_D \psi\, \dq = 2k \,\eta\,\delta_{i,j} \,\delta_{\kappa,\iota}.
\ea
The second one is
\ba\label{eq-psi-tau-5-2}
- k \int_D \left(  \psi \, q_j \right) \cdot \nabla_{q_i} \left(q_i^\kappa q_i^\iota\right) \dq &= - k \int_D \sum_{\a=1}^d \psi \, q_j^\a \, \d_{q_i^\a} \left(q_i^\kappa q_i^\iota\right) \dq\\
&= - k \int_D \sum_{\a=1}^d \psi \, q_j^\a \, \left(q_i^\kappa\delta_{\a,\iota} + q_i^\iota \delta_{\a,\kappa}\right) \dq\\
&= - k \int_D \psi \,  \left(q_i^\kappa q_j^\iota  + q_i^\iota q_j^\kappa\right) \dq.
\ea

Thus, by \eqref{eq-psi-tau-5-0}--\eqref{eq-psi-tau-5-2}, we have that
\ba\label{eq-psi-tau-5-3}
&\int_D  \frac{1}{4\lambda} \sum_{i=1}^K\sum_{j=1}^K A_{ij}\, \Div_{q_i}\left( M \nabla_{q_j} \left( \frac{\psi}{M} \right)\right) \left(k\sum_{h=1}^K q_h q_h^{\rm T}\right) \dq \\
& =\frac{1}{4\lambda} \sum_{i=1}^K\sum_{j=1}^K A_{ij}\,(-k) \int_D \left( \nabla_{q_j} \psi + \psi \, q_j \right) \cdot \nabla_{q_i} \left(q_i q_i^{\rm T}\right) \dq \\
& =\frac{1}{4\lambda} \sum_{i=1}^K\sum_{j=1}^K A_{ij}\,(2k\,\eta \,\delta_{i,j}) \,\II - \frac{1}{4\lambda} \sum_{i=1}^K\sum_{j=1}^K A_{ij}\,k\,\int_D \psi \, \left(q_i q_j^{\rm T} + q_j q_i^{\rm T}\right) \dq\\
& =\frac{1}{4\lambda} {\rm tr}\, (A)\,(2k\,\eta ) \,\II - \frac{1}{4\lambda} \sum_{i=1}^K\sum_{j=1}^K A_{ij}\,k\,\int_D \psi \, \left(q_i q_j^{\rm T} + q_j q_i^{\rm T}\right) \dq,
\ea
where $\tr(\cdot)$ denotes the trace of a square matrix.
Finally, by restricting ourselves at this point to the dumbbell model, where $K=1$, we have that $A = A_0>0$ is a constant, and
we thus obtain
\ba\label{eq-psi-tau-5}
&\int_D \frac{1}{4\lambda} \sum_{i=1}^K\sum_{j=1}^K A_{ij}\, \Div_{q_i}\left( M \nabla_{q_j} \left( \frac{\psi}{M} \right)\right) \left(k\sum_{h=1}^K q_h q_h^{\rm T}\right) \dq \\
& =\frac{1}{4\lambda}\,A_0 \,(2k\,\eta ) \,\II - \frac{1}{4\lambda} A_0 \,(2\TT)\\
& =\frac{k\,A_0}{2\lambda}\eta  \,\II - \frac{A_0}{2\lambda} \TT.
\ea
By combining \eqref{eq-psi-tau-1}%, \eqref{eq-psi-tau-2}, \eqref{eq-psi-tau-3},
--\eqref{eq-psi-tau-4} and \eqref{eq-psi-tau-5}, we then obtain
\[
\d_t \TT + {\rm Div}_x(\vu\,\TT) - \left(\nabla_x \vu \,\TT + \TT\, \nabla_x^{\rm T} \vu \right) = \e \Delta_x \TT + \frac{k\,A_0}{2\lambda}\eta  \,\II - \frac{A_0}{2\lambda} \TT,
\]
which is precisely \eqref{04a}.
We note that if $\TT$ solves \eqref{04a} for a given $\vu$, then taking the transpose of
\eqref{04a} we see that $\TT^{\rm T}$ solves \eqref{04a}. Hence the symmetry of $\TT$
from the definition \eqref{def-tau-eta} is encoded in the macroscopic equation
\eqref{04a}.

Having derived the model, we now focus our attention on its mathematical analysis. We begin by establishing \textit{a priori} bounds that will form the basis of the weak compactness argument leading to the proof of existence of a weak solution to the system under consideration.

Finally, we note from (\ref{eq-psi-tau-5-3}) that in the case of $K \geq 2$
one would need to introduce the family of symmetric tensors
$\TT^{i,j}(t,x) := k\,\int_D \psi(t,x,q)\,q_i q_j^{\rm T} \dq \in \mathbb{R}^{d \times d}$ for $i, j=1,\ldots,K$,
to obtain a closed macroscopic model; see
\cite{Barrett-Suli-in-prep} for a discussion of this in the incompressible case.

\section{A priori bounds}\label{sec:a-priori}

This section is devoted to the derivation of \textit{a priori} bounds for the regularized compressible Oldroyd-B model \eqref{01}--\eqref{04} with $\a,\,\de >0$ subject to the boundary conditions
(\ref{05a})--(\ref{07a}), and given proper initial conditions.

\subsection{Initial data and \textit{a priori} bound}

%Let $(\vr,\vu,\eta,\TT)$ be a smooth solution to \eqref{01}--\eqref{04}, (\ref{05a})--(\ref{07a}).
We adopt the following hypotheses on the initial data:
\ba\label{ini-data}
&\vr(0,\cdot) = \vr_0(\cdot) \ \mbox{with}\ \vr_0 \geq 0 \ {\rm a.e.} \ \mbox{in} \ \O, \quad \vr_0 \in L^\gamma(\O),\\
&\vu(0,\cdot) = \vu_0(\cdot) \in L^r(\O;\R^d) \ \mbox{for some $r\geq 2\g'$}\ \mbox{such that}\ \vr_0|\vu_0|^2 \in L^1(\O),\\
&\eta(0,\cdot)=\eta_0 \ \mbox{with}\ \eta_0 \geq 0 \ {\rm a.e.} \ \mbox{in} \ \O, \quad \eta_0 \in L^2 (\O),\\ %\quad \int_\Omega \eta_0(x)\,\dx = 1,
&\TT(0,\cdot) = \TT_0(\cdot) \ \mbox{with}\ \TT_0=\TT_0^{\rm T} > 0 \ \mbox{a.e. in}  \ \O ,\quad  \TT_0 \in L^2(\O; \R^{d \times d}), \ \tr(\log \TT_0) \in L^1(\O).
\ea
Here $\g'$ denotes the conjugate exponent to $\g>1$, i.e., $1/\g + 1/\g'=1$, and
$\TT_0 > 0$ signifies that $\TT_0$ is positive definite.
%$\mbox{tr}$ signifies trace (of a square matrix).

Because the density $\vr$ is required to be a nonnegative function,
we have assumed that the initial datum $\vr_0$
is nonnegative.
Since the probability density function $\psi$ is, by definition, nonnegative,
then the definitions of $\eta$ and $\TT$
stated in  \eqref{def-tau-eta}
for the Navier--Stokes--Fokker--Planck system automatically imply that
$\eta$ must be a nonnegative function and
$\TT$ must be a symmetric nonnegative definite matrix a.e.\ on $(0,T] \times
\Omega$.
However, this information concerning the nonnegativity  of $\TT$
is not \textit{a priori} encoded in the
macroscopic counterpart of
this kinetic model, the compressible Oldroyd-B model \eqref{01a}--\eqref{07a}.
Furthermore, because of the presence of the logarithmic term in the alpha-regularised model,
see (\ref{02}), we
require $\TT > 0$ a.e.\ in $(0,T] \times \Omega$.
We have therefore assumed nonnegativity/positivity of the initial data for $\eta$ and $\TT$,
respectively, in \eqref{ini-data}.
For the purposes of the formal energy estimates developed in this section, we will temporarily
\textit{assume}
that $(\vr,\vu,\eta,\TT)$ is a smooth
solution to \eqref{01}--\eqref{04}, (\ref{05a})--(\ref{07a}), (\ref{ini-data})
with $\a,\,\de >0$, and, in addition, that
$\vr \geq 0$, $\eta >0$ and $\TT >0$ a.e.\ in $(0,T] \times \Omega$.
We stress that the energy estimates below, and these nonnegativity and positivity constraints
on $\vr$, $\eta$ and $\TT$, will be made rigorous in the case of $d=2$ later in the paper.

We deduce from $\eqref{ini-data}_1$ and $\eqref{ini-data}_2$ by using H\"older's inequality that
\[
%\be\label{ini-data-complement}
(\vr\vu)(0,\cdot) = \vr_0 \vu_0 =\sqrt{\vr_0}\sqrt{\vr_0} \vu_0 \in L^{\frac{2\g}{\g+1}}(\O;\R^d).
%\ee
\]

\medskip

For the fluid density $\vr$, integration of \eqref{01} over $\O$ with respect to the spatial variable $x$, performing partial
integration and noting the no-slip boundary condition \eqref{05a} for the velocity field gives
\[
\frac{\rm d}{\dt} \int_\O \vr(t,x)\,\dx =0 ~~~\Longrightarrow~~~ \int_\O \vr(t,x)\,\dx = \int_\O \vr_0\,\dx,
\qquad  \mbox{$t \in (0,T]$}.
\]

For the polymer number density $\eta$, integrating \eqref{03} over $\O$ with respect
to the spatial variable $x$, performing
partial integration and noting, in addition to (\ref{05a}),
the homogeneous Neumann boundary condition \eqref{06a} on $\eta$ gives
\be\label{a-priori-eta}
\frac{\rm d}{\dt} \int_\O \eta(t,x)\,\dx =0 ~~~\Longrightarrow~~~ \int_\O \eta(t,x)\,\dx = \int_\O \eta_0\,\dx,
\qquad \mbox{$t \in (0,T]$}.
\ee

In order to derive a formal energy identity we take the inner product of the momentum equation \eqref{02} with the velocity field $\vu$, integrate over $\O$ with respect to the spatial variable $x$, and perform partial integration noting the  no-slip boundary condition \eqref{05a} for $\vu$. In order to explain the details of the calculation, we shall perform the computations term by term. We begin by noting that for the first term in \eqref{02} we have
\[
\int_\O \d_t(\vr \vu) \cdot \vu \,\dx = \int_\O (\d_t\vr) |\vu |^2\,\dx +  \int_\O \vr\, \d_t \frac{|\vu|^2}{2}  \,\dx = \frac{1}{2}\,\frac{\rm d}{\dt} \int_\O   \vr |\vu|^2\,\dx + \frac 12 \int_\O (\d_t\vr) |\vu |^2\,\dx,
\qquad  \mbox{$t \in (0,T]$}.
\]
For the second term in \eqref{02},
\begin{align*}
\int_\O \Div_{ x}(\vr \vu \otimes \vu) \cdot \vu \,\dx &= - \int_\O (\vr \vu \otimes \vu) : \nabla_x\vu \,\dx = -\sum_{i,j=1}^d \int_\O \vr \vu_i \vu_j\, \d_{x_j}\vu_i \,\dx\\
& = - \sum_{i,j=1}^d \int_\O \vr  \vu_j\, \frac{1}{2}\d_{x_j}|\vu_i|^2 \,\dx
= \frac{1}{2} \int_\O \Div_{ x}(\vr \vu)\, | \vu |^2\,\dx,
\qquad  \mbox{$t \in (0,T]$}.
\end{align*}
For the third term in \eqref{02},
\[
\int_\O \nabla_x p(\vr) \cdot \vu \,\dx = -\int_\O (a \vr^\gamma)\, \Div_{ x} \vu \,\dx,
\qquad  \mbox{$t \in (0,T]$}.
\]
Multiplication of \eqref{01} by $\g\vr^{\gamma-1}$ gives
\[
\d_t \vr^\g + \Div_x (\vr^\g \vu) + (\g-1) \vr^\g \Div_{ x}\vu = 0.
%\qquad  \mbox{$t \in (0,T]$}.
\]
Thus, thanks to the boundary condition \eqref{05a} and our assumption that $\gamma>1$, we have that
\[
\int_\O \nabla_x p(\vr) \cdot \vu \,\dx = \int_\O \frac{a}{\g-1} \left(\d_t \vr^\g + \Div_x (\vr^\g \vu)\right) \dx = \frac{a}{\g-1}\,\frac{\rm d}{\dt} \int_\O  \vr^\g\,\dx,
\qquad  \mbox{$t \in (0,T]$}.
\]
For the fourth term in \eqref{02},
\[
\int_\O \nabla_x\left(kL\eta+\de\,\eta^2\right) \cdot \vu \,\dx = -kL \int_\O \eta\, \Div_{ x} \vu \,\dx -\de \int_\O\eta^2\, \Div_{ x} \vu \,\dx,
\qquad  \mbox{$t \in (0,T]$}.
\]
Let $b: [0,\infty) \to \R $ be a continuous function and $b\in C^1(0,\infty)$.
Recalling our assumption $\eta>0$, multiplication of \eqref{03} by $b'(\eta)$
yields that
\be\label{eq-eta-renomal}
\d_t b(\eta) + \Div_x \big(b(\eta) \vu\big) + \big(b'(\eta) \eta - b(\eta)\big) \Div_{ x}\vu = \e b'(\eta)\,\Delta_x \eta.
%\qquad  \mbox{$t \in (0,T]$}.
\ee
By choosing $b(\eta):=\eta \log \eta + 1$,
%and noting that $b$ is a strictly positive function on  $(0,\infty)$
we obtain from \eqref{eq-eta-renomal} that
\[
\d_t (\eta \log \eta + 1) + \Div_x \big((\eta \log \eta + 1) \vu\big) +  \eta\, \Div_{ x}\vu - \Div_{ x} \vu  = \e (1+\log \eta)\,\Delta_x \eta.
%\qquad  \mbox{$t \in (0,T]$}.
\]
Thanks to the boundary conditions \eqref{05a} and \eqref{06a} we then have that
\ba\label{a-priori-vu-4-1}
 -k L\int_\O \eta\, \Div_{ x} \vu \,\dx = k L\frac{\rm d}{\dt} \int_\O  (\eta \log \eta + 1) \,\dx
 + \e \,k L\int_\O  \frac{1}{\eta} |\nabla_x \eta|^2 \,\dx,
\qquad  \mbox{$t \in (0,T]$}.
\ea
By choosing $b(\eta):=\eta^2$ in \eqref{eq-eta-renomal} we obtain from \eqref{eq-eta-renomal} that
\[
\d_t (\eta^2) + \Div_x (\eta^2 \vu )
+  \eta^2 \Div_{ x}\vu  = 2 \e \eta\, \Delta_x \eta.
\]
Hence, thanks again to the boundary conditions \eqref{05a} and \eqref{06a}, we obtain
\ba\label{a-priori-vu-4-2}
 -\de \int_\O \eta^2\, \Div_{ x} \vu \,\dx = \de \,\frac{\rm d}{\dt} \int_\O \eta^2  \,\dx +
2\, \e \,\de\,
\int_\O  |\nabla_x \eta|^2 \,\dx,
\qquad  \mbox{$t \in (0,T]$}.
\ea
By combining \eqref{a-priori-vu-4-1} and \eqref{a-priori-vu-4-2}, the fourth term in
\eqref{02} can be therefore rewritten as follows:
\begin{align*}
\int_\O \nabla_x\left(k L \eta+\de\,\eta^2\right) \cdot \vu \,\dx &= \frac{\rm d}{\dt} \int_\O \left(k L (\eta \log \eta + 1) + \de \,\eta^2 \right)\dx\\
 &\qquad + \int_\O \left( \frac{k L}{\eta}
 +  2\de\right) \e\, |\nabla_x \eta|^2 \dx,
\qquad  \mbox{$t \in (0,T]$}.
\end{align*}
For the fifth term in \eqref{02} we have
\begin{align*}
-\int_\O \Div_{ x}\, \SSS(\nabla_x \vu) \cdot \vu \,\dx
&= \int_\O \left(\mu^S \left( \frac{\nabla_x \vu + \nabla^{\rm T}_x \vu}{2} - \frac{1}{d} (\Div_x \vu) \II \right) + \mu^B (\Div_x \vu) \II\right) : \nabla_x \vu \,\dx\\
&= \int_\O \mu^S \left| \frac{\nabla_x \vu + \nabla^{\rm T}_x \vu}{2} - \frac{1}{d} (\Div_x \vu) \II \right|^2 + \mu^B |\Div_x \vu|^2\,\dx,
\qquad  \mbox{$t \in (0,T]$}.
\end{align*}
For the sixth term in \eqref{02} we have
\begin{align*}
\int_\O \Div_{ x}\TT \cdot \vu \,\dx &= - \int_\O \TT : \nabla_x \vu \,\dx,
\qquad  \mbox{$t \in (0,T]$}.
\end{align*}
For the seventh term in \eqref{02} we have
\begin{align*}
-\frac{\a}{2}\,\int_\O \nabla_x  \tr\left(\log \TT \right)  \cdot \vu \,\dx &= \frac{\a}{2} \int_\O  \tr\left(\log \TT \right)  (\Div_{ x} \vu) \,\dx,
\qquad  \mbox{$t \in (0,T]$}.
\end{align*}
Therefore, by summing up the above identities, we deduce, on noting (\ref{01}), that
\ba\label{a-priori-vu}
&\frac{\rm d}{\dt} \int_\O \left[ \frac{1}{2} \vr |\vu|^2 + \frac{a}{\g-1} \vr^\g + \left(k L  (\eta \log \eta + 1) + \de \,\eta^2\right)\right]\dx \\
&\quad +  \int_\O \left( \frac{k L}{\eta}
 +  2\de\right) \e\, |\nabla_x \eta|^2\,\dx + \int_\O \mu^S \left| \frac{\nabla_x \vu + \nabla^{\rm T}_x \vu}{2} - \frac{1}{d} (\Div_x \vu) \II \right|^2 + \mu^B |\Div_x \vu|^2\,\dx\\
&= - \int_\O \TT : \nabla_x \vu \,\dx + {  \frac{\a}{2} \int_\O  \tr\left(\log \TT \right)
\Div_{ x} \vu \,\dx } + \int_\O \vr\,\ff \cdot  \vu \,\dx,
\qquad  \mbox{$t \in (0,T]$}.
\ea

\medskip

In order to complete the derivation of a (formal) energy identity for the system, we need to deal with the first two terms on the right-hand side of equation \eqref{a-priori-vu}. As far as the first term on the right-hand side of
\eqref{a-priori-vu} is concerned, by taking the trace of equation \eqref{04}, integrating over $\O$ with respect to the spatial variable $x$, and using the boundary conditions \eqref{05a} and \eqref{07a}, we deduce that
\ba\label{a-priori-tau}
\frac{\rm d}{\dt} \int_\O \tr \left(\TT\right) \,\dx  + \frac{A_0}{2\l}\int_\O \tr\left(\TT
\right) \,\dx
= \frac{k\,A_0\,d}{2\lambda} \int_\O (\eta+\alpha)  \,\dx  +  2 \int_\O \TT : \nabla_x \vu \,\dx,
\qquad  \mbox{$t \in (0,T]$}.
\ea
Here, we have also noted that $\tr(\PP\,\QQ^{\rm T}) = \PP:\QQ$
for all $\PP,\,\QQ \in {\mathbb R}^{n \times n}$.
Therefore, $\eqref{a-priori-vu}+\frac{1}{2}\eqref{a-priori-tau}$ gives
\ba\label{a-priori-1}
&\frac{\rm d}{\dt} \int_\O \left[ \frac{1}{2} \vr |\vu|^2 + \frac{a}{\g-1} \vr^\g + \left(k L  (\eta \log \eta + 1) + \de \,\eta^2\right)+ \frac{1}{2}\tr\left(\TT\right) \right]\dx \\
&\quad +  \int_\O \left( \frac{k L}{\eta}
 +  2\de\right) \e\, |\nabla_x \eta|^2 \,\dx + \int_\O \mu^S \left| \frac{\nabla_x \vu + \nabla^{\rm T}_x \vu}{2} - \frac{1}{d} (\Div_x \vu) \II \right|^2 + \mu^B |\Div_x \vu|^2\,\dx + \frac{A_0}{4\l}\int_\O \tr\left(\TT\right) \,\dx\\
&=  \int_\O \vr\,\ff \cdot  \vu \,\dx + \frac{k\,A_0\,d}{4\lambda} \int_\O (\eta+\alpha)  \,\dx + {  \frac{\a}{2} \int_\O  \tr\left(\log \TT \right)  \Div_{ x} \vu \,\dx},
\qquad  \mbox{$t \in (0,T]$}.
\ea
We are left to deal with the final term on the right-hand side of \eqref{a-priori-1} (which is the same as the second term on the right-hand side of \eqref{a-priori-vu}); this requires some nontrivial calculations, which we shall next discuss.

\subsection{A logarithmic bound}\label{sec:est-log}

In this section, inspired by the study of the incompressible Oldroyd-B model in
\cite{Hu-Lelievre, Bris-Lelievre, Barrett-Boyaval}, we derive a logarithmic bound on the extra stress tensor $\TT$ and a bound on $\TT^{-1}$ . These are needed both to complete
the derivation of the (formal) energy identity for the system, as well as in the construction of the sequence of approximating solutions.
As we shall see below, the computations aimed at dealing with the final term on the right-hand side of \eqref{a-priori-1} will yield a term in the (formal) energy identity for the system
\eqref{01}--\eqref{04}, (\ref{05a})--(\ref{07a}), (\ref{ini-data}),
which will ultimately ensure the positivity of $\TT$ a.e. on $\Omega \times (0,T]$.
For the purposes of the formal calculations that will now follow,
we temporarily \textit{assume} that $\TT$ is a symmetric positive definite matrix.

Let $\TT^{-1}=(\TT^{-1})_{1\leq \kappa,\iota\leq d}$ be the inverse of $\TT$. We compute the inner product of equation \eqref{04} and $\TT^{-1}$. We recall the following formula, usually referred to as Jacobi's formula: 
\be\label{Jacobi-formula}
\d (\det \TT) = (\det \TT)\, \tr\,(\TT^{-1}\d  \TT); \qquad \mbox{hence} %, if $Q=Q^{\rm T}>0$, also}
\quad \d \left(\log \det \TT \right)=\tr\,(\TT^{-1}\d \TT)= \d \TT:\TT^{-1},
\ee
where, in the present context, $\d$ is a derivative in space or time.
Since $\TT $ is symmetric positive definite, we can define its real logarithm, $\log \TT$, which is a symmetric matrix such that ${\rm e}^{\log \TT}=\TT$. Indeed, upon diagonalization of $\TT$, using the orthogonal  $d \times d$ matrix $\OO$, we have that
\[
 \TT  = \OO \,{\rm diag\,}\{\l_1,\l_2,\ldots, \l_d\} \,\OO^{\rm T}
\qquad \mbox{and therefore}\qquad
\log \TT = \OO \, {\rm diag\,}\{\log\l_1,\log \l_2,\ldots, \log\l_d\}\, \OO^{\rm T},
\]
where $\l_\kappa>0$, $\kappa=1,\ldots, d$, are the eigenvalues of $\TT$.
Thus, we have the following identity:
\be\label{tr-log-tau}
\tr\left(\log \TT\right) = \log \det \TT.
\ee
By \eqref{Jacobi-formula} and \eqref{tr-log-tau}, we have for the first term in \eqref{04} that
\be\label{est-log-1}
\d_t \TT : \TT^{-1} = \d_t \left(\log \det \TT \right) = \d_t \tr\left(\log \TT \right).
\ee
For the second term in \eqref{04} we have that
\ba\label{est-log-2}
{\rm Div}_x (\vu\,\TT)  : \TT^{-1} & = \left( (\vu\cdot \nabla_x)\, \TT + (\Div_{ x} \vu)\,\TT  \right)
:\TT^{-1}
= (\vu \cdot \nabla_x)  \tr\left(\log \TT \right) + d\,\Div_{ x}   \vu.
\ea
For the third term in \eqref{04} we have
\begin{align*}
&- \left(\nabla_x \vu \,\TT + \TT\, \nabla_x^{\rm T} \vu \right)  : \TT^{-1}
= - \tr\left(\left(\nabla_x \vu \,\TT + \TT\, \nabla_x^{\rm T} \vu \right)\TT^{-1}\right)
= -2\,\tr(\nabla_x \vu)
=-2\, \Div_{ x} \vu.
\end{align*}
Thus, by taking the inner product of equation \eqref{04} with $\TT^{-1}$, we obtain
\ba\label{eq-log-tau}
\d_t \tr\left(\log \TT \right) + (\vu \cdot \nabla_x) \tr\left( \log \TT \right) + (d-2) \,\Div_{ x} \vu = \e\, \Delta_x \TT : \TT^{-1} +  \frac{k\,A_0}{2\lambda}(\eta+\alpha)  \,\tr\left(\TT^{-1}\right)
 - \frac{d\,A_0}{2\lambda}.
\ea
After integrating \eqref{eq-log-tau} over $\O$, performing partial integration and noting the boundary condition \eqref{05a} on $\vu$, we have that, for $t \in (0,T]$,
\ba\label{est-log-f1}
\frac{\rm d}{\dt} \int_\O \tr\left(\log \TT  \right) \dx =
\int_\O (\Div_{ x}\vu)\, \tr\left(\log \TT \right) \dx + \int_\O \e\, \Delta_x \TT : \TT^{-1}\,\dx  + \frac{k\,A_0}{2\lambda}
\int_\O (\eta+\alpha)
\,\tr\left(\TT^{-1}\right)\,\dx - \frac{d\,A_0}{2\lambda}|\O|.
\ea
To proceed, we require the following lemma whose proof is elementary but rather lengthy and has been therefore relegated to Appendix \ref{appendixa}.
\begin{lemma}\label{lem-lap-tau-inverse}
Let $\PP \in W^{2,2}(\Omega;\mathbb{R}^{d \times d})
\cap C^1(\overline{\Omega};\R^{d\times d})$ be a symmetric matrix function,
which is positive definite, uniformly on $\overline{\Omega}$,
and satisfies $\d_{\bf n} \PP =0 \mbox{ on } \d\O$;
then,
\ba\label{est-log-lap-tau}
\int_\O \Delta_x \PP : \PP^{-1}\,\dx &= \sum_{j=1}^d  \int_\O
\tr \left(\left((\d_{x_j} \PP)(\PP^{-1})\right)^2\right) \dx
\geq \frac{1}{d}  \int_\O \left|\nabla_x \tr\left(\log \PP \right)\right|^2 \dx.
\ea
\end{lemma}

We continue our formal calculations under the assumption that $\TT(t,\cdot)$ satisfies the hypotheses of Lemma
\ref{lem-lap-tau-inverse} for $t \in (0,T]$.
By subtracting  $\frac{\a}{2}\eqref{est-log-f1}$ from the \textit{a priori} bound \eqref{a-priori-1} and using
\eqref{est-log-lap-tau}, we finally obtain the following (formal) energy identity:
\ba\label{a-priori-log-f2}
&\frac{\rm d}{\dt} \int_\O \left[ \frac{1}{2} \vr |\vu|^2 + \frac{a}{\g-1} \vr^\g + \left(k L  (\eta \log \eta + 1) + \de \,\eta^2\right)+ \frac{1}{2}\tr\left(\TT -
\a \,\log \TT \right) \right]\dx \\
&\quad +  \int_\O \left( \frac{k L}{\eta}
 +  2\de\right) \e\, |\nabla_x \eta|^2 \,\dx + \frac{\a\,\e}{2}\sum_{j=1}^d  \int_\O \tr \left(\left((\d_{x_j} \TT)(\TT^{-1})\right)^2\right) \dx\\
&\quad + \int_\O \mu^S \left| \frac{\nabla_x \vu + \nabla^{\rm T}_x \vu}{2} - \frac{1}{d} (\Div_x \vu) \II \right|^2 + \mu^B |\Div_x \vu|^2\,\dx + \frac{A_0}{4\l}\int_\O \tr\left(\TT\right) \,\dx
+ \frac{\a\,k\,A_0}{4\lambda} \int_\O (\eta+\alpha)  \,\tr\left(\TT^{-1}\right)\,\dx \\
&=   \int_\O \vr\,\ff \cdot  \vu \,\dx + \frac{k\,A_0\,d}{4\lambda} \int_\O (\eta+\alpha)  \,\dx + \frac{\a\,d\,A_0}{4\lambda}|\O|,
\qquad  \mbox{$t \in (0,T]$}.
\ea
This (formal) energy identity will be the starting point for the development of the weak compactness argument leading to the proof of existence of a global-in-time large data finite-energy weak solution to the compressible Oldroyd-B system under consideration.

To this end, we make some preliminary observations. Let us denote by $\l_\kappa$, $\kappa=1,\ldots,d$, the
eigenvalues of the symmetric positive definite matrix $\TT$. Then,
\[
\tr \left(\TT- \a  \,\log \TT  \right) = \sum_{\kappa=1}^d \left(\l_\kappa - \a \log \l_\kappa\right)  \geq \sum_{\kappa=1}^d \left(\a - \a \log \a\right) = d\left(\a - \a \log \a\right).
\]
Hence, for any $\a>0$, we have that
 \ba\label{tau-log-tau2}
\tr \left(\TT - \a\, \log \TT  \right) + d\left( \a \log \a - \a \right) \geq 0.
\ea

Motivated by \eqref{a-priori-log-f2} and \eqref{tau-log-tau2}, for $t \in [0,T]$ we consider the following nonnegative energy functional:
\be\label{def-energy-free}
E(t): = \int_\O \left[ \frac{1}{2} \vr |\vu|^2 + \frac{a}{\g-1} \vr^\g + \left(k L  (\eta \log \eta + 1) + \de \,\eta^2\right)+ \frac{1}{2}\left(\tr \left(\TT- \a\, \log \TT \right) + d\left( \a \log \a - \a \right) \right) \right]\dx.
\ee
H\"older's inequality then gives
$$
 \int_\O \vr\,\ff \cdot  \vu \,\dx    \leq \|\ff\|_{L^\infty((0,T]\times \O;\R^d)} \|\sqrt\vr |\vu|\|_{L^2(\O;\R^d)} \|\sqrt\vr \|_{L^2(\O)} \leq C\,E(t),
\qquad  \mbox{$t \in (0,T]$},
$$
for some positive constant $C=C(\ff, a, \gamma)$.
Also, as $\eta(\log \eta -1) + 1 \geq 0$ for $\eta\geq 0$, we have that
\begin{align}\label{kL1}
\int_\O (\eta (t,x) +\alpha) \,\dx\leq \frac{\max\{1,\alpha |\Omega|\}}{\min\{1,kL\}}(E(t)+1),
\qquad  \mbox{$t \in (0,T]$}, \qquad \mbox{when $L>0$ and $\de \geq 0$};
\end{align}
similarly, as $\eta + \alpha \leq (\frac{1}{2} + \alpha) + \frac{1}{2\de}\,\de\eta^2$, by integrating this inequality over $\Omega$ we deduce that
\begin{align}\label{kL2}
\int_\O (\eta (t,x) +\alpha)\,\dx \leq \max\left\{\left(\frac{1}{2}+\alpha\right)|
\Omega|,\frac{1}{2\de}\right\}(E(t) + 1),
\qquad  \mbox{$t \in (0,T]$},\qquad \mbox{when $L\geq 0$ and
$\de>0$.}
\end{align}

Thus, integrating \eqref{a-priori-log-f2} over the time interval $[0,t]$ with respect to the temporal variable and noting (\ref{est-log-lap-tau}) implies that
\ba
\label{a-priori-2-0}
 E(t) &+ \int_0^t \int_\O \left( \frac{k L}{\eta}
 +  2\de\right) \e\, |\nabla_x \eta|^2 \,\dx\,\dt'  + \frac{\a\,\e}{2d}  \int_0^t  \int_\O \left|\nabla_{x} \tr\left(\log \TT \right)\right|^2\dx \, \dt' \\
 &+  \int_0^t  \int_\O \mu^S \left| \frac{\nabla_x \vu + \nabla^{\rm T}_x \vu}{2} - \frac{1}{d} (\Div_x \vu) \II \right|^2 + \mu^B |\Div_x \vu|^2\,\dx \,\dt'  \\
 &+ \frac{A_0}{4\l} \int_0^t  \int_\O \tr\left(\TT\right) \,\dx \,\dt'
 + \frac{\a\,k\,A_0}{4\lambda} \int_0^t  \int_\O (\eta+\alpha)  \,\tr\left(\TT^{-1}\right)\,\dx \,\dt' \\
 \leq  &\ E_0 + C \int_0^t E(t') \,\dt'  + C \,t,
\qquad  \mbox{$t \in (0,T]$},
%\ea
\ea
where, if $L>0$,  the positive constant $C$ depends only on
$\|\ff\|_{L^\infty((0,T]\times \O;\R^d)}$
and the parameters  $a,\gamma, k, d, L, \l, A_0, \a, |\O|$,
but it is independent of $\de\geq 0$; whereas if $\de>0$, then the positive constant $C$ depends only on $\|\ff\|_{L^\infty((0,T]\times \O;\R^d)}$ and the parameters $a,\gamma, k, d, \de, \l, A_0, \a, |\O|$,
but it is independent of $L\geq 0$.
The initial energy
\be\label{def-energy-free0}
E_0: = \int_\O \left[ \frac{1}{2} \vr_0 |\vu_0|^2 + \frac{a}{\g-1} \vr_0^\g +
\left(k L  (\eta_0 \log \eta_0 + 1) + \de \,\eta_0^2\right)+
\frac{1}{2}\left(\tr\left(\TT_0-  \a\, \log \TT_0 \right) + d\left( \a \log \a - \a \right) \right) \right]\dx
\ee
is finite thanks to the assumptions on the initial data stated in \eqref{ini-data}. Thus, Gronwall's inequality implies that
\ba\label{a-priori-2}
 E(t) &+ \int_0^t \int_\O \left( \frac{k L}{\eta}
 +  2\de\right) \e\, |\nabla_x \eta|^2 \,\dx\,\dt'  + \frac{\a\,\e}{2d}  \int_0^t  \int_\O \left|\nabla_{x} \tr\left(\log \TT \right)\right|^2\,\dx \, \dt' \\
 &+  \int_0^t  \int_\O \mu^S \left| \frac{\nabla_x \vu + \nabla^{\rm T}_x \vu}{2} - \frac{1}{d} (\Div_x \vu) \II \right|^2 + \mu^B |\Div_x \vu|^2\,\dx \,\dt'  \\
 &+ \frac{A_0}{4\l} \int_0^t  \int_\O \tr\left(\TT\right) \,\dx \,\dt' + \frac{\a\,k\,A_0}{4\lambda} \int_0^t  \int_\O (\eta+\a)  \,\tr\left(\TT^{-1}\right)\,\dx \,\dt' \\
\leq &\   (E_0+C\,t)\, {\rm e}^{Ct},
\qquad  \mbox{$t \in (0,T]$},
\ea
where, if $L>0$,  the positive constant $C$ depends only on
$\|\ff\|_{L^\infty((0,T]\times \O;\R^d)}$ and the parameters  $a,\gamma,
k, d, L, \l, A_0, \a, |\O|$, but it is independent of $\de\geq 0$;
whereas if $\de>0$, then the positive constant $C$ depends only on
$\|\ff\|_{L^\infty((0,T]\times \O;\R^d)}$ and the parameters $a,\gamma, k, d, \de, \l, A_0, \a, |\O|$, but it is independent of $L\geq 0$.

\medskip

Next we recall
Korn's inequality (see, for example,
\cite{Dain}): there exists a positive constant $C=C(d,\O)$ such that
\be\label{korn}
\|\nabla_x \vv \|_{L^2(\O;\R^{d\times d})} \leq C \left\| \frac{\nabla_x \vv
+ \nabla^{\rm T}_x \vv}{2} - \frac{1}{d} (\Div_x \vv) \II \right\|_{L^2(\O;\R^{d\times d})}
\qquad \forall\, \vv \in W_0^{1,2}(\Omega,\R^d).
\ee
Thus we deduce the following (formal) inclusions from the \textit{a priori} inequality \eqref{a-priori-2}:
\ba\label{a-priori-3}
&\vr\in L^\infty(0,T; L^\g(\O)), \quad \vu \in L^2(0,T; W^{1,2}_0(\O;\R^d)), \quad \vr|\vu|^2 \in L^\infty(0,T; L^1(\O)),\\
&\eta \in L^\infty(0,T; L^2(\O)) \cap L^2(0,T; W^{1,2}(\O)),\quad (\eta+\a)\, \tr\left(\TT^{-1}\right) \in L^1(0,T; L^1(\O)),\\
& \tr \left(\TT- \a\,\log \TT \right) \in L^\infty(0,T; L^1(\O)),
\quad \nabla_{x} \tr\left(\log \TT \right) \in L^2 (0,T; L^2(\O;\R^d)).
\ea

\subsection{A further bound in two space dimensions}\label{sec:2d-est}

In this section we will show that when $d=2$ one can establish stronger bounds on $\TT$ than those stated in
\eqref{a-priori-2} and \eqref{a-priori-3}.
The key step is to take the inner product of \eqref{04} with $\TT$ and integrate over $\O$ with respect to $x$.  Direct calculations imply that
\begin{align*}
&\frac{1}{2}\frac{\rm d}{\dt} \int_\O  | \TT|^2 \,\dx + \e \int_\O |\nabla_x \TT|^2 \,\dx + \frac{A_0}{2\l}\int_\O | \TT|^2 \,\dx \\
&\quad= - \int_\O {\rm Div}_x (\vu\,\TT) : \TT \,\dx + \int_\O \left(\nabla_x \vu \,\TT + \TT\, \nabla_x^{\rm T} \vu \right) : \TT \,\dx + \frac{k\,A_0}{2\lambda}\int_\O (\eta+\a) \,\tr\left(\TT\right) \,\dx \\
&\quad \leq \frac{1}{2}\int_\O  (\Div_{ x} \vu) \,  |\TT|^2 \,\dx + 2 \int_\O |\nabla_x \vu|\, |\TT|^2 \,\dx  + \frac{A_0}{8\l} \int_\O |\tr\left(\TT\right)|^2 \,\dx + \frac{k^2 A_0}{2\l}
\int_\O (\eta+\a)^2 \,\dx\\
&\quad \leq  3 \int_\O |\nabla_x \vu|\, |\TT|^2 \,\dx  + \frac{A_0}{4\l} \int_\O |\TT|^2 \,\dx +
\frac{k^2 A_0}{2\l} \int_\O (\eta+\a)^2 \,\dx,
\qquad  \mbox{$t \in (0,T]$}.
\end{align*}
Thus, for $t \in (0,T]$,
\ba\label{a-priori-tau-f2}
\frac{1}{2}\frac{\rm d}{\dt} \int_\O  | \TT|^2 \,\dx + \e \int_\O |\nabla_x \TT|^2 \,\dx + \frac{A_0}{4\l}\int_\O | \TT|^2 \,\dx  \leq  3\, \| \nabla_x \vu\|_{L^2(\O;\R^{2\times 2})}\, \|\TT\|_{L^4(\O;\R^{2\times 2})}^2    + \frac{k^2 \,A_0}{2\l} \int_\O (\eta+\a)^2 \,\dx.
\ea
We recall the following Gagliardo--Nirenberg inequality: let $G\subset\R^d$ be a bounded Lipschitz domain; then, for any $r\in [2,\infty)$ if $d=2$, and $r\in [2,6]$ if $d=3$, one has, for any $v\in W^{1,2}(G)$, that:
\be\label{G-N-ineq}
\|v\|_{L^r(G)} \leq C(r,d,\O) \|v\|_{L^2(G)}^{1-\th} \|v\|_{W^{1,2}(G)}^\th,\quad \th:=d\left(\frac{1}{2}-\frac{1}{r}\right).
\ee
Hence, in the case of $d=2$ and $G=\Omega \subset \mathbb{R}^2$, we have, for $t \in (0,T]$, that
\be\label{sobolev-R2}
\|\TT\|_{L^4(\O;\R^{2\times 2})}^2 \leq C\,\|\TT\|_{L^2(\O;\R^{2\times 2})} \| \TT\|_{W^{1,2}(\O;\R^{2\times 2})} \leq C\,\|\TT\|_{L^2(\O;\R^{2\times 2})}  \left(\|\TT\|_{L^2(\O;\R^{2\times 2})}+ \| \nabla_x\TT\|_{L^{2}(\O;\R^{2\times 2\times 2})}\right).
\ee
This implies, for $t \in (0,T]$, that
\be\label{sobolev-R2-a}
3\, \| \nabla_x \vu\|_{L^2(\O;\R^{2\times 2})}\, \|\TT\|_{L^4(\O;\R^{2\times 2})}^2 \leq C\,\| \nabla_x \vu\|_{L^2(\O;\R^{2\times 2})}^2 \|\TT\|_{L^2(\O;\R^{2\times 2})}^2 + \frac{A_0}{8\l}\|\TT\|_{L^2(\O;\R^{2\times 2})}^2+ \frac{\e}{2}  \|\nabla_x \TT\|_{L^2(\O;\R^{2\times 2 \times 2})}^2.
\ee
As $\de>0$, we have from the
\textit{a priori} bound (\ref{a-priori-2}) that
$\|\eta\|^2_{L^\infty(0,t;L^2(\O))}\leq C(t,E_0,\de^{-1})$.
We thus deduce from \eqref{a-priori-tau-f2} and \eqref{sobolev-R2-a}
that
\[
\frac{\rm d}{\dt} \int_\O  | \TT|^2 \,\dx + \e \int_\O |\nabla_x \TT|^2 \,\dx + \frac{A_0}{4\l}\int_\O | \TT|^2 \,\dx  \leq  C\, \| \nabla_x \vu\|_{L^2(\O)}^2 \|\TT\|_{L^2(\O)}^2
+ C(t,E_0,\de^{-1}). \,
\qquad  \mbox{$t \in (0,T]$}.
\]
Now, Gronwall's inequality implies that
\ba\label{a-priori-tau-f4}
\|\TT(t,\cdot)\|_{L^2(\O;\R^{2\times 2})}^2 \leq
{\rm e}^{C\int_0^t \| \nabla_x \vu(t',\cdot)\|_{L^2(\O;\R^{2\times 2})}^2\,\dt'}
\left( \|\TT_0\|_{L^2(\O;\R^{2\times 2})}^2 + C(t,E_0,\de^{-1}) \right).
\qquad  \mbox{$t \in (0,T]$}.
\ea
Finally, by invoking the \textit{a priori} bound \eqref{a-priori-2} again we deduce that
\ba\label{a-priori-tau-f5}
 \int_\O  | \TT(t)|^2 \,\dx + \e  \int_0^t \int_\O |\nabla_x \TT|^2  \,\dx\,\dt' + \frac{A_0}{4\l}\int_0^t \int_\O | \TT|^2 \dx\, \dt'  \leq C(t,E_0,\de^{-1},\|\TT_0\|_{L^2(\O;\R^{2\times 2})}^2),
\qquad  \mbox{$t \in (0,T]$}.
\ea
Thus, in the case of $d=2$  and $\de>0$,
we can supplement \eqref{a-priori-3} with the following stronger inclusion for $\TT$:
\[
\TT \in L^\infty(0,T; L^2(\O;\R^{2\times 2})) \cap L^2(0,T;W^{1,2}(\O;\R^{2\times 2})).
\]

Finally, we will also consider the case of  $d=2$ and $\de=0$ in the final step of our existence proof. To this end we require bounds on $\|\nabla_x \vu\|_{L^2(0,T;L^2(\Omega;\mathbb{R}^{2\times 2}))}$ and the
second term on the right-hand side of \eqref{a-priori-tau-f2} that are uniform in $\de$ as $\de\to 0$. The derivation of the required $\de$-uniform bounds is the subject of the following remark.

\begin{remark}\label{rem-eta2-tau}
The $\de$-uniform bound on $\|\nabla_x \vu\|_{L^2(0,T;L^2(\Omega;\mathbb{R}^{2 \times 2}))}$ is a direct consequence of \eqref{a-priori-2}, with $L>0$, and \eqref{korn}. To show that the constant on the right-hand side of \eqref{a-priori-tau-f5} is uniform as $\de\to 0$, by \eqref{a-priori-tau-f2}
it suffices  to show that the norm $\|\eta\|_{L^2(0,T;L^2(\O))}$ is uniformly bounded as $\de\to 0$.
As $|\eta \log \eta| \leq \eta \log \eta + 1$ for all $\eta \geq 0$, it follows from \eqref{a-priori-2}, considered in the case when $L>0$ and $\de \geq 0$, that
\ba\label{est-tau-de=0-1}
\| \eta  \log \eta \|_{L^\infty(0,T;L^1(\O))} + \| \nabla_x \eta ^{\frac{1}{2}} \|_{L^2(0,T;L^{2}(\O;\R^2))}  \leq C (T,E_0,L^{-1}),
\ea
where, for any (fixed) $L>0$, the constant $C(T,E_0,L^{-1})$ is bounded as $\de\to 0$. By direct computation
\[
\int_\O |\nabla_x \eta| \,\dx = \int_\O |2 \eta^{\frac{1}{2}}\nabla_x \eta^{\frac{1}{2}}| \,\dx \leq 2\, \| \eta^{\frac{1}{2}} \|_{L^2(\O)} \|\nabla_x \eta^{\frac{1}{2}} \|_{L^2(\O)} = 2\, \| \eta \|_{L^1(\O)}^{\frac{1}{2}} \|\nabla_x \eta^{\frac{1}{2}} \|_{L^2(\O)},
\qquad  \mbox{$t \in (0,T]$},
\]
and by \eqref{est-tau-de=0-1} we therefore have that
\[
\| \eta \|_{L^2(0,T;W^{1,1}(\O))}   \leq C (T,E_0,L^{-1}).
\]
As $d=2$, the Sobolev embedding of $W^{1,1}(\O)$ into $L^2(\Omega)$ then gives that
\[
\| \eta \|_{L^2(0,T;L^{2}(\O))}   \leq C (T,E_0,L^{-1}),
\]
as required. Consequently the constant appearing on the right-hand side of the inequality \eqref{a-priori-tau-f5} is independent
of $\de$, and \eqref{a-priori-tau-f5} therefore provides a uniform bound on $\TT$ in $L^\infty(0,T;L^2(\Omega;\R^{2 \times 2}) \cap
L^2(0,T;W^{1,2}(\Omega;\R^{2 \times 2}))$ as $\de \rightarrow 0$, for any (fixed) $L>0$.
\end{remark}

Motivated by these formal calculations, we shall now embark on a rigorous argument aimed at proving the existence of global-in-time large data finite-energy weak solutions to the compressible Oldroyd-B model
for the case $d=2$.

\section{Weak solutions, main results and the construction of approximating solutions}\label{sec:main-result}

The rest of the paper is devoted to the proof of the existence of global-in-time large data finite-energy weak solutions to the regularized compressible Oldroyd-B model \eqref{01}--\eqref{04},
\eqref{05a}--\eqref{07a}, (\ref{ini-data})
in the case $d=2$, followed by passage to the limit $\alpha \rightarrow 0$ with the regularization parameter $\alpha>0$ under the assumption that $\de>0$. Finally, we cover the entire range of $\de \geq 0$ by passing to the limit $\de \rightarrow 0$ assuming that $L>0$.

\subsection{Weak solutions and main results}

Our main result is the proof of the existence of global-in-time large data finite-energy weak solutions to the compressible Oldroyd-B model in the two-dimensional setting. First of all, we give the definition of a finite-energy weak solution to the $\alpha$-regularized problem \eqref{01}--\eqref{04},
\eqref{05a}--\eqref{07a}.
\begin{definition}\label{def-weaksl} Let $T>0$ and let $\O\subset \R^d$ be a bounded $C^{2,\beta}$ domain, with $0<\beta<1$. Let $\ff\in L^\infty((0,T]\times \O;\R^d)$. We say that $(\vr,\vu,\eta,\TT)$ is a \textit{finite-energy
weak solution} in $(0,T]\times \O$ to the system of equations \eqref{01}--\eqref{04},
\eqref{05a}--\eqref{07a},
with $\de>0$ fixed and $\alpha > 0$, supplemented by the initial data \eqref{ini-data}, if:
\begin{itemize}

\item $\vr \geq 0 \ {\rm a.e.\  in} \ (0,T] \times \Omega,\quad  \vr \in  C_w ([0,T];  L^\gamma(\Omega)),\quad \vu\in L^{2}(0,T;W_0^{1,2}(\Omega; \R^d)),$ \quad $\TT$ {\rm is symmetric},
\ba\label{weak-est}
&\vr \vu \in C_w([0,T]; L^{\frac{2 \gamma}{ \gamma + 1}}(\Omega; \R^d)),\quad \vr |\vu|^2 \in L^\infty(0,T; L^{1}(\Omega)),\\
&\eta  \geq 0  \ {\rm a.e.\  in} \ (0,T] \times \Omega,\quad \eta \in C_w ([0,T];  L^2(\Omega)) \cap L^2 (0,T;  W^{1,2}(\Omega)),\\
&{ \TT > 0  \ {\rm a.e.\  in} \ (0,T] \times \Omega,\quad \TT \in C_w ([0,T];  L^2(\Omega;\R^{d \times d})) \cap L^2 (0,T;  W^{1,2}(\Omega;\R^{d \times d}))},\\
&{  \tr\left(\log \TT \right)\in L^\infty (0,T;  L^{1}(\Omega))\cap L^2 (0,T;  W^{1,2}(\Omega)), \quad (\eta +\a) \,\tr\left(\TT^{-1}\right) \in  L^1 (0,T;  L^{1}(\Omega)).}
\ea

\item For any $t \in (0,T]$ and any test function $\phi \in C^\infty([0,T] \times \Ov{\Omega})$, one has

\be\label{weak-form1}
\int_0^t\intO{\big[ \vr \partial_t \phi + \vr \vu \cdot \Grad \phi \big]} \,\dt' =
\intO{\vr(t, \cdot) \phi (t, \cdot) } - \intO{ \vr_{0} \phi (0, \cdot) },
\ee
\be\label{weak-form2}
\int_0^t \intO{ \big[ \eta \partial_t \phi + \eta \vu \cdot \Grad \phi - \e \nabla_x\eta \cdot \nabla_x \phi \big]} \, \dt' =  \intO{ \eta(t, \cdot) \phi (t, \cdot) } - \intO{ \eta_{0} \phi (0, \cdot) }.
\ee

\item For any $t \in (0,T]$ and any test function $\vvarphi \in C^\infty([0,T]; \DC(\Omega;\R^d))$, one has
\ba\label{weak-form3}
&\int_0^t \intO{ \big[ \vr \vu \cdot \partial_t \vvarphi + (\vr \vu \otimes \vu) : \Grad \vvarphi  + p(\vr)\, \Div_x \vvarphi + \big(kL\eta+\de\,\eta^2\big)\, \Div_x \vvarphi  - \SSS(\nabla_x \vu) : \Grad \vvarphi \big] } \, \dt'\\
&= \int_0^t \intO{ \TT : \nabla_x\vvarphi - \frac{\a}{2} \,\tr\left(\log \TT\right) \,\Div_{ x}\vvarphi - \vr\, \ff \cdot \vvarphi} \, \dt' + \intO{ \vr \vu (t, \cdot) \cdot \vvarphi (t, \cdot) } - \intO{ \vr_{0} \vu_{0} \cdot \vvarphi(0, \cdot) }.
\ea

\item For any $t \in (0,T]$ and any test function $ \YY \in C^\infty([0,T] \times \Ov{\Omega};\R^{d \times d})$, one has
\ba\label{weak-form4}
&\int_0^t \intO{ \left[ \TT : \partial_t \YY +  (\vu \,\TT) ::
\nabla_{x} \YY  + \left(\nabla_x \vu \,\TT + \TT\, \nabla_x^{\rm T} \vu \right):\YY
- \e \nabla_x  \TT :: \nabla_x \YY  \right]}\,\dt'\\
&=\int_0^t \intO{ \left[ - \frac{k\,A_0}{2\lambda}(\eta+\a) \,\tr\left(\YY\right)
+ \frac{A_0}{2\lambda} \TT:\YY \right]} \, \dt' + \intO{ \TT (t, \cdot) : \YY (t, \cdot) } - \intO{ \TT_{0} : \YY(0, \cdot) },
\ea
where the terms involving the notation $::$ are
\be\label{def-::}
(\vu \,\TT) :: \nabla_{x} \YY = \sum_{\kappa=1}^d \vu_\kappa \,\YY : \d_{x_\kappa} \YY,\quad \nabla_x  \TT :: \nabla_x \YY = \sum_{\kappa=1}^d \d_{x_\kappa}\TT : \d_{x_\kappa} \YY.
\ee

\item The continuity equation holds in the sense of renormalized solutions:
\be\label{weak-renormal}
\d_t b(\vr) +\Div_x (b(\vr)\vu) + (b'(\vr)\vr - b(\vr))\,\Div_x \vu =0 \quad \mbox{in} \ \mathcal{D}' ((0,T)\times \O),
\ee
for any $b\in C[0,\infty)\cap C^1(0,\infty)$ such that
\be\label{cond-b-renormal}
|b'(s)|<C s^{-\lambda_0}\quad \forall \,s\in (0,1] \qquad
\mbox{and} \qquad |b'(s)|<Cs^{\lambda_1}\quad \forall \,s \geq 1,
\ee
where $\lambda_0 <1$ and $\lambda_1 \in (-1,\infty)$; see {\rm(6.2.9)} and
{\rm(6.2.10)} in
\cite{N-book}.
\item  For a.e. $t \in (0,T]$ the following \emph{energy inequality} holds:
\ba\label{energy1}
 &E(t) + 2\e\int_0^t \int_\O 2 k L   \,|\nabla_x \eta^{\frac12}|^2 +  \de\,  |\nabla_x \eta|^2 \,\dx\,\dt'  + \frac{\a\,\e}{2d}  \int_0^t  \int_\O \left|\nabla_{x} \tr\left(\log \TT \right)\right|^2\,\dx \, \dt' \\
 &\qquad +  \int_0^t  \int_\O \mu^S \left| \frac{\nabla_x \vu + \nabla^{\rm T}_x \vu}{2} - \frac{1}{d} (\Div_x \vu) \II \right|^2 + \mu^B |\Div_x \vu|^2\,\dx \,\dt'  \\
 &\qquad+ \frac{A_0}{4\l} \int_0^t  \int_\O \tr\left(\TT\right) \,\dx \,\dt'
 + \frac{\a\,k\,A_0}{4\lambda} \int_0^t  \int_\O (\eta+\a)  \,\tr\left(\TT^{-1}\right)\,\dx \,\dt' \\
 &\leq  E_0 + \int_0^t \left[\int_\O \vr\,\ff \cdot  \vu \,\dx
 + \frac{k\,A_0\,d}{4\lambda} \int_\O (\eta+\a)  \,\dx + \frac{\a\,d\,A_0}{4\lambda}|\O|\right] \dt',
\ea
where $E(t)$ and $E_0$ are defined by \eqref{def-energy-free} and \eqref{def-energy-free0}.
\end{itemize}
\end{definition}

\begin{remark}\label{rem-def1} Definition \ref{def-weaksl} is fairly standard. The energy inequality \eqref{energy1} identifies an important class of weak solutions, usually termed \emph{dissipative or finite-energy} weak solutions. We note that, given a smooth solution, the energy inequality \eqref{energy1} can be derived by integrating the \textit{a priori} bound \eqref{a-priori-log-f2} over $[0,t]$ with respect to the temporal variable and using Lemma \ref{lem-lap-tau-inverse}.
\end{remark}

\begin{remark}\label{rem-def2} In Definition \ref{def-weaksl}, we assume sufficient regularity for $\TT$ in \eqref{weak-est}. This allows us to choose $\TT$ as a test function in the weak formulation \eqref{weak-form4} and to derive the following inequality in the two-dimensional setting:
\ba\label{energy2}
&\frac{1}{2}\int_\O  | \TT(t)|^2 \,\dx +  \e \int_0^t \int_\O |\nabla_x \TT|^2 \,\dx \,\dt' + \frac{A_0}{4\l} \int_0^t \int_\O | \TT|^2 \,\dx \,\dt'\\
&\quad \leq \frac{1}{2}\int_\O | \TT_0|^2 \,\dx   + 3 \int_0^t\int_\O |\nabla_x \vu|\, |\TT|^2 \,\dx\,\dt'  + \frac{k^2\,A_0}{2\l} \int_0^t \int_\O  (\eta+\a)^2 \,\dx \,\dt',
\qquad  \mbox{for a.e. $t \in (0,T]$}.
\ea
Given a symmetric positive definite matrix function $\TT$ satisfying \eqref{weak-est}, all of the terms appearing in \eqref{energy2} are meaningful. Moreover, by the argument presented in Section \ref{sec:2d-est}, we can derive the uniform bounds stated in inequality \eqref{a-priori-tau-f5}.

\end{remark}

We are now ready to state our first main theorem, which asserts the global-in-time existence of large data finite-energy weak solutions to the $\a$-regularized compressible Oldroyd-B model in the two-dimensional setting when $\de>0$.

\begin{theorem}\label{thm}
Let $d=2$ and suppose that $\g>1$, $\de>0$ and $\a>0$. Then, there exists a finite-energy weak solution $(\vr,\vu,\eta,\TT)$ to the $\a$-regularized compressible Oldroyd-B model \eqref{01a}--\eqref{07a} with initial data \eqref{ini-data}, in the sense of Definition \ref{def-weaksl}. Moreover, the extra stress tensor $\TT $ appearing as the fourth component of such a weak solution $(\vr,\vu,\eta,\TT)$  satisfies the inequality \eqref{a-priori-tau-f5}, and the constant
on the right-hand side of \eqref{a-priori-tau-f5} is independent of $\de$ as long as $L>0$.

\end{theorem}

The proof of Theorem \ref{thm} involves four levels of approximation, which are described in Section \ref{sec:4level-app}; the respective passages to the limits with the four levels of approximation are carried out in Sections \ref{sec:4th-app-full}--\ref{sec:completion-proof}.

Our second main result is stated in Theorem \ref{thm-f}, and concerns passing to the limit $\alpha \rightarrow 0$ with the regularisation parameter $\alpha>0$ in the sequence of solutions whose existence is asserted by Theorem \ref{thm}, thus proving the existence of large data finite-energy global weak solutions, in the sense of Definition \ref{def-weaksl-f}, to \eqref{01a}--\eqref{07a}, with $L\geq 0$ and $\de>0$. Finally, in Theorem \ref{thm-ff} we pass to the limit $\de \rightarrow 0$, assuming that $L>0$, to deduce the existence of large data finite-energy global weak solutions to the compressible Oldroyd-B model with stress diffusion in two space dimensions, for the entire range of parameters $\de \in [0,\infty)$, including $\de=0$. We conclude with a further result, which shows that if the initial polymer number density has stronger integrability than $L \log L(\Omega)$, say $\eta_0 \in L^q(\Omega)$, $q > 1$, then the regularity and the integrability properties of $\eta(t,\cdot)$ for $t \in (0, T ]$ are also improved.

Before embarking on the technical part of the paper, we recall, in Section \ref{sec-prelim}, a number of preliminary results, which will be required in the proofs.

\section{Preliminaries}\label{sec-prelim}

In this section we recall some technical tools that will be required in the rest of the paper.

\subsection{Classical mollifiers}\label{sec:mollifier} Let $\zeta \in C_c^\infty(\R^d)$ be a nonnegative, radially symmetric function such that
$$
{\rm supp\,}\zeta \subset B(0,1), \quad \int_{\R^d} \zeta(x)\,\dx =1.
$$
We define the mollification kernel
$$
\zeta_\th(\cdot)= \frac{1}{\th^d}\zeta\bigg(\frac{\cdot}{\th}\bigg),\quad \mbox{for any $\th>0$}.
$$
For any locally integrable function $v$ defined on $\R^d$ with values in a Banach space $X$, we define the classical (Friedrichs) {\em mollifier} $S_\th$ as the following convolution operator:
 $$
 S_\th[v] := \zeta_\th * v = \int_{\R^d} \zeta_\th (x-y)\ v(y) \, {\rm d}y.
 $$
Some of the key properties of $S_\th$ are summarized in the next lemma.
  \begin{lemma}[Theorem 10.1 in \cite{F-N-book}]\label{lem-mollifier1}
Let $X$ be a Banach space. If $v \in L^1_{\rm loc}(\R^d;X)$, we have that $S_\th [v]\in C^\infty(\R^d;X)$. In addition, the following hold:
\begin{itemize}
\item[(i)] If $v\in L^p_{\rm loc}(\R^d;X),\  1 \leq p < \infty$, then $S_\th [v] \in  L^p_{\rm loc}(\R^d;X)$ and
$$S_\th [v]\to v \ {\rm in}\  L^p_{\rm loc}(\R^d;X), \quad  {\rm as}\  \th  \to 0.$$

\item[(ii)] If $v\in L^p(\R^d;X),\  1 \leq p < \infty$, then
$$\|S_\th [v]\|_{L^p(\R^d;X)}\leq  \|v\|_{L^p(\R^d;X)};\quad S_\th [v]\to v \ {\rm in}\  L^p(\R^d;X), \quad  {\rm as}\  \th   \to 0.$$

\item[(iii)] If $v\in L^\infty(\R^d;X)$, then
$$\|S_\th [v]\|_{L^\infty(\R^d;X)}\leq  \|v\|_{L^\infty(\R^d;X)}.$$

\end{itemize}

\end{lemma}

%The following lemma is taken from lemma 10.12 and Corollary 10.3 in \cite{F-N-book}. One can also find similar results in Lemma 2.3 in \cite{Lions-InC}.

%\begin{lemma}\label{lem-mollifier2}
%Let $d \geq 2,\ 1\leq \b ,q <\infty$,  where $\frac{1}{q}+\frac 1\b =\frac 1r \leq 1$.
%\begin{itemize}
%\item[(i)] If $\vr \in  L^{\Gamma}_{\rm loc}(\R^d),\  \vu \in W^{1,q}_{\rm loc}(\R^d;\R^d)$, then
%$$\Div_x \left(S_\th \left[\vr\vu\right]\right)-\Div_x \left(S_\th [\vr]\vu\right) \to 0 \ \mbox{in}\  L^r_{\rm %loc}(\R^d) \quad  {\rm as}\  \th   \to 0.$$

%\item[(ii)] If $\vr \in  L^{\Gamma}_{\rm loc}((0,T]\times \R^d),\  \vu \in L^q_{\rm loc}((0,T); W^{1,q}_{\rm loc}(\R^d;\R^d))$, then
%$$\Div_x \left(S_\th \left[\vr\vu\right]\right)-\Div_x \left(S_\th [\vr]\vu\right) \to 0 \ \mbox{in}\  L^r_{\rm loc}((0,T]\times\R^d) \quad  {\rm as}\  \th   \to 0.$$

%\end{itemize}

%\end{lemma}

\subsection{%Korn's inequality and
The Bogovski{\u\i} operator}

We recall the Bogovski{\u\i} operator, whose construction can be found in \cite{bog} and in Chapter III of Galdi's book \cite{Galdi-book}; see also Lemma 3.17 in \cite{N-book}.

\begin{lemma}\label{lem-bog}
 Let $1<p<\infty$ and suppose that $G\subset \R^d$ is a bounded  Lipschitz domain. Let $L^{p}_0(G)$ be the space of all $L^p(G)$ functions with zero mean value. Then, there exists a linear operator ${\mathcal B}_G$ from $L_0^p(G)$ to $W_0^{1,p}(G;\R^d)$ such that for any $\rho \in L_0^p(G)$ one has
$$
\Div_x \mathcal{B}_G (\rho) =\rho \quad \mbox{in} \ G; \quad \|\mathcal{B}_G (\rho)\|_{W_0^{1,p}(G;\R^d)} \leq c (p,d,G)\, \|\rho\|_{L^p(G)}.
$$
%where the constant $c$ depends only on $p,\ d$ and the Lipschitz character of $G$.

If, in addition, $\rho=\Div_x {\bf g}$ for some ${\bf g} \in L^{q}(G;\R^d), \  1<q<\infty, \  {\bf g} \cdot {\bf n}  = 0 $ on $\partial G$, then the following inequality holds:
$$
\|\mathcal{B}_G (\rho)\|_{L^{q}(G;\R^d)} \leq c (d, q,G)\, \| {\bf g} \|_{L^q(G;\R^d)}.
$$
\end{lemma}

\medskip

\subsection{Compactness theorems} We begin by recalling the following result, usually referred to as the Aubin--Lions--Simon compactness theorem (see Simon \cite{Simon}).
\begin{lemma}\label{lem-ALS}
Let $X_0$, $X$ and $X_1$ be three Banach spaces with $X_0 \subset X \subset X_1$. Suppose that $X_0 \hookrightarrow \hookrightarrow X$, i.e. $X_0$ is compactly embedded in $X$,  and that $X \hookrightarrow X_1$, i.e. $X$ is continuously embedded in $X_1$. For $1 \leq  p, q \leq \infty$, let
$$Y = \{ v \in L^p (0, T; X_0) : \d_t v \in L^q (0, T; X_1) \}.$$
Then, the following properties hold:
\begin{itemize}
\item[(i)] If $p <\infty $, then the embedding of $Y$ into $L^p(0, T; X)$ is compact;
\item[(ii)] If $p  =\infty$ and $q  >  1$, then the embedding of $Y$ into $C([0, T]; X)$ is compact.
\end{itemize}
\end{lemma}

We shall also require the following generalization of the Aubin--Lions--Simon compactness theorem due to Dubinski\u{\i} \cite{Dubin} (see also Barrett \& S\"uli \cite{Barrett-Suli3}). Before stating the result, we recall the
concept of \emph{seminormed set} (in the sense of Dubinski\u{\i}).
A subset $X_0$ of a linear space $X$ over the field of real numbers is said to be a \emph{seminormed set} if
$$
\l \, v \in X_0, \quad \mbox{for any $ \l\in [0,\infty)$ and any $v\in X_0$},
$$
and there exists a functional on $X_0$ (namely the \emph{seminorm} of $X_0$), denoted by $[\cdot ]_{X_0}$, satisfying the following two properties:
\begin{itemize}
\item[(i)] $[v]_{X_0}\geq 0$ for any $v\in X_0$, and $[v]_{X_0}= 0$ if and only if $v=0$;
\item[(ii)] $[\l\,v]_{X_0} = \l\, [v]_{X_0}$ for any $\l \in [0,\infty)$ and any $v\in X_0$.
\end{itemize}

A subset $B$ of a seminormed set $X_0$ is said to be bounded if there exists a positive constant $c>0$ such that $[v]_{X_0} \leq c$ for any $v\in B$. A seminormed set $X_0$ contained in a normed linear space $X$ with norm $\|\cdot \|_{X}$  is said to be \emph{continuously embedded} in $X$, and we write
$X_0 \hookrightarrow X$, if there exists a constant $c>0$ such that
\[\|v\|_X \leq c \,[v]_{X_0}, \quad \mbox{for any $v\in X_0$}.\]
The embedding of a seminormed set $X_0$ into a normed linear space $X$ is said to be \emph{compact} if from any bounded infinite set of elements of $X_0$ one can extract a subsequence that converges in $X$.

We remark here that, for the sake of simplicity of the exposition and our mathematical notations, the extraction of subsequences from sequences (e.g. the extraction of weakly or weakly-* convergent subsequences from bounded sequences, or the extraction of almost everywhere convergent subsequences from strongly convergent sequences) will not be explicitly indicated.

\begin{lemma}[Dubinski\u{\i}'s compactness theorem]\label{lem-Dubinskii}
Suppose that $X_0$ is a seminormed set that is compactly embedded into a Banach space $X$, which is continuously embedded into another Banach space $X_1$. Then, for any $1 \leq  p, q < \infty$, the embedding
$$
%\be\label{Dubin}
\{ v \in L^p (0, T; X_0) : \d_t v \in L^q (0, T; X_1) \} \hookrightarrow L^p (0, T; X)
%\ee
$$
is compact.
\end{lemma}

\subsection{On $C_w([0, T];X)$ type spaces}

Let $X$ be a Banach space. We denote by $C_w([0, T];X)$ the set of all functions $v\in L^\infty(0, T;X)$ such that the mapping
$t \in  [0, T] \mapsto \langle \phi,v(t)\rangle_{X}\in \R$ is continuous on $[0, T]$ for all $\phi \in X'$. Here and throughout the paper, we use $X'$ to denote the
dual space of $X$, and $\langle \cdot,\cdot \rangle_{X}$ to denote the duality pairing between $X'$ and $X$.

Whenever $X$ has a predual $E$, in the sense that $E'=X$, we denote by $C_{w*}([0, T];X)$ the set of all functions $v\in L^\infty(0, T;X)$ such that the mapping $t \in  [0, T] \mapsto
\langle v(t),\phi\rangle_{E}\in \R$ is continuous on $[0, T]$ for all $\phi \in E$. We reproduce Lemma 3.1 from \cite{Barrett-Suli}.
\begin{lemma}\label{lem-Cw-Cw*}
Suppose that $X$ and $Y$ are Banach spaces.
\begin{itemize}
\item[(i)] Assume that the space $X$ is reflexive and is continuously embedded in the space $Y$; then,
$$ L^\infty(0, T;X) \cap C_w([0, T];Y) = C_w([0, T];X).$$
\item[(ii)] Assume that $X$ has a separable predual $E$ and $Y$ has a predual $F$ such that $F$ is continuously
embedded in $E$; then,
   $$ L^\infty(0, T;X) \cap C_{w*}([0, T];Y) = C_{w*}([0, T];X).$$
\end{itemize}
\end{lemma}

Part (i) is due to Strauss \cite{Strauss} (cf. Lions \& Magenes \cite{Lions-Magenes}, Lemma 8.1, Ch. 3, Sec. 8.4); part
(ii) is proved analogously, via the sequential Banach--Alaoglu theorem.

We recall the following Arzel\`{a}--Ascoli type result, and refer to Lemma 6.2 in \cite{N-book} for its proof.
\begin{lemma}\label{lem-Cw} Let $r, s \in (1,\infty)$ and let $G$ be a bounded Lipschitz domain in $\R^d, \ d\geq 2$. Suppose that $(g_n)_{n\in \N}$ is a sequence of functions in $C_w([0, T]; L^s(G))$ such that $(g_n)_{n\in \N}$ is bounded in $C([0, T]; W^{-1,r}(G)) \cap L^\infty(0, T;L^s(G))$. Then, there exists a subsequence (not indicated) such that the following hold:
\begin{itemize}
\item[(i)] $g_n\to g$ in $C_w([0, T]; L^s(G))$;
\item[(ii)] If, in addition, $r\leq \frac{d}{d-1}$, or $r>\frac{d}{d-1}$ and $s>\frac{d\,r}{d+r}$, then $g_n\to g$ strongly in $C([0, T]; W^{-1,r}(G))$.
\end{itemize}
\end{lemma}

\subsection{Regularity of the parabolic Neumann problem}
We first introduce fractional-order Sobolev spaces. Let $G$ be the whole space $\R^d$ or a bounded Lipschitz domain in $\R^d$. For any $k\in \N$, $\b \in (0,1)$ and $s\in [1,\infty) $, we define
\[
%\be\label{def-frac-Sob}
 W^{k+\b,s}(G):=\left\{v\in W^{k,s}(G) : \| v \|_{W^{k+\b,s}(G)}<\infty \right\}\!,
%\ee
\]
where
$$
\| v \|_{W^{k+\b,s}(G)}:=\| v \|_{W^{k,s}(G)} + \sum_{|\a|=k} \left(\int_G\int_G
\frac{|\d^\a v(x)- \d^\a v(y)|^s}{|x-y|^{d+\b s}} \ \dx \ {\rm d}y\right)^{\frac{1}{s}}.
$$
%We recall the following classical compact embedding theorem (see Theorem 7.1 in \cite{NPV}).
%\begin{lemma}\label{lem-frac-sob} Let $\O\subset \R^d$ be a bounded Lipschitz domain and suppose that $k\in \N$, \ $\b \in (0,1)$ and $s\in [1,\infty)$; then, the embedding
%of $W^{k+\b,s}(\O)$ into $W^{k,s}(\O)$ is compact.
%\end{lemma}

The following classical results are taken from Section 7.6.1 in \cite{N-book}. Let $G$ be a bounded domain in $\R^d$ and consider the parabolic initial-boundary-value problem:
\ba\label{para-ib-Neumann}
\d_t \rho - \e\, \Delta_x \rho &= h \quad &&\mbox{in} \ (0,T]\times G,\\
\rho(0,\cdot) &= \rho_0  \quad && \mbox{in} \   G,\\
\d_{\bf n}\rho &=0 \quad && \mbox{in} \ (0,T]\times \d G.
\ea
Here $\e>0$, $\rho_0$ and $h$ are known functions, and $\rho$ is the unknown solution. The first regularity result of relevance to us here is encapsulated in the following lemma.

\begin{lemma}\label{lem-parabolic-1}
Let $0<\beta < 1, \ 1<p,q<\infty$ and suppose that $G$ is a bounded domain in $\R^d$,
$$
G \in C^{2,\beta},\quad \rho_0 \in W^{2-\frac{2}{p},q}_{\bf n},\quad h\in L^p(0,T;L^q(G)),
$$
where $W^{2-\frac{2}{p},q}_{\bf n}$ is the completion of the linear space $\{v \in  C^\infty(\overline G): \ \d_{\bf n} v|_{G} = 0 \}$
 in the norm of $W^{2-\frac{2}{p},q}(G)$.
Then, there exists a unique function $\rho$ satisfying
$$
\rho \in L^p(0,T;W^{2,q}(G)) \cap C([0,T];W^{2-\frac{2}{p},q}(G)), \quad \d_t \rho \in L^p(0,T;L^q(G))
$$
and solving $\eqref{para-ib-Neumann}_1$ a.e. in $(0,T]\times G$,  $\eqref{para-ib-Neumann}_2$ a.e. in $G$; in addition, $\rho$ satisfies $\eqref{para-ib-Neumann}_3$ in the sense of the normal trace, which is well defined since $\Delta_x \rho \in L^p(0,T;L^{q}(G))$.
Moreover, we have that
\begin{align*}
%&
\e^{1-\frac{1}{p}}\|\rho\|_{L^\infty(0,T;W^{2-\frac{2}{p},q}(G))}
+ \left\|\partial_t\rho \right\|_{L^p(0,T; L^q(G))}
+ \e \|\rho\|_{L^p(0,T; W^{2,q}(G))}\nonumber\\
%&\qquad\qquad
\leq C(p,q,G)\big[\e^{1-\frac{1}{p}}
\|\rho_0\|_{W^{2-\frac{2}{p},q}(G)} + \|h\|_{L^p(0,T;L^q(G))}\big].
\end{align*}

\end{lemma}

The second result that we state concerns parabolic problems with a divergence-form source term,
$h=\Div_x \bf g$.
\begin{lemma}\label{lem-parabolic-2}
Let $0<\beta < 1, \ 1<p,q<\infty$ and suppose that $G$ is a bounded domain in $\R^d$,
$$
G \in C^{2,\beta},\quad \rho_0 \in L^{q}(G),\quad {\bf g} \in L^p(0,T;L^q(G;\R^d)).
$$
Then, there exists a unique function $ \rho\in L^p(0,T;W^{1,q}(G)) \cap C([0,T];L^{q}(G)) $ satisfying $\eqref{para-ib-Neumann}_2$ a.e. in $G$ and
$$
\frac{\rm d}{\dt} \int_G \rho\, \phi\,\dx + \e \int_G \nabla_x\rho \cdot \nabla_x \phi \,\dx = - \int_G  {\bf g} \cdot \nabla_x \phi \,\dx \quad \mbox{in}\ {\mathcal D}'(0,T).
$$
Moreover, we have that
\begin{equation*} \e^{1-\frac{1}{p}}\|\rho\|_{L^\infty(0,T;L^q(G))}
+ \e \|\nabla_x \rho\|_{L^p(0,T;L^q(G;\R^d))} \leq C(p,q,G)
\left[\e^{1-\frac{1}{p}}\|\rho_0\|_{L^q(G)}
+ \|{\bf g}\|_{L^p(0,T;L^q(G;\R^d))}\right].
\end{equation*}
\end{lemma}

\section{Definition of the sequence of approximating solutions}\label{sec:4level-app}

We will prove Theorem \ref{thm} by means of a four-level approximation, inspired by the construction of approximate solutions in \cite{FNP, Feireisl-2001, F-N-book} for the study of the compressible Navier--Stokes equations and in \cite{Barrett-Boyaval} for the study of the incompressible Oldroyd-B model. In this section we will describe our four-level approximation scheme. In subsequent sections we will prove the existence of solutions to each of the approximation levels,
the convergence of the approximating solution sequence at each level, and will complete the proof of Theorem \ref{thm}. Finally, upon passing to the limits $\alpha \rightarrow 0$ and $\de\to 0$, we will deduce the existence of a global-in-time large data finite-energy weak solution to the original compressible Oldroyd-B model,
(\ref{01a})--(\ref{07a}),
for the entire range of $\de \in [0,\infty)$, including $\de=0$.

In the sequel, we shall occasionally retain the symbol $d$ in certain (in)equalities in order to emphasize the role of the number of space dimensions in the (in)equality concerned, but it will be understood throughout that the analysis that follows is restricted to the case of $d=2$.

\subsection{Mollification of the initial data} First of all, we consider a mollification of the initial data by using the mollifier introduced in Section \ref{sec:mollifier}.

Let $d=2$, let $\O \subset \R^2$ be a bounded $C^{2,\beta}$ domain, with $\beta \in (0,1)$, and let the initial data $\vr_0,\vu_0, \eta_0,\TT_0$ be given, as in \eqref{ini-data}. We consider the zero-extension of $(\vr_0,\vu_0, \eta_0,\TT_0)$ to the whole of $\mathbb{R}^2$, still denoted by the same symbols, outside of the domain $\O$. We then define for $\th>0$ the following mollified initial data:
 \ba\label{ini-data-mollified}
 \vr_{0,\th} = \th + S_\th[\vr_0]; \quad \vu_{0,\th} =  S_\th[\vu_0]; \quad \eta_{0,\th} := \th + S_\th[\eta_0];\quad \TT_{0,\th} := \th \,\II  + S_\th[\TT_0].
 \ea
Thanks to the properties of the classical Friedrichs mollifier listed in Lemma \ref{lem-mollifier1}, we have the following bounds and convergence results, as $\th \rightarrow 0$:
 \ba\label{ini-data-mollified-pt}
&\vr_{0,\th}\in C^\infty(\R^2), \  \th \leq  \vr_{0,\th} \leq C(\th),\ \vr_{0,\th} \to \vr_0 \ \mbox{in $L^\g(\O)$}; \\
& \vu_{0,\th}\in C^\infty(\R^2;\R^2), \  \vu_{0,\th} \to \vu_0 \
\mbox{in $L^{r}(\O;\R^2)$ for $r \in [1,\infty)$} , \ \vr_{0,\th}|\vu_{0,\th}|^2 \to \vr_{0}|\vu_{0}|^2 \ \mbox{in $L^1(\O)$};\\
& \vr_{0,\th}\vu_{0,\th}\in C^\infty(\R^2;\R^2), \   \vr_{0,\th}\vu_{0,\th} \to  \vr_{0}\vu_{0} \ \mbox{in $L^{\frac{2\g}{\g+1}}(\O;\R^2)$}; \\
&\eta_{0,\th}\in C^\infty(\R^2), \  \th \leq  \eta_{0,\th} \leq C(\th),\ \eta_{0,\th} \to \eta_0 \ \mbox{in $L^2(\O)$}; \\
&\TT_{0,\th}\in C^\infty(\R^2;\R^{2\times 2}), \  \th \leq  \TT_{0,\th}=\TT_{0,\th}^{\rm T} \leq C(\th),\ \TT_{0,\th} \to \TT_0 \ \mbox{in $L^2(\O;\R^{2\times 2})$},
\ea
where $C(\th)$ signifies a constant depending only on $\th$. By Sobolev embedding,
\[
%\ba\label{C-th}
\|\vr_{0,\th}\|_{L^\infty(\R^2)} \leq \th +  \|S_\th[\vr]\|_{L^\infty(\R^2)}
\leq \th + C\|S_\th[\vr]\|_{W^{2,\g}(\R^2)} \leq \th+ C \th^{-2},
%\ea
\]
and we can therefore take $C(\th) \approx \th^{-2}$ as $\th \rightarrow 0$.

\medskip

\subsection{First level: artificial pressure approximation}\label{sec:1st-app}

Let $\s_1>0$ be small and $\Gamma\geq 4$. We consider the following system of equations, which results from a modification of the pressure in the system \eqref{01}--\eqref{04}:
\begin{align*}
%\label{01-l1}
\d_t \vr + \Div_x (\vr \vu) &= 0,\\
%\label{02-l1}
\d_t (\vr\vu)+ \Div_x (\vr \vu \otimes \vu) + \nabla_x p(\vr) + \boxed{\s_1 \nabla_x \vr^\Gamma} &+ \nabla_x\big(k L \eta+\de\,\eta^2\big) -
\Div_x \SSS (\nabla_x \vu) \nonumber\\
&=\Div_x \TT   -  \frac{\a}{2}\,\nabla_x  \tr\left(\log \TT \right) + \vr\, \ff,\\
%\label{03-l1}
\d_t \eta + \Div_x (\eta \vu) &= \e \Delta_x \eta,\\
%\label{04-l1}
\d_t \TT + {\rm Div}_x (\vu\,\TT) - \left(\nabla_x \vu \,\TT + \TT\, \nabla_x^{\rm T} \vu \right) &= \e \Delta_x \TT + \frac{k\,A_0}{2\lambda}(\eta+\alpha)  \,\II - \frac{A_0}{2\lambda} \TT.
\end{align*}
We impose the same boundary conditions as in \eqref{05a}--\eqref{07a} and we consider the mollified initial data defined in \eqref{ini-data-mollified}, satisfying \eqref{ini-data-mollified-pt}.

\subsection{Second level: dissipative approximation}\label{sec:2nd-app}

Let $\s_2>0$ be small. We consider the following system of equations, where a dissipative term is added to the continuity equation and, in order to maintain an energy bound,
a term is added to the momentum equation:
\begin{align}
\label{01-l2}
\d_t \vr + \Div_x (\vr \vu) &= \boxed{\s_2 \Delta_x \vr}\,,\\
%\label{02-l2}
\nonumber
 \d_t (\vr\vu)+ \Div_x (\vr \vu \otimes \vu) +\nabla_x p(\vr) + \boxed{\s_1 \nabla_x \vr^\Gamma} &+ \boxed{\s_2 \nabla_x \vu \nabla_x \vr} + \nabla_x\big(k L \eta+\de\,\eta^2\big) -
\Div_x \SSS (\nabla_x \vu) \nonumber \\
\nonumber
&=\Div_x \TT - \frac{\a}{2}\,\nabla_x  \tr\left(\log \TT \right)  +\vr\, \ff,
\\
%\label{03-l2}
\nonumber
\d_t \eta + \Div_x (\eta \vu) &= \e \Delta_x \eta,\\
%\label{04-l2}
\nonumber
\hspace{-7mm}\d_t \TT + {\rm Div}_x (\vu\,\TT) - \left(\nabla_x \vu \,\TT + \TT\, \nabla_x^{\rm T} \vu \right) &= \e \Delta_x \TT + \frac{k\,A_0}{2\lambda}(\eta+\alpha)  \,\II - \frac{A_0}{2\lambda} \TT.
\end{align}
We consider the mollified initial data defined in \eqref{ini-data-mollified}, satisfying \eqref{ini-data-mollified-pt}.
Since the $\sigma_2$-regularized equation \eqref{01-l2} is now parabolic, in addition to the boundary conditions stated in \eqref{05a}--\eqref{07a} we shall also require that
\be\label{boundary-vr}
\d_{\bf n} \vr =0 \quad \mbox{on}\  (0,T]\times \d\O.
\ee

\subsection{Third level: Galerkin approximation}\label{sec:3rd-app}

By the classical theory of eigenvalue problems for symmetric linear elliptic operators (see, for example, Theorem 1 in Section 6.5 in \cite{Evans-book}), one deduces the existence of an infinite sequence of eigenvalues  $0<\l_1 \leq \l_2 \leq \cdots $ with $\l_n \to \infty,\ n\to \infty$, and an associated orthogonal eigenfunction basis in $L^2(\O;\R^2)$, denoted by $(\pmb{\psi}_n)_{n \in \mathbb{N}}$, such that
\[
%  \be\label{eq-sp-lap}
 -\Delta_x \pmb{\psi}_n = \l_n \pmb{\psi}_n \ \mbox{in $\O$};\quad \pmb{\psi}_n=\pmb{0}\ \mbox{on $\d\O$}.
% \ee
\]
Moreover, $\pmb{\psi}_n \in W^{1,2}_0(\O;\R^2) \cap W^{2,2}(\O;\R^2) \cap C^\infty(\O;\R^2)$ and $\pmb{\psi}_n\in C^{2,\beta} (\overline\O;\R^2)$ since $\O$ is a $C^{2,\beta}$ domain, with $0<\beta<1$; by a classical Schauder type elliptic regularity estimate and Sobolev embedding, one also has that
\be\label{est-psi-n}
\|\pmb{\psi}_n\|_{C^{2,\beta}(\overline{\O};\R^2)} \leq C(\l_n) \,\|\pmb{\psi}_n\|_{L^2(\O;\R^2)},\qquad \mbox{with $C(\l_n) \leq C\l_n^{2}$,\; for $n = 1,2,\dots$}.
%\|\pmb{\psi}_n\|_{W^{N,2}(\O;\R^2)} \leq C(\l_n,N) \,\|\pmb{\psi}_n\|_{L^2(\O;\R^2)},\quad  \|\pmb{\psi}_n\|_{C^{N }(\overline %\O;\R^2)} \leq C(\l_n,N) \,\|\pmb{\psi}_n\|_{L^2(\O;\R^2)}.
\ee

We define the $n$-dimensional Hilbert space $X_n$, with inner product $\langle \cdot,\cdot \rangle$, by
\[
%\be\label{def-Xn}
X_n:= {\rm span\,}\{\pmb{\psi}_1, \dots,\pmb{\psi}_n\}, \quad \langle \vv, \ww\rangle =
\int_\O \vv\cdot \ww \,\dx, \quad  \vv,\ww\in X_n.
%\ee
\]
We denote by $P_n$ the orthogonal projection in $L^2(\O;\R^2)$ onto the linear subspace $X_n$, and we consider the following problem:

\ba\label{02-l3}
&\vu_n \in C([0,T],X_n), \ \vu_n(0) = \vu_{0,n} = P_n \vu_{0,\th}; \quad  \mbox{for any $\vvarphi\in X_n$:}\\
&\int_\O \d_t (\vr_n\vu_n) \cdot \vvarphi \,\dx \\
&\quad + \int_\O \left[ \Div_x (\vr_n \vu_n \otimes \vu_n) +\nabla_x p(\vr_n) + \boxed{\s_1 \nabla_x \vr_n^\Gamma} + \boxed{\s_2 \nabla_x \vu_n \nabla_x \vr_n}+ \nabla_x\big(k L \eta_n + \de\,\eta_n^2\big) - \Div_x \SSS(\nabla_x \vu_n)\right]\cdot \vvarphi \,\dx \\
&= \int_\O \left[ \Div_x \TT_n  -  \frac{\a}{2}\,\nabla_x  \tr\left(\log \TT_n \right)  +  \vr_n\, \ff \right]\cdot \vvarphi\,\dx,
\ea
where $\vu_{0,\th}$ is the mollified initial datum for $\vu_0$ defined in \eqref{ini-data-mollified}, and $\vr_n, \eta_n,\TT_n$ are determined by the parabolic equations
\begin{align}
\label{01-l3}
\d_t \vr_n + \Div_x(\vr_n \vu_n) &= \boxed{\s_2 \Delta_x \vr_n}\,,\\
\label{03-l3}
\d_t \eta_n + \Div_x (\eta_n \vu_n) &= \e \Delta_x \eta_n,\\
\label{04-l3}
\d_t \TT_n + {\rm Div}_x (\vu_n\,\TT_n) - \left(\nabla_x \vu_n \,\TT_n + \TT_n \, \nabla_x^{\rm T} \vu_n \right) &= \e \Delta_x \TT_n + \frac{k\,A_0}{2\lambda}(\eta_n+\a)  \,\II - \frac{A_0}{2\lambda} \TT_n,
\end{align}
subject to the boundary conditions stated in \eqref{05a}--\eqref{07a} and \eqref{boundary-vr}, and the mollified initial data defined in \eqref{ini-data-mollified}, satisfying \eqref{ini-data-mollified-pt}, for $\vr_n$, $\vu_n$, $\eta_n$ and $\TT_n$.

\subsection{Fourth level: regularization of the extra stress tensor}\label{sec:5th-app}

As pointed out in Section \ref{sec:est-log}, the \textit{a priori} bounds are obtained by assuming the
that $\TT$ is symmetric positive definite, which we do not have \textit{a priori}. Thus, inspired by the work of Barrett \& Boyaval \cite{Barrett-Boyaval},
we will employ
a regularization for $\TT$ to construct
a family of symmetric positive definite approximations of $\TT$,
which satisfy bounds on their logarithm and inverse similar to the ones
in Section \ref{sec:est-log}. The regularization of the extra stress tensor $\TT$
in \cite{Barrett-Boyaval} needs to be modified slightly
%since the divergence-free condition assumed there is no longer
to remain valid in our context.

\medskip

Let $\s_3>0$ be small, in the sense that $\s_3 < \min\{\a,\th\}$, and define $\chi_{\s_3} (s): = \max \{\s_3,s\}$, $s\in \R$. We introduce the following generalization of scalar functions to symmetric matrix functions: let $g :\R\to \R$ be a scalar function and let $\PP\in \R^{d\times d}$ be a real symmetric matrix; then, one has the following diagonalization:
$$
\PP =  \OO \DD \OO^{\rm T},\quad \mbox{$\OO$ is an orthogonal matrix, $\DD={\rm diag}\,\{\l_1,\ldots, \l_d\}$,}
$$
where $\l_j,\ j=1,\ldots, d$, are the eigenvalues of $\PP$.

We define $g(\PP)$ and $g'(\PP)$ by the following formulae
\be\label{def-g-tau}
g(\PP) = \OO (g(\DD)) \OO^{\rm T},\quad g'(\PP) = \OO (g'(\DD)) \OO^{\rm T},
\ee
where
\be\label{def-g-tau-2}
g(\DD):={\rm diag}\,\{g(\l_1),\ldots, g(\l_d)\},\quad g'(\DD):={\rm diag}\,\{g'(\l_1),\ldots, g'(\l_d)\}.
\ee

With these definitions, we have the following lemma, whose proof is given in Appendix \ref{appendixb}.
\begin{lemma}\label{lem-pt-mf}
Let $g\in C^{1,\gamma}(\R)$, with $0 < \gamma \leq 1$, be concave or convex, and let $\PP\in W^{1,2}(0,T;\R^{d\times d})$
be symmetric. Then, the matrix function $t \in (0,T) \mapsto g(\PP(t)) \in \mathbb{R}^{d \times d}$, defined by \eqref{def-g-tau}, is differentiable a.e. on $(0,T)$ and satisfies the identity
\begin{align*}%\label{pt-mf}
\d_t \tr\left(g(\PP)\right) = \tr\left(g'(\PP) \,\d_t \PP \right) = \d_t \PP : g'(\PP)\qquad \mbox{a.e. on $(0,T)$}.
\end{align*}
\end{lemma}

\medskip

We now fix $d=2$, and state the fourth level of approximation as follows:
\ba\label{02-l5}
&\vu_{n,\s_3} \in C([0,T_n],X_n), \ \vu_{n,\s_3}(0) = \vu_{0,n} = P_n \vu_{0,\th}; \quad
\mbox{for any $\vvarphi\in X_n$:}
\\
&\int_\O \left[\d_t(\vr_{n,\s_3}\vu_{n,\s_3})
+ \Div_x (\vr_{n,\s_3} \vu_{n,\s_3} \otimes \vu_{n,\s_3})\right]
\cdot \vvarphi \,\dx \\
&  \quad + \int_\O \left[ \nabla_x p(\vr_{n,\s_3}) + \boxed{\s_1 \nabla_x\vr_{n,\s_3}^\Gamma} + \boxed{\s_2 \nabla_x \vu_{n,\s_3} \nabla_x \vr_{n,\s_3}} +\nabla_x\big(k L \eta_{n,\s_3} + \de\,\eta_{n,\s_3}^2 \big) - \Div_x \SSS(\nabla_x \vu_{n,\s_3})\right]\cdot \vvarphi \,\dx \\
&= \int_\O \left[\,\boxed{\Div_x \chi_{\s_3}(\TT_{n,\s_3}^S)}  -
\boxed{\frac{\a}{2}\,\nabla_x  \tr\left(\log \chi_{\s_3}(\TT_{n,\s_3}^S) \right)}  + \vr_{n,\s_3}\, \ff \right]\cdot \vvarphi\,\dx,
\ea

where $\vu_{0,\th}$ is the mollified initial datum for $\vu_0$ defined
in \eqref{ini-data-mollified},
\begin{align}
\TT_{n,\s_3}^S := \tfrac{1}{2} \left(\TT_{n,\s_3} + (\TT_{n,\s_3})^{\rm T}\right),
\label{TTS}
\end{align}
and $\vr_{n,\s_3}, \eta_{n,\s_3},\TT_{n,\s_3}$ are determined by the parabolic equations
\begin{align}
\label{01-l5}
\d_t \vr_{n,\s_3} + \Div_x(\vr_{n,\s_3} \vu_{n,\s_3}) &= \boxed{\s_2 \Delta_x \vr_{n,\s_3}}\,,\\
\label{03-l5}
\d_t \eta_{n,\s_3} + \Div_x (\eta_{n,\s_3} \vu_{n,\s_3}) &= \e \Delta_x \eta_{n,\s_3},
\end{align}
\ba
\label{04-l5}
&\d_t \TT_{n,\s_3} + {\rm Div}_x (\vu_{n,\s_3}\,
\boxed{\chi_{\s_3}(\TT_{n,\s_3}^S)} ) -
\left(\nabla_x \vu_{n,\s_3} \,\boxed{\chi_{\s_3}(\TT_{n,\s_3}^S)}
+ \boxed{ \chi_{\s_3}(\TT_{n,\s_3}^S)} \, \nabla_x^{\rm T} \vu_{n,\s_3} \right)
\\
&\hspace{1in}= \e \Delta_x \TT_{n,\s_3} + \frac{k\,A_0}{2\lambda}
(\eta_{n,\s_3}+\a)  \,\II
- \frac{A_0}{2\lambda} \boxed{ \chi_{\s_3}(\TT_{n,\s_3}^S)}\,,
\ea
and the boundary conditions in \eqref{05a}--\eqref{07a} and \eqref{boundary-vr}.
The equations \eqref{01-l5}--\eqref{04-l5} will be
 considered subject to the initial data defined in \eqref{ini-data-mollified}, satisfying \eqref{ini-data-mollified-pt} for $\vr_{n,\s_3}, \eta_{n,\s_3},\TT_{n,\s_3}$.

Compared to the regularization of the extra tensor performed in \cite{Barrett-Boyaval} in the incompressible case,  in the compressible case considered here we require the additional
regularization
$$
%{\rm Div}_x (\vu_{n,\s_3}\, \chi_{\s_3}(\TT_{n,\s_3})) \quad\mbox{ and }\quad
\frac{A_0}{2\lambda} \chi_{\s_3}(\TT_{n,\s_3}^S)
$$
featuring in \eqref{04-l5} in order derive sufficiently strong bounds on
$\log(\TT_{n,\s_3})$ and $\TT_{n,\s_3}^{-1}$ (see Section \ref{logarithmic-bb} below).

\section{The fourth level of approximation}\label{sec:4th-app-full}

For any $\s_3>0$, sufficiently small, and any $n\in \N$, the problem \eqref{02-l5} is a system of ordinary differential equations in $\vu_{n,\s_3}$ with respect to $t$ because $X_n$ is a finite-dimensional space; the equations \eqref{01-l5}--\eqref{04-l5} are all of parabolic type and are all well-posed given any smooth $\vu_{n,\s_3}$. Thus, locally in time,  over a time interval $[0,T_{n,\s_3}]$, for some $T_{n,\s_3}>0$, the existence of a unique solution, denoted by $(\vr_{n,\s_3},\vu_{n,\s_3},\eta_{n,\s_3},\TT_{n,\s_3})$, to the problem at the fourth level of approximation, posed in Section \ref{sec:5th-app}, is classical, see \cite{Lions-C, FNP,F-book,N-book}.

Since $\vu_{n,\s_3} \in C([0,T_{n,\s_3}],X_n)$, by the definition of $X_n$ in Section \ref{sec:3rd-app}, we have
\be\label{est-5th-app-u}
\vu_{n,\s_3} \in C([0,T_{n,\s_3}],C^{2,\beta}(\overline \O;\R^2)),\quad \|\vu_{n,\s_3}(t)\|_{C^{2,\beta}(\overline \O;\R^2)} \leq C(n)\|\vu_{n,\s_3}(t)\|_{L^{2}( \O ;\R^2)} \ \mbox{ for all $t\in [0,T_{n,\s_3}]$}.
\ee

By similar arguments as in Section 2.1 in \cite{FNP} concerning well-posedness and uniform bounds for parabolic equations, we have, for all $t\in (0,T_{n,\s_3}]$, that
\ba\label{est-5th-app-vr-eta}
&(\vr_{n,\s_3},\  \eta_{n,\s_3}, \ \TT_{n,\s_3}) \in C([0,T_{n,\s_3}];W^{1,2}(\O) \times W^{1,2}(\O) \times W^{1,2}(\O;\R^{2 \times 2})),\\
&(\vr_{n,\s_3},\  \eta_{n,\s_3}, \ \TT_{n,\s_3}) \in
L^2(0,T_{n,\s_3};W^{2,2}(\O) \times W^{2,2}(\O) \times W^{2,2}(\O;\R^{2 \times 2})),\quad \TT_{n,\s_3} \mbox{ is symmetric},\\
& \th \exp\left(-\int_0^t \|\Div_{ x} \vu_{n,\s_3}(t')\|_{L^\infty(\O)}\,\dt'\right) \leq \vr_{n,\s_3}(t,x) \leq C(\th) \exp\left(\int_0^t \|\Div_{ x} \vu_{n,\s_3}(t')\|_{L^\infty(\O)}\,\dt'\right),\\
& \th \exp\left(-\int_0^t \|\Div_{ x} \vu_{n,\s_3}(t')\|_{L^\infty(\O)}\,\dt'\right) \leq \eta_{n,\s_3}(t,x) \leq C(\th) \exp\left(\int_0^t \|\Div_{ x} \vu_{n,\s_3}(t')\|_{L^\infty(\O)}\,\dt'\right),\\
& \|\vr_{n,\s_3}(t)\|_{W^{1,2}(\O)}^2 + \int_0^t \|\vr_{n,\s_3}(t')\|_{W^{2,2}(\O)}^2\,\dt'
\leq C\left(t,\th,\s_2,\|\nabla_x \vu_{n,\s_3}\|_{L^\infty((0,T_{n,\s_3})\times \O; \R^{2\times 2})}\right),\\
& \|\eta_{n,\s_3}(t)\|_{W^{1,2}(\O)}^2 + \int_0^t \|\eta_{n,\s_3}(t')\|_{W^{2,2}(\O)}^2\,\dt' \leq C\left(t,\th,\|\nabla_x \vu_{n,\s_3}\|_{L^\infty((0,T_{n,\s_3})\times \O;\R^{2\times 2})}\right),\\
& \|\TT_{n,\s_3}(t)\|_{W^{1,2}(\O;\R^{2\times 2})}^2 + \int_0^t \|\TT_{n,\s_3}(t')\|_{W^{2,2}(\O;\R^{2\times 2})}^2\,\dt' \leq C\left(t,\th,\|\nabla_x \vu_{n,\s_3}\|_{L^\infty((0,T_{n,\s_3})\times \O;\R^{2\times 2})}\right).
\ea
The symmetry of $\TT_{n,\s_3}$ can be deduced by using the symmetry of equation
\eqref{04-l5}, the symmetry of $\TT_{n,\s_3}^S$,
the symmetry of the initial datum $\TT_{n,\s_3}(0) = \TT_{0,\th}\geq \th$,
the symmetry of the trace operator appearing on the right-hand side of (\ref{02-l5}),
and the uniqueness of the solution to equation \eqref{04-l5}. The latter is a consequence of \eqref{est-5th-app-u} and the Lipschitz continuity of $\chi_{\sigma_3}$ defined over the space of real symmetric matrices, which follows from the Lipschitz continuity of $\chi_{\sigma_3}$ considered as a mapping from $\mathbb{R}$ into $\mathbb{R}$ (cf. Theorem 1.1 in \cite{Wihler}).

The bound on $\TT_{n,\s_3}$ in \eqref{est-5th-app-vr-eta} can be derived similarly to those on the scalar functions  $\vr_{n,\s_3}$ and $\eta_{n,\s_3}$, by observing that, for any real symmetric matrix $\PP\in \R^{d\times d}$, one has
\[
%\be\label{est-level5-rk}
|\chi_{ \s_3} (\PP)| \leq \s_3 + |\PP|.
%\ee
\]
In the rest of this section we shall derive uniform bounds on the solution sequence, which guarantee that the existence time $T_{n.\sigma_3}$ identified above can be extended to $T$.

\subsection{Uniform bounds}

We shall now develop some bounds that are uniform in $\s_3$ in the limit of $\s_3\to 0$.
%We shall present the calculations for $d \in \{2,3\}$, although, ultimately, the results in this section will be used with $d=2$ %subsequently.
Similarly to the \textit{a priori} bound \eqref{a-priori-vu},
we deduce by taking $\vvarphi=\vu_{n,\s_3}$ in \eqref{02-l5}, noting (\ref{01-l5}) and that
$$\frac{1}{2}\int_{\Omega} \Delta_x \vr_{n,\s_3}\,|\vu_{n,\s_3}|^2 \,\dx
= - \int_{\Omega} (\nabla_x
\vu_{n,\s_3}
\nabla_x \vr_{n,\s_3}) \cdot \vu_{n,\s_3}\, \dx, $$
and combining with
(\ref{01-l5}) tested with $b_1'(\vr_{n,\s_3})$
and (\ref{03-l5}) tested with $b_2'(\eta_{n,\s_3})$, where
$b_1(r)=\frac{a}{\g-1}r^\g+
\frac{\s_1}{{\Gamma}-1}r^\Gamma$ and
$b_2(r)=kL(r \log r+1) +\de r^2$,
that
\ba\label{a-priori-vu-level5}
&\frac{\rm d}{\dt} \int_\O \left[ \frac{1}{2} \vr_{n,\s_3} |\vu_{n,\s_3}|^2
+ \frac{a}{\g-1} \vr_{n,\s_3}^\g + {\frac{\s_1}{{\Gamma}-1} \vr_{n,\s_3}^{\Gamma}}
+ k L  (\eta_{n,\s_3} \log \eta_{n,\s_3} + 1) + \de \,\eta_{n,\s_3}^2\right]\dx \\
&\quad +   {\s_2\int_\O (a\g\vr_{n,\s_3}^{\g-2}
+  { \s_1 \Gamma \vr_{n,\s_3}^{\Gamma-2}} )|\nabla_x \vr_{n,\s_3}|^2\,\dx}
+ \e \int_\O\left( \,\frac{k L}{\eta_{n,\s_3}}  +  2 \,\de   \right)
|\nabla_x \eta_{n,\s_3}|^2 \,\dx \\
&\quad + \int_\O \mu^S \left| \frac{\nabla_x \vu_{n,\s_3} + \nabla^{\rm T}_x \vu_{n,\s_3}}{2} - \frac{1}{d} (\Div_x \vu_{n,\s_3}) \II \right|^2 + \mu^B |\Div_x \vu_{n,\s_3}|^2\,\dx \\
&= - \int_\O  {\chi_{\s_3}(\TT_{n,\s_3})} : \nabla_x \vu_{n,\s_3} \,\dx + \frac{\a}{2}
\int_\O  \tr\left(\log \chi_{\s_3}(\TT_{n,\s_3}) \right)  \Div_{ x} \vu_{n,\s_3} \,\dx + \int_\O \vr_{n,\s_3}\,\ff \cdot  \vu_{n,\s_3} \,\dx,
\ea
for a.e. $t \in (0,T_{n,\sigma_3}]$, where we have used that $\TT_{n,\s_3}^S=\TT_{n,\s_3}$ (cf. the paragraph
following eq. \eqref{est-5th-app-vr-eta}).

Similarly as in \eqref{a-priori-tau},
on taking the trace of (\ref{04-l5}) and integrating over $\O$,
we have that, for a.e. $t \in (0,T_{n,\sigma_3}]$,
\ba\label{a-priori-tau-level5}
\frac{\rm d}{\dt} \int_\O \tr\left( \TT_{n,\s_3}\right) \,\dx  + \frac{A_0}{2\l}\int_\O
\tr\left({\chi_{\s_3}(\TT_{n,\s_3})}\right) \dx = \frac{k\,A_0\,d}{2\lambda}
\int_\O (\eta_{n,\s_3}+\a)  \,\dx  +  2 \int_\O  {\chi_{\s_3}(\TT_{n,\s_3})} : \nabla_x \vu_{n,\s_3} \,\dx .
\ea

\subsection{A logarithmic bound}\label{logarithmic-bb}
Following \cite{Barrett-Boyaval}, we introduce the logarithmic cut-off function $G_{\s_3}: \R\to \R$, defined by
\[
%\be\label{def-G}
G_{\s_3} (s) = \left\{\begin{array}{cl}
\log s & \mbox{ if $s \geq \s_3$},\\
{\s_3}^{-1}{s} + \log \s_3   - 1 & \mbox{ if $s \leq \s_3$}.
\end{array}
\right.
%\ee
\]
Since $G_{\s_3}' (s) = \chi_{\s_3}(s)^{-1}$ for all $s\in \R$, we have, for any real symmetric matrix $\TT$, that
\[
%\be\label{def-G'}
G_{\s_3}' (\TT) = \chi_{\s_3}(\TT)^{-1}.
%\ee
\]
%
%As in the previous section, we shall present the calculations for $d \in \{2,3\}$, although, ultimately, the results in this section will be used with $d=2$ only in the subsequent discussion.
It follows from (\ref{04-l5}), (\ref{est-5th-app-u}) and (\ref{est-5th-app-vr-eta})$_6$ that
$\TT_{n,\s_3} \in W^{1,2}(0,T_{n,\s_3};L^2(\O,\R^{d \times d}))$.
Hence, by Lemma \ref{lem-pt-mf}, we have, as $\TT_{n,\s_3}^S=\TT_{n,\s_3}$
and $G_{\s_3} \in C^{1,1}(\R)$ is concave, that
\[
%\be\label{est-log-1-level5}
\d_t \TT_{n,\s_3} : G_{\s_3}' (\TT_{n,\s_3} )
= \d_t \tr\left(G_{\s_3}(\TT_{n,\s_3} )  \right)
\qquad \mbox{a.e.\ on } (0,T_{n,\s_3}] \times \Omega.
%\ee
\]
Further, by \eqref{est-log-2} and \eqref{est-log-1} we deduce that
\ba\label{est-log-2-level5}
&{\rm Div}_x (\vu_{n,\s_3}\,\chi_{\s_3}(\TT_{n,\s_3}))  : G_{\s_3}' (\TT_{n,\s_3} ) \\
& \qquad = \left[ (\vu_{n,\s_3} \cdot \nabla_x)\, \chi_{\s_3}(\TT_{n,\s_3}) + (\Div_{ x} \vu_{n,\s_3})\,\chi_{\s_3}(\TT_{n,\s_3})  \right] : G_{\s_3}' (\TT_{n,\s_3} ) \\
& \qquad = \left((\vu_{n,\s_3} \cdot \nabla_x)\, \chi_{\s_3}(\TT_{n,\s_3}) \right): \chi_{\s_3}(\TT_{n,\s_3})^{-1} + (\Div_{ x} \vu_{n,\s_3})\,\chi_{\s_3}(\TT_{n,\s_3}) : \chi_{\s_3}(\TT_{n,\s_3})^{-1}\\
& \qquad = (\vu_{n,\s_3} \cdot \nabla_x) \tr\left(\log \chi_{\s_3} (\TT_{n,\s_3} ) \right)
+ d \,\Div_{ x} \vu_{n,\s_3}
\qquad \mbox{a.e.\ on } (0,T_{n,\s_3}] \times \Omega
\ea
and
\begin{align*}
%\ba\label{est-log-3-level5}
&- \left(\nabla_x \vu_{n,\s_3} \,\chi_{\s_3}(\TT_{n,\s_3}) + \chi_{\s_3}(\TT_{n,\s_3})\, \nabla_x^{\rm T} \vu_{n,\s_3} \right)  : G_{\s_3}' (\TT_{n,\s_3} )\\
 & \qquad =- \left(\nabla_x \vu_{n,\s_3} \,\chi_{\s_3}(\TT_{n,\s_3}) + \chi_{\s_3}(\TT_{n,\s_3})\, \nabla_x^{\rm T} \vu_{n,\s_3} \right)  : \chi_{\s_3}(\TT_{n,\s_3})^{-1}\\
&\qquad =-2\, \Div_{ x} \vu_{n,\s_3}
\qquad \mbox{a.e.\ on } (0,T_{n,\s_3}] \times \Omega.
%\ea
\end{align*}
Thus, by taking the Frobenius inner product of \eqref{04-l5} with
$G_{\s_3}' (\TT_{n,\s_3})$ we get that
\ba\label{eq-log-tau-level5}
&\d_t \tr\left(G_{\s_3}(\TT_{n,\s_3} )\right) + (\vu_{n,\s_3} \cdot \nabla_x)
\tr\left( \log \chi_{\s_3} (\TT_{n,\s_3} ) \right)
+(d-2)\, \Div_{ x} \vu_{n,\s_3}  \\
&\quad = \e \Delta_x \TT_{n,\s_3} : \chi_{\s_3} (\TT_{n,\s_3} )^{-1} +
\frac{k\,A_0}{2\lambda}(\eta_{n,\s_3}+\a)  \,\tr\left(\chi_{\s_3} (\TT_{n,\s_3} )^{-1}\right) -
 \frac{ d\,A_0}{2\lambda}
\qquad \mbox{a.e.\ on } (0,T_{n,\s_3}] \times \Omega.
%\chi_{\s_3} (\TT_{n,\s_3}) :
%\,\chi_{\s_3} (\TT_{n,\s_3} )^{-1}.
\ea
Integrating \eqref{eq-log-tau-level5} over $\O$ implies, for a.e.\ $t \in (0,T_{n,\s_3}]$, that
\ba\label{est-log-f1-level5}
\frac{\rm d}{\dt} \int_\O \tr\left(G_{\s_3}(\TT_{n,\s_3} \right) \dx
&=  \int_\O  \tr\left(\log \chi_{\s_3} (\TT_{n,\s_3} ) \right) \Div_{ x} \vu_{n,\s_3}\dx
+ \e\int_\O \Delta_x \TT_{n,\s_3} : \chi_{\s_3} (\TT_{n,\s_3} )^{-1}\,\dx \\
&\quad + \frac{k\,A_0}{2\lambda}\int_\O  (\eta_{n,\s_3}+\a)  \,\tr\left(
\chi_{\s_3} (\TT_{n,\s_3} )^{-1}\right)\,\dx - \frac{d\,A_0}{2\lambda}|\O|.
\ea
To proceed, we require the following generalization of Lemma \ref{lem-lap-tau-inverse},
whose proof is elementary but rather lengthy and has been therefore relegated
to Appendix \ref{appendixc}.

\begin{lemma}\label{lem-lap-tau-inverses3}
For $\sigma_3>0$, let $T_{\sigma_3}>0$, and suppose that
$\PP \in C([0,T_{\sigma_3}];W^{1,2}(\O;\R^{d\times d}))$
%\cap W^{1,2}(0,T_{\sigma_3};L^2(\O;\R^{d\times d}))$
is a symmetric matrix function, with $\Delta_x \PP \in
L^2(0,T_{\sigma_3};L^2(\O;\R^{d\times d}))$, satisfying a
homogeneous Neumann boundary condition on $\partial\Omega$;
then, $\chi_{\sigma_3}(\PP)^{-1} \in L^\infty(0,T_{\sigma_3};
W^{1,2}(\Omega;\mathbb{R}^{d \times d}))$, and
\ba\label{est-log-lap-tau-level5}
\int_\O  \Delta_x \PP : \chi_{\s_3}(\PP)^{-1}\,\dx  &= -\int_\O
\nabla_x \PP  :: \nabla_x\chi_{\s_3} (\PP)^{-1}\,\dx \geq \frac{1}{d}
\int_\O \left|\nabla_x \tr\left( \log \chi_{\s_3} (\PP)
\right)\right|^2\,\dx, \qquad \mbox{a.e. on $(0,T_{\sigma_3}]$}.
\ea
\end{lemma}
Thanks to \eqref{a-priori-vu-level5}, \eqref{a-priori-tau-level5},
\eqref{est-log-f1-level5} and \eqref{est-log-lap-tau-level5}
%\footnote{At present Lemma \ref{lem-lap-tau-inverses3} is written for smooth
%symmetric $\TT$. Need to justify its use for $\TT_{n,\s_3}$
%or write Lemma for such regularity.  }
we then obtain, for a.e.\ $t \in (0,T_{n,\s_3}]$, that
\ba\label{a-priori-log-f2-level5}
&\frac{\rm d}{\dt} \int_\O \left[ \frac{1}{2} \vr_{n,\s_3} |\vu_{n,\s_3}|^2 + \frac{a}{\g-1} \vr_{n,\s_3}^\g + {\frac{\s_1}{{\Gamma}-1} \vr_{n,\s_3}^{\Gamma}} + k L  (\eta_{n,\s_3}
\log \eta_{n,\s_3} + 1) + \de \,\eta_{n,\s_3}^2 \right]\dx \\
&\quad + \frac{1}{2}\frac{\rm d}{\dt} \int_\O %\left[
\tr \left(\TT_{n,\s_3} %\right)
-
\a\, %\,\tr\left(
G_{\s_3}(\TT_{n,\s_3})\right)%\right]
\dx\\
&\quad +   {\s_2\int_\O (a\g\vr_{n,\s_3}^{\g-2} +  { \s_1 \Gamma \vr_{n,\s_3}^{\Gamma-2}} )|\nabla_x \vr_{n,\s_3}|^2\,\dx} + \e \int_\O \left( \frac{k L}{\eta_{n,\s_3}}  +  2 \,\de  \right)
|\nabla_x \eta_{n,\s_3}|^2 \,\dx \\
&\quad + \int_\O \mu^S \left| \frac{\nabla_x \vu_{n,\s_3} + \nabla^{\rm T}_x \vu_{n,\s_3} }{2} - \frac{1}{d} (\Div_x \vu_{n,\s_3}) \II \right|^2 + \mu^B |\Div_x \vu_{n,\s_3}|^2\,\dx
+ \frac{A_0}{4\l}\int_\O \tr\left(\chi_{\s_3}(\TT_{n,\s_3})\right) \,\dx
\\
&\quad
+ \frac{\a\,k\,A_0}{4\lambda} \int_\O  (\eta_{n,\s_3}+\a)  \,\tr\left(\chi_{\s_3} (\TT_{n,\s_3} )^{-1}\right)\dx +  \frac{\a\,\e}{2d}  \int_\O \left|\nabla_x \tr\left(\log \chi_{\s_3} (\TT_{n,\s_3} ) \right)\right|^2\,\dx\\
&\leq   \int_\O \vr_{n,\s_3}\,\ff \cdot  \vu_{n,\s_3} \,\dx +
\frac{k\,A_0\,d}{4\lambda} \int_\O (\eta_{n,\s_3}+\a)  \,\dx + \frac{\a\,d\,A_0}{4\lambda}|\O|.
\ea
Since we can take $\s_3 < \a$, it is straightforward to see that, for any $s\in \R$, one has
$$
%f(s):=
s - \a \, G_{\s_3}(s) \geq \a - \a \log \a.
$$
Thus,
\[
%\ba\label{tr-tau-tr-G-tau-level5}
\tr \left(\TT_{n,\s_3} - \a\,G_{\s_3}(\TT_{n,\s_3})\right) = \sum_{j=1}^d \left(\l^{(j)}_{n,\s_3} - \a\, G_{\s_3}(\l^{(j)}_{n,\s_3})\right) \geq d(\a - \a \log \a),
%\ea
\]
where $\l^{(j)}_{n,\s_3},\ j=1, \ldots, d$, are the eigenvalues of the symmetric matrix $\TT_{n,\s_3}$. Then, similarly as in \eqref{def-energy-free}, we define the following nonnegative functional:
\begin{align*}
%\ba\label{def-energy-free-level5}
E_{n,\s_3}(t)&:=  \int_\O \left[ \frac{1}{2} \vr_{n,\s_3} |\vu_{n,\s_3}|^2 + \frac{a}{\g-1} \vr_{n,\s_3}^\g + {\frac{\s_1}{{\Gamma}-1} \vr_{n,\s_3}^{\Gamma}} + k L  (\eta_{n,\s_3} \log \eta_{n,\s_3} + 1) + \de \,\eta_{n,\s_3}^2 \right]\dx \\
& \qquad  + \frac{1}{2}\int_\O \left[  \tr \left(\TT_{n,\s_3} - \a \,G_{\s_3}(\TT_{n,\s_3})\right)
+ d( \a \log \a - \a ) \right]\dx.
%\ea
\end{align*}
Similarly to (\ref{a-priori-2-0}),
integrating  \eqref{a-priori-log-f2-level5} over $[0,t]$ for any $t\in (0,T_{n,\s_3}]$ then gives
\ba\label{a-priori-2-0-level5}
E_{n,\s_3}(t)  &+   \s_2 \int_0^t \int_\O (a\g\vr_{n,\s_3}^{\g-2}
+  { \s_1 \Gamma \vr_{n,\s_3}^{\Gamma-2}} )|\nabla_x \vr_{n,\s_3}|^2\,\dx\,\dt' +
\e\int_0^t \int_\O \left( \,\frac{k L}{\eta_{n,\s_3}}  +  2 \,\de   \right)
|\nabla_x \eta_{n,\s_3}|^2 \,\dx \,\dt' \\
& + \int_0^t \int_\O \mu^S \left| \frac{\nabla_x \vu_{n,\s_3} + \nabla^{\rm T}_x \vu_{n,\s_3}}{2} - \frac{1}{d} (\Div_x \vu_{n,\s_3}) \II \right|^2 + \mu^B |\Div_x \vu_{n,\s_3}|^2\,\dx \,\dt' \\
& + \frac{A_0}{4\l} \int_0^t \int_\O
\tr\left(\chi_{\s_3}(\TT_{n,\s_3})\right) \,\dx \,\dt'
+ \frac{\a\,k\,A_0}{4\lambda} \int_0^t \int_\O  (\eta_{n,\s_3}+\a)  \,\tr\left(
\chi_{\s_3} (\TT_{n,\s_3} )^{-1}\right)\,\dx\,\dt' \\
& + \frac{\a\,\e}{2d} \int_0^t   \int_\O \left|\nabla_x
\tr\left(\log \chi_{\s_3} (\TT_{n,\s_3}) \right)\right|^2\,\dx\,\dt' \\
\leq  &\ E_{0,\th} + \int_0^t \int_\O \vr_{n,\s_3}\,\ff \cdot  \vu_{n,\s_3} \,\dx\,\dt'  + \frac{k\,A_0\,d}{4\lambda} \int_0^t \int_\O (\eta_{n,\s_3}+\a)  \,\dx \,\dt' + \frac{\a\,d\,A_0}{4\lambda}|\O|t,
\\
\leq  &\ E_{0,\th} + C \int_0^t E_{n,\s_3}(t') \,\dt'  + C \,t,
\ea
where the initial energy $E_{0,\th}$ is defined as
\ba\label{def-energy-free0-level5}
E_{0,\th}: = &\int_\O \left[ \frac{1}{2} \vr_{0,\th} |\vu_{0,\th}|^2 + \frac{a}{\g-1} \vr_{0,\th}^\g + + {\frac{\s_1}{{\Gamma}-1} \vr_{0,\th}^{\Gamma}} + \left(k L  (\eta_{0,\th} \log \eta_{0,\th} + 1) + \de \,\eta_{0,\th}^2\right) \right]\dx\\
&+ \frac{1}{2}\int_\O \left[ \tr \left(\TT_{0,\th} - \a \,\log \TT_{0,\th} \right)
+ d( \a \log \a - \a )  \right]\dx.
\ea
Here we have used the fact that $\TT_{0,\th} \geq \th >\s_3$, which implies that $G_{\s_3}(\TT_{0,\th}) \equiv \log \TT_{0,\th}$.
Then, Gronwall's inequality implies, for any $t \in (0,T_{n,\s_3}]$, that
\ba\label{a-priori-2-level5}
E_{n,\s_3}(t)  &+   \s_2 \int_0^t \int_\O (a\g\vr_{n,\s_3}^{\g-2} +  { \s_1 \Gamma \vr_{n,\s_3}^{\Gamma-2}} )|\nabla_x \vr_{n,\s_3}|^2\,\dx\,\dt' +
\e\int_0^t \int_\O \left( \frac{k L}{\eta_{n,\s_3}}  +  2 \,\de \right)
|\nabla_x \eta_{n,\s_3}|^2 \,\dx \,\dt' \\
\qquad &+ \int_0^t \int_\O \mu^S \left| \frac{\nabla_x \vu_{n,\s_3} + \nabla^{\rm T}_x \vu_{n,\s_3}}{2} - \frac{1}{d} (\Div_x \vu_{n,\s_3}) \II \right|^2 + \mu^B |\Div_x \vu_{n,\s_3}|^2\,\dx \,\dt' \\
\qquad &+ \frac{A_0}{4\l} \int_0^t \int_\O \tr\left(\chi_{\s_3}(\TT_{n,\s_3})\right) \,\dx
\,\dt' + \frac{\a\,k\,A_0}{4\lambda} \int_0^t \int_\O  (\eta_{n,\s_3}+\a)  \,
\tr\left(\chi_{\s_3} (\TT_{n,\s_3} )^{-1}\right)\,\dx\,\dt' \\
\qquad&+  \frac{\a\,\e}{2d} \int_0^t   \int_\O \left|\nabla_x \tr
\left(\log \chi_{\s_3} (\TT_{n,\s_3} ) \right)\right|^2\,\dx\,\dt' \\
\leq  &\ ( E_{0,\th}+C\,t)\, {\rm e}^{Ct},\qquad t \in (0,T_{n,\s_3}],
\ea
where $C$ is a positive constant, independent of $t$ and of
the approximation parameters $(\Gamma, \s_1,\s_2,\s_3,n)$.

\subsection{Maximal existence time}

In this section, we shall use the uniform bound \eqref{a-priori-2-level5} to show that $T_{n,\s_3}$, the maximal time of existence for solutions to the fourth level of approximation, is in fact equal to the final time $T$.

By Korn's inequality (\ref{korn}), a partial result from the bound \eqref{a-priori-2-level5} is that
\be\label{nabla-u-L2-level5}
\int_0^{T_{n,\s_3}} \| \nabla_x \vu_{n,\s_3}(t)\|^2_{L^2(\O;\R^{2\times 2})}\, \dt \leq ( E_{0,\th}+C\,T_{n,\s_3})\, {\rm e}^{CT_{n,\s_3}}\leq C(E_{0,\th},T).
\ee
Thanks to Friedrichs' inequality, \eqref{nabla-u-L2-level5} implies that
\be\label{nabla-u-L2-level5a}
\int_0^{T_{n,\s_3}} \|\vu_{n,\s_3}(t)\|^2_{W^{1,2}(\O;\R^2)}\, \dt \leq C(E_{0,\th},T).
\ee
By the equivalence of the $W^{1,2}(\Omega)$ and $W^{1,\infty}(\Omega)$ norms in the finite-dimensional linear space $X_n$ (see \eqref{est-psi-n}), and the Cauchy--Schwarz inequality over $(0,T_{n,\s_3}]$, we then have from \eqref{nabla-u-L2-level5a} that
\be\label{nabla-u-Linf-level5}
\int_0^{T_{n,\s_3}} \| \nabla_x \vu_{n,\s_3}(t)\|_{L^\infty(\O;\R^{2\times 2})}\,\dt \leq C(n,E_{0,\th},T).
\ee
Using \eqref{nabla-u-Linf-level5} it follows from the third line of \eqref{est-5th-app-vr-eta} that we have the following lower and upper bounds on $\vr_{n,\s_3}$ in terms of positive constants:
\[
%\be\label{nabla-r-Linf-level5}
 C(\th,n,E_{0,\th},T)^{-1} \leq    \vr_{n,\s_3}\leq C(\th,n,E_{0,\th},T).
%\ee
\]
Together with the following  partial result from \eqref{a-priori-2-level5}:
$$
\sup_{t\in (0,T_{n,\s_3}]}\|\vr_{n,\s_3}|\vu_{n,\s_3}|^2(t)\|_{L^1(\O)} \leq C( E_{0,\th},T),
$$
we obtain
$$
%\be\label{u-L2-inf1-level5}
\sup_{t\in (0,T_{n,\s_3}]}\| \vu_{n,\s_3}(t)\|_{L^2(\O;\R^2)} \leq C(\th,n,E_{0,\th},T).
%\ee
$$
Again by the properties in $\eqref{est-psi-n}$ of functions in $X_n$, we have
\be\label{u-L2-inf2-level5}
\sup_{t\in (0,T_{n,\s_3}]}\| \vu_{n,\s_3}(t)\|_{C^{2,\beta}(\overline{\O};\R^2)} \leq C(\th,n,E_{0,\th},T).
\ee
Hence, by a continuity argument, the existence time can exceed $T_{n,\s_3}$. Since the bound in \eqref{a-priori-2-level5} is independent of $n,\s_3$, this process can be repeated a finite number of times, as long as the existence time $T_{n,\s_3}<T $, until the final time $T$ is reached, and therefore the maximal existence time $T_{n,\s_3}=T$. Moreover, by \eqref{est-5th-app-vr-eta} and \eqref{u-L2-inf2-level5},  we have the following bounds that are uniform with respect to $\s_3$:
\ba\label{est-level5-2}
\sup_{t\in (0,T]}\| \vu_{n,\s_3}(t)\|_{C^{2,\beta}(\overline{\O};\R^2)} &\leq C(\th,n,E_{0,\th},T),\\
C(\th,n,E_{0,\th},T)^{-1} \leq    \vr_{n,\s_3}(t,x) &\leq C(\th,n,E_{0,\th},T) \quad \mbox{ for all $(t,x)\in (0,T]\times \O$},\\
C(\th,n,E_{0,\th},T)^{-1} \leq  \eta_{n,\s_3}(t,x) &\leq C(\th,n,E_{0,\th},T) \quad \mbox{ for all $(t,x)\in (0,T]\times \O$},\\
\sup_{t\in (0,T]}\| \vr_{n,\s_3}(t)\|_{W^{1,2}( {\O})} + \|\vr_{n,\s_3}\|_{L^2(0,T; W^{2,2}( {\O}))}  &\leq
C(\s_2, \th,n,E_{0,\th},T),\\
\sup_{t\in (0,T]}\| \eta_{n,\s_3}(t)\|_{W^{1,2}( {\O})} + \| \eta_{n,\s_3}\|_{L^2(0,T; W^{2,2}( {\O}))}  &\leq C(\th,n,E_{0,\th},T),\\
\sup_{t\in (0,T]}\| \TT_{n,\s_3}(t)\|_{W^{1,2}( {\O};\R^{2\times 2})} + \| \TT_{n,\s_3}\|_{L^2(0,T; W^{2,2}( {\O};\R^{2\times 2}))}  &\leq C(\th,n,E_{0,\th},T)
\ea
and
\ba\label{est-level5-3}
\sup_{t\in (0,T]} \int_\O \left[  \tr\left(\TT_{n,\s_3} - \a \,G_{\s_3}(\TT_{n,\s_3})\right) +
d( \a \log \a - \a ) \right]\dx  &\leq C(E_{0,\th},T),\\
 \int_0^T \int_\O  \tr\left(\chi_{\s_3} (\TT_{n,\s_3} )^{-1}\right) \,\dx \,\dt
 &\leq C(\th, n,E_{0,\th},T).
\ea

After these preparatory considerations, we are now ready to pass to the limit $\s_3\to 0$ in the fourth level of approximation, so as to deduce the existence of solutions to the third level
of approximation. %in the limit of $\s_3 \rightarrow 0$.
This will be the subject of the next section.

\section{The third level of approximation}\label{sec:3rd-app-full}

This section is devoted to studying the limit of the solution sequence $(\vr_{n,\s_3},\vu_{n,\s_3},\eta_{n,\s_3},\TT_{n,\s_3})$ as $\s_3\to 0$, and to showing that the resulting limit is a solution to the third level of approximation formulated in Section \ref{sec:3rd-app}. We will also derive bounds on this solution limit that are uniform with respect to $n$, in preparation for passage to the limit $n \rightarrow \infty$ in the next section.

\subsection{Time derivative bounds and strong convergence}\label{sec:level3-1}
The bounds in \eqref{est-level5-2} imply the following weak convergence results,
as $\s_3\to 0$:
\begin{alignat*}{2}%\label{cov-vr-l5-1}
\begin{aligned}
\vr_{n,\s_3} &\to \vr_n, &&\quad \mbox{weakly-* in}\  L^\infty(0,T; W^{1,2}( {\O}))
\cap L^2(0,T; W^{2,2}( {\O})),\\
\eta_{n,\s_3} &\to \eta_n, &&\quad \mbox{weakly-* in}\  L^\infty(0,T; W^{1,2}( {\O}))\cap L^2(0,T; W^{2,2}( {\O})),\\
\TT_{n,\s_3} &\to \TT_n, &&\quad \mbox{weakly-* in}\  L^\infty(0,T; W^{1,2}( {\O};\R^{2\times 2}))\cap L^2(0,T; W^{2,2}( {\O};\R^{2\times 2})).
\end{aligned}
\end{alignat*}
From equations \eqref{01-l5}--\eqref{04-l5} and the bounds in \eqref{est-level5-2}
we then have that

\ba\label{est-time-derivative-l5}
 \big\|(\d_t \vr_{n,\s_3}, \d_t \eta_{n,\s_3},\d_t \TT_{n,\s_3})\big\|_{L^2(0,T; L^{2}( {\O}) \times L^{2}( {\O}) \times L^{2}( {\O};\R^{2 \times 2}))} \leq C(\s_2, \th,n,E_{0,\th},T).
\ea

We can therefore use the Aubin--Lions--Simon compactness theorem (cf. Lemma \ref{lem-ALS}) to deduce the following strong convergence results, on noting that $d=2$, as $\s_3\to 0$:
\begin{alignat}{2}\label{cov-vr-l5-2}
\begin{aligned}
\vr_{n,\s_3}  &\to \vr_n, &&\quad \mbox{~~~strongly in}\  L^2(0,T; W^{1,q}( {\O}))\cap
C ([0,T]; L^{q}( {\O})) \qquad \forall q \in [1,\infty),\\
\eta_{n,\s_3} &\to \eta_n,&&\quad \mbox{~~~strongly in}\  L^2(0,T; W^{1,q}( {\O}))\cap
C ([0,T]; L^{q}( {\O})) \qquad \forall q \in [1,\infty),\\
\TT_{n,\s_3} &\to \TT_n,&&\quad \mbox{~~~strongly in}\  L^2(0,T; W^{1,q}( {\O};\R^{2 \times 2}))\cap C ([0,T]; L^{q}( {\O};\R^{2 \times 2})) \qquad \forall q \in [1,\infty),\\
\chi_{\s_3}(\TT_{n,\s_3}) &\to [\TT_n]_{+},&&\quad \mbox{~~~strongly in}\  L^2(0,T; W^{1,q}( {\O};\R^{2 \times 2}))\cap C ([0,T]; L^{q}( {\O};\R^{2 \times 2})) \qquad \forall q \in [1,\infty),
\end{aligned}
\end{alignat}
where the limit $\TT_n$ is also real
symmetric as
$\TT_{\s_3,n}$ is real symmetric for all $\s_3>0,\ n\in \N$.
%Here $2^*= \frac{2d}{d-2}$ if $d\geq 3$ and $2^*=\infty$ if $d=2$.
In addition, it follows from (\ref{est-level5-2})$_{2,3}$ that
\begin{align}
\vr_n,\,\eta_n \ge 0 \   \mbox{ a.e. in} \ (0,T]\times \O.
\label{gen}
\end{align}

Thanks to \eqref{est-level5-2}, \eqref{a-priori-2-level5}
and equation \eqref{02-l5} it follows that
\ba\label{est-vu-l5}
 \big\|\d_t \vu_{n,\s_3} \big\|_{L^2(0,T; L^{2}( {\O};\R^2))} \leq C(\s_2, \th,n,E_{0,\th},T).
\ea
Thus, thanks to the embedding $C^{2,\beta}(\overline{\O};\R^2) \hookrightarrow W^{2+\beta,2}(\Omega;\R^2)$, the compact
embedding $W^{2+\beta,2}(\Omega;\R^2) \hookrightarrow \hookrightarrow W^{2,2}(\Omega;\R^2)$, \eqref{est-level5-2},
and the Aubin--Lions--Simon theorem, we have the following strong convergence result, as $\s_3\to 0$, for each fixed $n \in \N$:
\ba\label{cov-vu-l5}
\vu_{n,\s_3} \to \vu_{n} \ \mbox{strongly in}\  C([0,T]; W^{2,2}( {\O};\R^2)).
\ea

\subsection{Positivity of the extra tress tensor}\label{sec:level3-2}
Employing a technique from \cite{Barrett-Boyaval},
we will now show by using the bound \eqref{est-level5-3}$_2$ that the
limit $\TT_n$ obtained in \eqref{cov-vr-l5-2} is, almost everywhere,
a real symmetric positive definite matrix.

%We begin by noting that, as a consequence of $\mbox{\eqref{cov-vr-l5-2}}_3$,
%$\TT_n$ must be positive semidefinite a.e. on $(0,T] \times \Omega$,
%for all $n \in \mathbb{N}$.
Assume that $\TT_n$ is not symmetric positive definite a.e.\
in $D_n \subset (0,T] \times \Omega$; then, there exists a ${\bf q} \in L^\infty((0,T]\times \Omega;\R^d)$
such that
\begin{align}
[\TT_n]_{+} {\bf q} = {\bf 0}
\mbox{ a.e.\ in } (0,T]\times \Omega
\quad \mbox{with} \quad |{\bf q}|=1
\mbox{ a.e.\ in } D_n
\mbox{ and } {{\bf q}}={\bf 0} \mbox{ a.e.\ in } ((0,T]\times \Omega) \setminus D_n.
\label{evec}
\end{align}
On noting (\ref{est-level5-3})$_2$ and using the Cauchy--Schwarz inequality, we then have that
\ba
|D_n| &= \int_0^T \int_{\Omega} |{\bf q}|^2 \,\dx\,\dt=
\int_0^T \int_{\Omega}
\left([\chi_{\s_3}(\TT_{n,\s_3})]^{-\frac{1}{2}}\,
{\bf q} \right)\cdot \left([\chi_{\s_3}(\TT_{n,\s_3})]^{\frac{1}{2}}\,
{\bf q}\right) \,\dx\,\dt \\
& \leq \left(\int_0^T \int_{\Omega}
|\chi_{\s_3}(\TT_{n,\s_3})^{-1}| \,\dx\,\dt
\right)^{\frac{1}{2}} \left(\int_0^T \int_{\Omega}
{\bf q}^{\rm T}\chi_{\s_3}(\TT_{n,\s_3})\,{\bf q}
\,\dx\,\dt
\right)^{\frac{1}{2}} \\
& \leq C \left(\int_0^T \int_{\Omega}
{\bf q}^{\rm T}\chi_{\s_3}(\TT_{n,\s_3})\,{\bf q}
\,\dx\,\dt
\right)^{\frac{1}{2}},
\label{eveca}
\ea
where $C$ is independent of $\s_3$.
Passing to the limit $\s_3 \rightarrow 0$ in (\ref{eveca}), and noting
(\ref{cov-vr-l5-2})$_4$ and (\ref{evec}), yields that $|D_n|=0$.
Hence, $\TT_n$ is symmetric positive definite a.e. in $(0,T] \times \Omega$.
Finally, it follows from (\ref{a-priori-2-level5}) and (\ref{cov-vr-l5-2})$_4$
that, as $\s_3 \to 0$:
\begin{align}
\nabla_x \tr\left(\log \chi_{\s_3} (\TT_{n,\s_3}) \right)
\rightarrow \nabla_x \tr\left(\log \TT_{n}\right)
\qquad \mbox{weakly in }L^2(0,T;L^2(\O;\R^d)).
\label{logcon}
\end{align}

\subsection{Convergence to the third level of approximation}

By \eqref{est-time-derivative-l5} and \eqref{est-vu-l5},
we have weak convergence of the time derivatives.  By the strong convergence results established in Section \ref{sec:level3-1} and the positivity of $\TT_n$ shown in Section \ref{sec:level3-2},  letting $\s_3\to 0$ in the fourth level of approximation \eqref{02-l5}--\eqref{04-l5} implies that the limit $(\vr_{n },\vu_{n },\eta_{n },\TT_{n })$ is a solution to the third level of approximation, \eqref{02-l3}--\eqref{04-l3}.
 The attainment of the boundary conditions in \eqref{05a}--\eqref{07a} and \eqref{boundary-vr} for $(\vr_{n },\vu_{n },\eta_{n },\TT_{n })$ follows from the attainment of the boundary conditions for $(\vr_{n,\s_3},\vu_{n,\s_3},\eta_{n,\s_3},\TT_{n,\s_3})$. The initial data for $\vr_n,\ \eta_n, \ \TT_n$ are attained in the sense of the $L^q(\Omega)$ norm for any $q<\infty$ by the first three statements in \eqref{cov-vr-l5-2}, and the initial datum for $\vu_n$ is attained in the sense of the $W^{2,2}(\Omega;\R^{2\times 2})$
norm by \eqref{cov-vu-l5}.

Moreover, by the convergence results established in Section
\ref{sec:level3-1}, weak lower-semicontinuity of the norm in $L^p$ spaces and Fatou's lemma,
letting $\s_3\to 0$ in the bounds \eqref{a-priori-2-0-level5} and \eqref{a-priori-2-level5}
gives, for a.e. $t\in (0,T]$, the following inequalities:
\ba\label{a-priori-2-0-level3}
E_{n }(t)  &+   4\s_2 \int_0^t \int_\O \left( \frac{a}{\gamma} | \nabla_x \vr_{n}^{\frac{\g}{2}} |^2 +   \frac{\s_1}{\Gamma} |\nabla_x \vr_{n }^{\frac{\Gamma}{2}} |^2 \right)\dx\,\dt'
+ 2\e \int_0^t \int_\O \left( 2k L |\nabla_x \eta_{n}^{\frac 12}|^2  +  \de |\nabla_x \eta_{n}|^2  \right)\dx \,\dt' \\
\quad &+ \int_0^t \int_\O \mu^S \left| \frac{\nabla_x \vu_{n} + \nabla^{\rm T}_x \vu_{n}}{2} - \frac{1}{d} (\Div_x \vu_{n}) \II \right|^2 + \mu^B |\Div_x \vu_{n}|^2\,\dx \,\dt' \\
\quad &+ \frac{A_0}{4\l} \int_0^t \int_\O \tr\left(\TT_{n}\right) \dx \,\dt'
+ \frac{\a\,k\,A_0}{4\lambda} \int_0^t \int_\O  (\eta_{n}+\a)  \,\tr\,   (\TT_{n}^{-1})\,\dx\,\dt' + \frac{\a\,\e}{2d} \int_0^t   \int_\O \left|\nabla_x \tr\left(\log \TT_{n} \right)\right|^2\dx\,\dt' \\
 \leq  &E_{0,\th} + \int_0^t \int_\O \vr_{n}\,\ff \cdot  \vu_{n} \,\dx\,\dt'  + \frac{k\,A_0\,d}{4\lambda} \int_0^t \int_\O (\eta_{n}+\a)  \,\dx \,\dt' + \frac{\a\,d\,A_0}{4\lambda}|\O|t,
\ea
with $d=2$ and
\ba\label{a-priori-2-level3}
E_{n}(t)  &+    4\s_2 \int_0^t \int_\O \left( \frac{a}{\gamma}
| \nabla_x \vr_{n}^{\frac{\g}{2}} |^2 +   \frac{\s_1}{\Gamma} |\nabla_x \vr_{n }^{\frac{\Gamma}{2}} |^2 \right)\dx\,\dt'
+2\e \int_0^t \int_\O \left( 2k L |\nabla_x \eta_{n}^{\frac 12}|^2  +  \de |\nabla_x \eta_{n}|^2  \right)\dx \,\dt' \\
&+ \int_0^t \int_\O \mu^S \left| \frac{\nabla_x \vu_{n} + \nabla^{\rm T}_x \vu_n}{2} - \frac{1}{d} (\Div_x \vu_{n}) \II \right|^2 + \mu^B |\Div_x \vu_{n}|^2\,\dx \,\dt' \\
&+ \frac{A_0}{4\l} \int_0^t \int_\O \tr\left(\TT_{n}\right) \dx \,\dt' + \frac{\a\,k\,A_0}{4\lambda} \int_0^t \int_\O  (\eta_{n}+\a)  \,\tr\left(\TT_{n}^{-1}\right)\dx\,\dt'  + \frac{\a\,\e}{2d} \int_0^t   \int_\O \left|\nabla_x \tr\left( \log    \TT_{n}  \right)\right|^2\dx\,\dt' \\
\leq  &( E_{0,\th}+C\,t)\, {\rm e}^{Ct},
\ea
with $d=2$, where the energy $E_n$ is defined by
\begin{align*}
%\ba\label{def-energy-free-level3}
E_{n}(t)&:=  \int_\O \left[ \frac{1}{2} \vr_{n} |\vu_{n}|^2 + \frac{a}{\g-1} \vr_{n}^\g + {\frac{\s_1}{{\Gamma}-1} \vr_{n}^{\Gamma}} + k L  (\eta_{n} \log \eta_{n} + 1) + \de \,\eta_{n}^2 \right]\dx \\
&  \qquad + \frac{1}{2}\int_\O \left[  \tr \left(\TT_{n} - \a \,\log \TT_{n}\right) +
d( \a \log \a - \a ) \right]\dx,
%\ea
\end{align*}
with $d=2$, and the initial energy $E_{0,\th}$ is the same as in \eqref{def-energy-free0-level5}.

\section{The second level of approximation}\label{sec:2nd-app-full}

Our objective in this section is to study the limit of the solution sequence $(\vr_{n },\vu_{n },\eta_{n },\TT_{n })$ as $n \to \infty$, in order to deduce the existence of a solution to the second level of approximation, stated in Section \ref{sec:2nd-app}. To this end, we need to derive bounds on $(\vr_{n },\vu_{n },\eta_{n },\TT_{n })$ that are uniform in $n$. We note here that while the steps performed hitherto can be extended to the case of $d=3$, with some restrictions on $q$ in
(\ref{cov-vr-l5-2}), in what follows we shall have to restrict ourselves to the case of $d=2$.

\subsection{Uniform bounds and convergence}\label{sec-con-level3}
We summarize the $n$-uniform bounds that follow from \eqref{a-priori-2-level3}
as $\de >0$:
\ba\label{est-level3-1}
\| \vr_n \|_{L^\infty(0,T;L^\g(\O))} + \s_1 \| \vr_n \|_{L^\infty(0,T;L^{\Gamma}(\O))} &\leq C(E_{0,\th},T),\\
\| \nabla_x (\vr_n^{\frac{\g}{2}})\|_{L^2((0,T)\times \O;\R^2)} +
\s_1\, \| \nabla_x (\vr_n^{\frac{\Gamma}{2}})\|_{L^2((0,T)\times \O;\R^2)} &\leq C(\s_2, E_{0,\th},T),\\
\| \eta_n \|_{L^\infty(0,T;L^2(\O))} + \|  \eta_n \|_{L^2(0,T;W^{1,2}(\O))} + \|  \eta_n^{\frac{1}{2}} \|_{L^2(0,T;W^{1,2}(\O))} &\leq C( E_{0,\th},T),\\
\|\vr_n  |\vu_n|^2\|_{L^\infty(0,T;L^1(\O))} + %&\leq C(E_{0,\th},T),\\
\| \vu_n\|_{L^2(0,T;W^{1,2}_0(\O;\R^2))}  &\leq C(E_{0,\th},T), \\
\| \tr\left(\TT_n - \alpha \,\log \TT_n \right) \|_{L^\infty(0,T;L^1(\O))}
+  \|  (\eta_{n}+\a)  \,\tr\,   (\TT_{n}^{-1}) \|_{L^1(0,T;L^1(\O))}  &\leq C(E_{0,\th},T),\\
%\| \tr\left(\log \TT_n\right) \|_{L^\infty(0,T;L^1(\O))} +
\| \nabla_x \tr\left(\log \TT_n\right) \|_{L^2(0,T;L^2(\O;\R^2))} &\leq C(E_{0,\th},T).
\ea

Multiplying \eqref{01-l3} by $\vr_n$ and integrating the result over $\O$ implies that
\ba\label{est-vr-l3-1}
\frac{1}{2}\frac{\rm d}{\dt} \int_\O \vr_n^2 \,\dx + \s_2 \int_\O |\nabla_x\vr_n|^2 \,\dx &= -\int_\O \Div_x(\vr_n\vu_n)\,\vr_n\,\dx = - \frac{1}{2} \int_\O  (\Div_x \vu_n)\,\vr_n^2\,\dx\\
&\leq \frac{1}{4}\left( \int_\O |\Div_x \vu_n|^2 \,\dx +\int_\O \vr_n^4\,\dx \right).
\ea
Combining this with \eqref{est-level3-1} and recalling
from Section \ref{sec:1st-app} that ${\Gamma} \geq 4$, we deduce that
$$
%\be\label{est-vr-f1}
\|  \vr_n \|_{L^2(0,T;W^{1,2}(\O))} \leq C(\s_1,\s_2,E_{0,\th},T).
%\ee
$$

For $d=2$, as is assumed to be the case here,
taking the inner product of (\ref{04-l5}) with $\TT_n$, integrating the result over $\O$ and applying
the same argument as in Section \ref{sec:2d-est} additionally gives
\ba\label{a-priori-tau-f5-level3}
 \int_\O  | \TT_n(t)|^2 \,\dx + \e  \int_0^t \int_\O |\nabla_x \TT_n |^2  \,\dx\,\dt'
 + \frac{A_0}{4\l}\int_0^t \int_\O | \TT_n|^2 \,\dx\,\dt'  \leq C\left(t,E_{0,\th},\|\TT_{0,\th}\|_{L^2(\O;\R^{2\times 2})}^2\right),
\ea
and furthermore
$$
%\ba\label{est-tau-level3}
\| \TT_n \|_{L^\infty(0,T;L^2(\O;\R^{2\times 2}))} + \|  \TT_n \|_{L^2(0,T;W^{1,2}(\O;\R^{2\times 2}))}  \leq C\left(T,E_{0,\th},\|\TT_{0,\th}\|_{L^2(\O;\R^{2\times 2})}^2\right).
%\ea
$$

Therefore, we have the following weak convergence results, as $n\to \infty$:
\begin{alignat}{2}
\begin{aligned}\label{con-level3-1}
\vr_n  &\to \vr_{\s_2} \quad  &&\mbox{weakly-* in} \  L^\infty(0,T;L^{\Gamma}(\O))\cap L^2(0,T;W^{1,2}(\O)),\\
\eta_n &\to \eta_{\s_2} \quad &&\mbox{weakly-* in} \  L^\infty(0,T;L^2(\O))\cap L^2(0,T;W^{1,2}(\O)),\\
\vu_n  &\to \vu_{\s_2} \quad   &&\mbox{weakly in} \   L^2(0,T;W_0^{1,2}(\O;\R^2)),\\
\TT_n &\to \TT_{\s_2} \quad  &&\mbox{weakly-* in} \  L^\infty(0,T;L^2(\O;\R^{2\times 2}))\cap L^2(0,T;W^{1,2}(\O;\R^{2\times 2})),\\
\tr\,\log \TT_n &\to \overline{\tr\,\log \TT_{\s_2}}   \quad  &&\mbox{weakly in} \   L^2(0,T;W^{1,2}(\O)).
\end{aligned}
\end{alignat}

The time derivative bounds obtained from \eqref{01-l3}--\eqref{04-l3} enable us to use the Aubin--Lions--Simon compactness theorem to obtain the following strong convergence results, as $n\to \infty$:
\begin{alignat}{2}
\begin{aligned}
\label{con-level3-2}
\vr_n  &\to \vr_{\s_2} \quad  &&\mbox{strongly in} \  L^2(0,T; L^{q}(\O))\qquad  \forall\, q \in [1,\infty),\\
\eta_n &\to \eta_{\s_2} \quad &&\mbox{strongly in} \  L^2(0,T;L^q (\O)) \qquad \forall\, q \in [1,\infty),\\
\TT_n &\to \TT_{\s_2} \quad &&\mbox{strongly in} \   L^2(0,T;L^q (\O;\R^{2\times 2})) \qquad \forall\, q \in [1,\infty).
\end{aligned}
\end{alignat}
It follows from (\ref{gen}) that
\begin{align}
\vr_{\s_2},\,\eta_{\s_2} \ge 0 \   \mbox{ a.e. in} \ (0,T]\times \O.
\label{ge2}
\end{align}

By Sobolev embedding we have that
$$
\|\vr_n^\frac{\Gamma}{2} \|_{L^2(0,T;L^{q}(\O))}\leq C
\|\vr_n^\frac{\Gamma}{2}\|_{L^2(0,T;W^{1,2}(\O))} \qquad \forall\, q \in [1,\infty),
$$
and therefore $\|\vr_n^{\Gamma} \|_{L^1(0,T;L^{q}(\O))} \leq C(\s_1, \s_2, E_{0,\th},T)$ for all $q \in [1,\infty)$.
Hence, by  $\eqref{est-level3-1}_1$ and interpolation between Lebesgue spaces, we deduce that
$$
%\ba\label{est-level3-2}
\| \vr_n \|_{L^{(2-\delta)\Gamma}((0,T)\times \O)} \leq C(\s_1,\s_2, E_{0,\th},T)\qquad \forall\, \delta \in (0,1).
%\ea
$$
%
%For example, we can take $\delta=\frac{10}{3}$ when $d=3$ and $\delta=2$ when $d=2$.
Together with \eqref{con-level3-2}, we obtain the strong convergence result, as
$n \rightarrow \infty$:
\begin{align}
%\ba
\label{con-level3-3}
& \vr_n \to \vr_{\s_2}  \  \mbox{strongly in} \  L^{\Gamma}((0,T)\times \O).
%\ea
\end{align}
Thus, we have the following convergence results, as
$n \rightarrow \infty$:
\begin{alignat*}{2}
%\begin{aligned}
%\label{con-nl-level3-1}
\vr_n\vu_n \to &~ \vr_{\s_2}\vu_{\s_2} \quad  &&\mbox{weakly-* in }  L^\infty(0,T;L^{\frac{2\g}{\g+1}}(\O;\R^2)),\\
\vr_n^\g \to \vr^\g_{\s_2},&\   \vr_n^\Gamma \to \vr^\Gamma_{\s_2} \quad
&&\mbox{strongly in }  L^1((0,T)\times \O).
%\end{aligned}
\end{alignat*}

\medskip

Next we shall deal with the nonlinear term $\vr_n\vu_n\otimes \vu_n$.
From \eqref{02-l3}, by the same argument as in Section 7.8.2 in \cite{N-book},
we have that
$$
%\be\label{est-Pn-ru-level3-1}
\| \d_t P_n(\vr_n\vu_n) \|_{L^{r_1}(0,T;\widetilde W^{-2,2}(\O;\R^2))} \leq C(\s_1,E_{0,\th},T),\quad \mbox{for some $ r_1 > 1$},
%\ee
$$
where $P_n$ is the orthogonal projection from $L^2(\Omega;\R^2)$ onto $X_n$
and  $\widetilde W^{-2,2}(\O;\R^2)$ is the dual space of $W_0^{1,2}(\O;\R^2)
\cap W^{2,2}(\O;\R^2)$. On the other hand, since $\Gamma \geq 4$, we deduce from \eqref{est-level3-1} and Sobolev embedding that
\begin{align*}
%\ba\label{est-Pn-ru-level3-2}
\|  P_n(\vr_n\vu_n) \|_{L^{2}(0,T;L^2(\O;\R^2))}  &\leq \|\vr_n\vu_n  \|_{L^{2}(0,T;L^2(\O;\R^2))}\\
&\leq C\|\vr_n\|_{L^{\infty}(0,T;L^{\Gamma}(\O))} \| \nabla_x\vu_n  \|_{L^{2}(0,T;L^2(\O;;\R^{2\times 2}))} \leq  C(\s_1,  E_{0,\th},T).
%\ea
\end{align*}
Thus, the Aubin--Lions--Simon compactness theorem gives
$$
%\ba\label{con-Pn-ru-level3-1}
  P_n( \vr_n\vu_n ) \to  \vr_{\s_2}\vu_{\s_2} \quad \mbox{strongly in}\ L^2(0,T; W^{-1,2}(\O;\R^2)).
%\ea
$$
By writing $\vr_n\vu_n = P_n( \vr_n\vu_n ) + (1-P_n)( \vr_n\vu_n )$ we deduce that
$$
%\ba\label{con-Pn-ru-level3-2}
\vr_n\vu_n \to  \vr_{\s_2}\vu_{\s_2} \quad \mbox{strongly in}\ L^2(0,T; W^{-1,2}(\O;\R^2)).
%\ea
$$
Thus we have the following convergence result for the convective term:
$$
%\ba\label{con-ruu-level3}
\vr_n\vu_n\otimes \vu_n \to  \vr_{\s_2}\vu_{\s_2}\otimes \vu_{\s_2} \quad \mbox{in}\  {\mathcal D}'((0,T)\times \O; \R^{2 \times 2}).
%\ea
$$

\medskip

Finally, we shall study the limit of the extra term $\nabla_x \vu_{n} \nabla_x \vr_n$. To this end, we will employ Lemma \ref{lem-parabolic-1} and Lemma \ref{lem-parabolic-2} to show that the limit $(\vr_{\s_2},\vu_{\s_2} )$ fulfills the parabolic equation \eqref{01-l2} a.e. in $(0,T]\times \O$. By
function space interpolation, we have that
$$
(\vr_n\vu_n)_{n\in \N} \mbox{\ is bounded in \ } L^\infty(0,T; L^\frac{2\Gamma}{\Gamma+1}(\O;\R^2)) \cap L^2(0,T; L^{\Gamma-}(\O;\R^2))\hookrightarrow L^r((0,T)\times \O;\R^2),
$$
for some $r>2$, where $\Gamma-$ denotes any number in the interval $[1,\Gamma)$.  We then apply Lemma \ref{lem-parabolic-2} to \eqref{boundary-vr} and \eqref{01-l3} to deduce that
$$
(\nabla_x\vr_n)_{n\in \N} \mbox{\ is bounded in \ } L^r((0,T)\times \O;\R^2) \mbox{ \ for some $r>2$. \ }
$$
Consequently,
$$
\Div_x ( \vr_n \vu_n ) = \vr_n \Div_x \vu_n + \nabla_x\vr_n \cdot \vu_n
\mbox{ \ is bounded in \ } L^s((0,T)\times \O) \mbox{ \ for some $s>1$. \ }
$$
The application of Lemma \ref{lem-parabolic-1} gives
$$
%\be\label{est-rn-dt}
\|\d_t \vr_n\|_{L^s((0,T)\times \O)} + \|\vr_n \|_{L^s(0,T;W^{2,s}(\O))} \leq C, \quad \mbox{\ for some $C>0$ independent of $n$}.
%\ee
$$

Letting $n\to \infty$ gives
$$
%\be\label{est-r-dt}
\d_t \vr_{\s_2} \in L^s((0,T)\times \O),  \quad  \vr_{\s_2}  \in L^s(0,T;W^{2,s}(\O))\cap L^r(0,T;W^{1,r}(\O)), \quad  \mbox{\ for some $r>2$ and $s>1$}.
%\ee
$$
Moreover, $\vr_{\s_2}$  and $\vu_{\s_2}$ satisfy \eqref{01-l2} a.e. in $(0,T]\times \O$, $\d_{\bf n} \vr_{\s_2} = 0$ on $(0,T]\times \d \O$
and $\vr_{\s_2}(0)=\vr_{0,\th}$.
Therefore, similarly as in \eqref{est-vr-l3-1}, we have, for every $t \in (0,T]$, that
\begin{align*}
%\ba\label{est-vr-vrn-1}
\|\vr_n(t)\|_{L^2(\O)}^2  + 2 \s_2 \|\nabla_x\vr_n\|^2_{L^2((0,t)\times \O;\R^2)}  &= \|\vr_{0,\th}\|_{L^2(\O)}^2 - \int_0^t \int_\O  (\Div_x \vu_n)\,\vr_n^2\,\dx,\\
\|\vr_{\s_2}(t)\|_{L^2(\O)}^2  + 2 \s_2 \|\nabla_x\vr_{\s_2}\|^2_{L^2((0,t)\times \O;\R^2)}  &= \|\vr_{0,\th}\|_{L^2(\O)}^2 - \int_0^t \int_\O  (\Div_x \vu_{\s_2})\,\vr_{\s_2}^2\,\dx.
%\ea
\end{align*}

Letting $n\to \infty$, noting (\ref{con-level3-3}), (\ref{con-level3-1})$_3$
and by the weak lower-semicontinuity of the $L^p$ norm we deduce, for any $t\in (0,T]$, that
\begin{align*}
%\ba\label{cov-vr-vrn-1}
&  \|\vr_n(t)\|_{L^2(\O)}^2 \to  \|\vr_{\s_2}(t)\|_{L^2(\O)}^2,\quad \|\nabla_x\vr_n\|^2_{L^2((0,t)\times \O;\R^2)} \to  \|\nabla_x\vr_{\s_2}\|^2_{L^2((0,t)\times \O;\R^2)}.
%\ea
\end{align*}
This implies the strong convergence of $\nabla_x\vr_n$ and, in addition,
$$
%\be\label{con-extra-vr-vn}
\nabla_x \vu_{n} \nabla_x \vr_n \to \nabla_x \vu_{\s_2} \nabla_x \vr_{\s_2}  \quad \mbox{in}\  {\mathcal D}'((0,T)\times \O;\R^2).
%\ee
$$
We shall combine the convergence results established in this section to show that the limit $(\vr_{\s_2},\vu_{\s_2},\eta_{\s_2},\TT_{\s_2})$ solves the second level of approximation; this will be done in Section
\ref{con-second-level}. Before doing so however we need to prove the positive definiteness of the limiting symmetric extra stress tensor $\TT_{\s_2}$.

\subsection{Positivity of the extra stress tensor}\label{sec:tau-positive-level2}
As $\TT_n$ is symmetric positive definite a.e.\ in $(0,T] \times \Omega$,
we have from (\ref{con-level3-2})$_3$ that
$\TT_{\s_2}$ is symmetric nonnegative definite a.e.\ in
$(0,T] \times \Omega$.
It follows from (\ref{est-level3-1})$_5$ and (\ref{ge2}) that
\begin{align}
\|\tr\left(\TT_{n}^{-1}\right)\|_{L^1(0,T;L^1(\Omega))} \leq C(E_{0,\th},T,\a).
\label{Tninv}
\end{align}
We now adapt the argument in Section \ref{sec:level3-2} to show that $\TT_{\s_2}$
is in fact symmetric positive definite on $(0,T]\times \Omega$.
Assume that $\TT_{\s_2}$ is not positive definite a.e.\
in $D_{\s_2} \subset (0,T] \times \Omega$.
Then there exists a ${\bf q} \in L^\infty((0,T]\times \Omega;\R^d)$
such that
\begin{align}
\TT_{\s_2} {\bf q} = {\bf 0}
\mbox{ a.e.\ in } (0,T]\times \Omega
\quad \mbox{with} \quad |{\bf q}|=1
\mbox{ a.e.\ in } D_{\s_2}
\mbox{ and } {{\bf q}}={\bf 0} \mbox{ a.e.\ in } ((0,T]\times \Omega)
\setminus D_{\s_2}.
\label{evec2}
\end{align}
On noting (\ref{Tninv}), we then have that
\ba
|D_{\s_2}| &= \int_0^T \int_{\Omega} |{\bf q}|^2 \,\dx\,\dt=
\int_0^T \int_{\Omega}
\left((\TT_{n})^{-\frac{1}{2}}\,
{\bf q} \right)\cdot \left((\TT_{n})^{\frac{1}{2}}\,
{\bf q}\right) \,\dx\,\dt \\
& \leq C \left(\int_0^T \int_{\Omega}
{\bf q}^{\rm T} \TT_{n}\,{\bf q}
\,\dx\,\dt
\right)^{\frac{1}{2}},
\label{evecas2}
\ea
where $C$ is independent of $n$.
Passing to the limit $n \rightarrow \infty$ in (\ref{evecas2}), and noting
(\ref{con-level3-2})$_3$ and (\ref{evec2}), yields that $|D_{\s_2}|=0$.
Hence, $\TT_{\s_2}$ is symmetric positive definite a.e.\ in $(0,T] \times \Omega$.
Finally, by \eqref{con-level3-1}$_5$ and \eqref{con-level3-2}$_3$ we deduce, as
$n \rightarrow \infty$, that
$$
\tr\left(\log \TT_n\right) \to \tr\left(\log \TT_{\s_2}\right)
\quad  \mbox{weakly in} \   L^2(0,T;W^{1,2}(\O)).
$$
\subsection{Convergence to the second level of approximation}\label{con-second-level}

 We have already shown that $\vr_{\s_2}$  and $\vu_{\s_2}$ satisfy
 \eqref{01-l2} a.e. in $(0,T]\times \O$, $\d_{\bf n} \vr_{\s_2} = 0$ on $(0,T]\times \d \O$
 and $\vr_{\s_2}(0)=\vr_{0,\th}$.

By the convergence results obtained in Section \ref{sec-con-level3} and a compactness argument, letting $n\to \infty$ in \eqref{02-l3} implies that, for any $\vvarphi \in C^\infty([0,T]; \DC({\Omega};\R^d))$, we have
\begin{align*}
%\ba\label{weak-form3-level2}
&\int_0^t \intO{ \big[ \vr_{\s_2} \vu_{\s_2} \cdot \partial_t \vvarphi + (\vr_{\s_2} \vu_{\s_2} \otimes \vu_{\s_2}) : \Grad \vvarphi  + (p(\vr_{\s_2})+\s_1\vr_{\s_2}^{\Gamma})\, \Div_x \vvarphi \big] } \, \dt'\\
& \quad + \int_0^t \intO{ \big[  \big( kL\eta_{\s_2} + \de\,\eta_{\s_2}^2\big)\, \Div_x \vvarphi  - \SSS(\nabla_x \vu_{\s_2}) : \Grad \vvarphi - \s_2 \nabla_x \vu_{\s_2} \nabla_x \vr_{\s_2}  \cdot \vvarphi \big] } \, \dt'\\
&= \int_0^t \intO{ \TT_{\s_2} : \nabla_x\vvarphi + \frac{\a}{2} (\tr\,\log \TT_{\s_2}) \,\Div_{ x}\vvarphi - \vr_{\s_2}\, \ff \cdot \vvarphi} \, \dt' \\
&\quad + \intO{ \vr_{\s_2} \vu_{\s_2} (t, \cdot) \cdot \vvarphi (t, \cdot) } - \intO{ \vr_{0,\th} \vu_{0,\th} \cdot \vvarphi(0, \cdot) }.
%\ea
\end{align*}
Again by the convergence results obtained in Section \ref{sec-con-level3}, we deduce that the weak formulations \eqref{weak-form2} and  \eqref{weak-form4} are satisfied by the limit $(\vr_{\s_2},\vu_{\s_2},\eta_{\s_2},\TT_{\s_2})$.

\medskip

Moreover, by the convergence results established in Section \ref{sec-con-level3}, weak lower-semicontinuity of the norm in $L^p$ spaces and Fatou's lemma,  letting $n \to \infty$ in the inequalities \eqref{a-priori-2-0-level3} and \eqref{a-priori-2-level3}
gives, for a.e. $t\in (0,T]$:
%
%\begin{align*}
\ba\label{a-priori-2-0-level2}
E_{\s_2}(t)  &+   4\s_2 \int_0^t \int_\O \left( \frac{a}{\g} | \nabla_x \vr_{\s_2}^{\frac{\g}{2}} |^2 +    \frac{\s_1}{\Gamma} |\nabla_x \vr_{\s_2}^{\frac{\Gamma}{2}} |^2 \right)\dx\,\dt'
+ 2\e\int_0^t \int_\O \left( 2k L |\nabla_x \eta_{\s_2}^{\frac 12}|^2  +  \de |\nabla_x \eta_{\s_2}|^2  \right)\dx \,\dt' \\
&+ \int_0^t \int_\O \mu^S \left| \frac{\nabla_x \vu_{\s_2} + \nabla^{\rm T}_x \vu_{\s_2} }{2} - \frac{1}{d} (\Div_x \vu_{\s_2}) \II \right|^2 + \mu^B |\Div_x \vu_{\s_2}|^2\,\dx \,\dt' \\
&+ \frac{A_0}{4\l} \int_0^t \int_\O \tr\left(\TT_{\s_2}\right) \dx \,\dt'
+ \frac{\a\,k\,A_0}{4\lambda} \int_0^t \int_\O  (\eta_{\s_2}+\a)  \,\tr\left(\TT_{\s_2}^{-1}\right)
\dx\,\dt' + \frac{\a\,\e}{2d} \int_0^t   \int_\O \left|\nabla_x \tr\left(\log \TT_{\s_2} \right)\right|^2\dx\,\dt' \\
\leq&\  E_{0,\th} + \int_0^t \int_\O \vr_{\s_2}\,\ff \cdot  \vu_{\s_2} \,\dx\,\dt'  + \frac{k\,A_0\,d}{4\lambda} \int_0^t \int_\O (\eta_{\s_2}+\a)  \,\dx \,\dt' + \frac{\a\,d\,A_0}{4\lambda}|\O|t,
\ea
%\end{align}
%
with $d=2$, and
\ba\label{a-priori-2-level2}
E_{\s_2}(t) &+   4\s_2 \int_0^t \int_\O \left( \frac{a}{\g} | \nabla_x \vr_{\s_2}^{\frac{\g}{2}} |^2 +   \frac{\s_1}{\Gamma} |\nabla_x \vr_{\s_2}^{\frac{\Gamma}{2}} |^2 \right)\dx\,\dt'
+ 2\e\int_0^t \int_\O \left( 2 k L |\nabla_x \eta_{\s_2}^{\frac 12}|^2  + \de |\nabla_x \eta_{\s_2}|^2  \right) \dx \,\dt' \\
&+ \int_0^t \int_\O \mu^S \left| \frac{\nabla_x \vu_{\s_2} + \nabla^{\rm T}_x \vu_{\s_2} }{2} - \frac{1}{d} (\Div_x \vu_{\s_2}) \II \right|^2 + \mu^B |\Div_x \vu_{\s_2}|^2\,\dx \,\dt' \\
&+ \frac{A_0}{4\l} \int_0^t \int_\O \tr \left(\TT_{\s_2}\right) \dx \,\dt'
+ \frac{\a\,k\,A_0}{4\lambda} \int_0^t \int_\O  (\eta_{\s_2}+\a)
\,\tr\left(\TT_{\s_2}^{-1}\right)\dx\,\dt'  + \frac{\a\,\e}{2d} \int_0^t
\int_\O \left|\nabla_x \tr\left(\log    \TT_{\s_2}  \right)\right|^2 \dx\,\dt' \\
\leq&\  ( E_{0,\th}+C\,t)\, {\rm e}^{Ct},
\ea
with $d=2$,  where the energy $E_{\s_2}$ is defined as
\begin{align*}
%\ba\label{def-energy-free-level2}
E_{\s_2}(t)&:=  \int_\O \left[ \frac{1}{2} \vr_{\s_2} |\vu_{\s_2}|^2 + \frac{a}{\g-1} \vr_{\s_2}^\g + {\frac{\s_1}{{\Gamma}-1} \vr_{\s_2}^{\Gamma}} + k L  (\eta_{\s_2} \log \eta_{\s_2} + 1) + \de \,\eta_{\s_2}^2 \right]\dx \\
&\quad\,  + \frac{1}{2}\int_\O \left[ \tr \left(\TT_{\s_2} - \a \,\log \TT_{\s_2}\right)
 + d( \a \log \a - \a )\right]\dx,
%\ea
\end{align*}
with $d=2$, and the initial energy $E_{0,\th}$ is the same as in \eqref{def-energy-free0-level5}.

Moreover, letting $n \to \infty$ in \eqref{a-priori-tau-f5-level3} implies, for a.e. $t\in (0,T]$, that
\ba\label{a-priori-tau-f5-level2}
 \int_\O  | \TT_{\s_2}(t)|^2 \,\dx + \e  \int_0^t \int_\O |\nabla_x \TT_{\s_2} |^2  \,\dx\,\dt' + \frac{A_0}{4\l}\int_0^t \int_\O | \TT_{\s_2}|^2 \,\dx\,\dt'  \leq C\left(T,E_{0,\th},\|\TT_{0,\th}\|_{L^2(\O)}^2\right).
\ea

\section{The first level of approximation}\label{sec:1st-app-full}

Now we let $\s_2\to 0$ in the solution sequence
$(\vr_{\s_2},\vu_{\s_2},\eta_{\s_2},\TT_{\s_2})$, in order to deduce the existence of a solution to the first level of approximation, formulated in Section \ref{sec:1st-app}. First, we derive uniform bounds on $(\vr_{\s_2},\vu_{\s_2},\eta_{\s_2},\TT_{\s_2})$ as $\s_2\to 0$. It follows directly from \eqref{a-priori-2-level2} and \eqref{a-priori-tau-f5-level2},
as $\de >0$, that
%
%\begin{align*}
\ba\label{est-level2}
\| \vr_{\s_2} \|_{L^\infty(0,T;L^\g(\O))} + \s_1^{\frac{1}{\Gamma}} \| \vr_{\s_2} \|_{L^\infty(0,T;L^{\Gamma}(\O))} &\leq C(E_{0,\th},T),\\
\sqrt{\s_2} \,\| \nabla_x (\vr_{\s_2}^{\frac{\g}{2}})\|_{L^2((0,T)\times \O;\R^2)} + \sqrt{\s_1} \sqrt{\s_2}\, \| \nabla_x (\vr_{\s_2}^{\frac{\Gamma}{2}})\|_{L^2((0,T)\times \O;\R^2)}
&\leq C(E_{0,\th},T),\\
\| \eta_{\s_2} \|_{L^\infty(0,T;L^2(\O))} + \|  \eta_{\s_2} \|_{L^2(0,T;W^{1,2}(\O))} + \|  \eta_{\s_2}^{\frac{1}{2}} \|_{L^2(0,T;W^{1,2}(\O))} &\leq C( E_{0,\th},T),\\
    \|\vr_{\s_2}  |\vu_{\s_2}|^2\|_{L^\infty(0,T;L^1(\O))} %&\leq C(E_{0,\th},T),\\
   + \| \vu_{\s_2}\|_{L^2(0,T;W^{1,2}_0(\O;\R^2))} &\leq C(E_{0,\th},T), \\
 \| \tr\left(\TT_{\s_2}-\a\,\log \TT_{\s_2}\right) \|_{L^\infty(0,T;L^1(\O))}
 +  \| (\eta_{\s_2}+\a)  \,\tr\, (\TT_{\s_2}^{-1}) \|_{L^1(0,T;L^1(\O))}  &\leq C(E_{0,\th},T),\\
 %\| \tr\,\log \TT_{\s_2} \|_{L^\infty(0,T;L^1(\O))} +
 \| \nabla_x \tr\left(\log \TT_{\s_2}\right) \|_{L^2(0,T;L^2(\O;\R^2))} &\leq C(E_{0,\th},T),\\
\| \TT_{\s_2} \|_{L^\infty(0,T;L^2(\O;\R^{2\times 2}))} + \|  \TT_{\s_2} \|_{L^2(0,T;W^{1,2}(\O;\R^{2\times 2}))}  &\leq C\left(E_{0,\th},T,\|\TT_{0,\th}\|_{L^2(\O;\R^{2\times 2})}^2\right).
\ea
%\end{align*}

The process of letting $\s_2\to 0$ can be performed similarly as in the study of the compressible Navier--Stokes system (see for example Section 3 in \cite{FNP}, where the Bogovski{\u\i} operator (Lemma \ref{lem-bog}) is needed to show the higher integrability of the density, and Lemma \ref{lem-Cw} is used to pass to the limit in the nonlinear terms $\vr_{\s_2}\vu_{\s_2}$ and
$\vr_{\s_2}\vu_{\s_2}\otimes \vu_{\s_2}$), by observing the strong convergence of the additional unknowns $\eta_{\s_2}$ and
$\TT_{\s_2}$. A key step in passing to the limit in a sequence of approximations to
the compressible Navier--Stokes system is the proof of strong convergence of the approximations to the density,
based on weak convergence of the, so called,
effective viscous flux. A helpful tool in the proof of this is Lemma 7.36 in \cite{N-book}; see also Lemma 5.6 in
\cite{Barrett-Suli} and Lemma 2.3 in \cite{BS2016}, which are the appropriate extensions of Lemma 7.36 in \cite{N-book}, required for
 deducing weak convergence of the effective viscous flux in the presence of the extra stress tensor,
 %appearing on the right-hand side of the momentum equation
 in a compressible Navier--Stokes--Fokker--Planck system. Unlike the compressible FENE models in \cite{Barrett-Suli} and \cite{BS2016},
 where only strong convergence of the approximations to the extra stress tensor in $L^r((0,T) \times \Omega)$, with $r \in [1,\frac{4(d+2)}{3d+4})$, was available, for the compressible Oldroyd-B model considered here these extensions are not needed: Lemma 7.36 from \cite{N-book} (suitably adapted to the case of $d=2$;
cf. Lemma 2.3 in \cite{BS2016}) directly applies, as in the case of the
compressible Navier--Stokes system, thanks to $\eqref{est-level2}_{6,7}$, yielding $\eqref{con-level2}_{3,5}$,
thus ensuring fulfillment of condition (7.5.4) of Lemma 7.36 in \cite{N-book}.
From the uniform estimates \eqref{est-level2} and the equations for
$\eta_{\s_2}$ and $\TT_{\s_2}$ we deduce, for any $r \in (1,2)$ as $d=2$, that
\ba
\|\d_t \eta_{\s_2}\|_{L^{2}(0,T;W^{-1,r}(\Omega))}
+ \| \d_t \TT_{\s_2}\|_{L^{2}(0,T;W^{-1,r}(\Omega;\R^{2\times 2}))} \leq C(E_{0,\th},T),
\label{s2dt}
\ea
where $W^{-1,r}(\O)$ is the dual of $W^{1,r'}_0(\O)$, with $1/r + 1/r'=1$.
These time derivative bounds, (\ref{est-level2}) and
the application of the Aubin--Lions--Simon compactness theorem implies, as $\s_2\to 0$, that
\begin{alignat}{2}\label{con-level2}
\begin{aligned}
% \vr_n &\to \vr_{\s_2} \quad  &&\mbox{weakly-* in} \  L^\infty(0,T;L^{\Gamma}(\O))\cap L^2(0,T;W^{1,2}(\O)),\\
\eta_{\s_2} &\to \eta_{\s_1} \quad &&\mbox{weakly-* in} \  L^\infty(0,T;L^2(\O))\cap L^2(0,T;W^{1,2}(\O)),\\
 \eta_{\s_2} &\to \eta_{\s_1} \quad &&\mbox{strongly in} \  L^2(0,T;L^q (\O)) \qquad \forall\, q \in [1,\infty),\\
% \vu_n &\to \vu_{\s_2} \quad  &&\mbox{weakly in} \   L^2(0,T;W_0^{1,2}(\O;\R^2)),\\
 \TT_{\s_2} &\to \TT_{\s_1} \quad  &&\mbox{weakly-* in} \  L^\infty(0,T;L^2(\O;\R^{2\times 2}))\cap L^2(0,T;W^{1,2}(\O;\R^{2\times 2})),\\
\TT_{\s_2} &\to \TT_{\s_1} \quad  &&\mbox{strongly in} \   L^2(0,T;L^q (\O;\R^{2\times 2})) \qquad \forall\, q \in [1,\infty),\\
 \tr\left(\log \TT_{\s_2}\right) &\to  \tr\left(\log \TT_{\s_1}\right)   \quad  &&\mbox{weakly in} \   L^2(0,T;W^{1,2}(\O)).
\end{aligned}
\end{alignat}
Then, by Lemmas \ref{lem-Cw-Cw*} and \ref{lem-Cw}, (\ref{s2dt}) and \eqref{con-level2},
$$
\eta_{\s_1} \in C_w([0,T];L^2(\O)),\quad \TT_{\s_1} \in C_w([0,T];L^2(\O;\R^{2\times 2})).
$$

The nonnegativity of $\vr_{\s_1}$ and $\eta_{\s_1}$ a.e. in $(0,T]\times \O$
follows from (\ref{gen}).
We note that since $\TT_{\sigma_2}$ is symmetric and positive definite a.e.\
in $(0,T] \times \Omega$, it follows from (\ref{con-level2})$_4$ that $\TT_{\sigma_1}$
is symmetric and positive semidefinite a.e.\ in $(0,T] \times \Omega$.
The positive definiteness of  $\TT_{\s_1}>0$ a.e. in $(0,T]\times \O$ can be deduced by an argument that is identical to the one in Section \ref{sec:tau-positive-level2}; we therefore omit the details and only state the conclusion. The limit  $(\vr_{\s_1},\vu_{\s_1},\eta_{\s_1},\TT_{\s_1})$ is a weak solution to the first level of approximation stated in Section \ref{sec:1st-app}.
Passing to the limit $\s_2 \rightarrow 0$ in
(\ref{a-priori-2-0-level2})--(\ref{a-priori-tau-f5-level2}),
one has the following bounds, for a.e. $t\in (0,T]$:
\ba\label{a-priori-2-0-level1}
E_{\s_1}(t)  &+ 2\e\int_0^t \int_\O \left( 2k L |\nabla_x \eta_{\s_1}^{\frac 12}|^2  +  \de |\nabla_x \eta_{\s_1}|^2  \right)\dx \,\dt' \\
\quad &+ \int_0^t \int_\O \mu^S \left| \frac{\nabla_x \vu_{\s_1} + \nabla^{\rm T}_x \vu_{\s_1} }{2} - \frac{1}{d} (\Div_x \vu_{\s_1}) \II \right|^2 + \mu^B |\Div_x \vu_{\s_1}|^2\,\dx \,\dt' \\
\quad &+ \frac{A_0}{4\l} \int_0^t \int_\O \tr\left(\TT_{\s_1}\right) \dx \,\dt'
+ \frac{\a\,k\,A_0}{4\lambda} \int_0^t \int_\O  (\eta_{\s_1}+\a)  \,\tr\left(\TT_{\s_1}^{-1}
\right)\dx\,\dt' + \frac{\a\,\e}{2d} \int_0^t   \int_\O \left|\nabla_x \tr
\left(\log \TT_{\s_1} \right)\right|^2\dx\,\dt' \\
\leq & \  E_{0,\th} + \int_0^t \int_\O \vr_{\s_1}\,\ff \cdot  \vu_{\s_1} \,\dx\,\dt'  + \frac{k\,A_0\,d}{4\lambda} \int_0^t \int_\O (\eta_{\s_1}+\a)  \,\dx \,\dt' + \frac{\a\,d\,A_0}{4\lambda}|\O|t,
\ea
with $d=2$, and
\ba\label{a-priori-2-level1}
E_{\s_1}(t)  &+  2\e\int_0^t \int_\O \e\left( 2 k L |\nabla_x \eta_{\s_1}^{\frac 12}|^2
+  \de |\nabla_x \eta_{\s_1}|^2  \right) \dx \,\dt' \\
&+ \int_0^t \int_\O \mu^S \left| \frac{\nabla_x \vu_{\s_1} + \nabla^{\rm T}_x \vu_{\s_1} }{2} - \frac{1}{d} (\Div_x \vu_{\s_1}) \II \right|^2 + \mu^B |\Div_x \vu_{\s_1}|^2\,\dx \,\dt' \\
&+ \frac{A_0}{4\l} \int_0^t \int_\O \tr\left(\TT_{\s_1} \right)\dx \,\dt' + \frac{\a\,k\,A_0}{4\lambda} \int_0^t \int_\O  (\eta_{\s_1}+\a)  \,\tr\left(\TT_{\s_1}^{-1}\right)\dx\,\dt'  + \frac{\a\,\e}{2d} \int_0^t   \int_\O \left|\nabla_x \tr\left(\log    \TT_{\s_1}  \right)\right|^2\dx\,\dt' \\
\leq & \  ( E_{0,\th}+C\,t)\,{\rm e}^{Ct},
\ea
with $d=2$, and
\ba\label{a-priori-tau-f5-level1}
\int_\O  | \TT_{\s_1}(t)|^2 \,\dx + \e  \int_0^t \int_\O |\nabla_x \TT_{\s_1} |^2  \,\dx\,\dt' + \frac{A_0}{4\l}\int_0^t \int_\O | \TT_{\s_1}|^2 (t',x)\,\dx\,\dt'  \leq C\left(T,E_{0,\th},\|\TT_{0,\th}\|_{L^2(\O;\R^{2\times 2})}^2\right),
\ea
where the energy $E_{\s_1}$ is defined by
\begin{align*}
%\ba\label{def-energy-free-level1}
E_{\s_1}(t)&:=  \int_\O \left[\frac{1}{2}  \vr_{\s_1} |\vu_{\s_1}|^2 + \frac{a}{\g-1} \vr_{\s_1}^\g + {\frac{\s_1}{{\Gamma}-1} \vr_{\s_1}^{\Gamma}} + k L  (\eta_{\s_1} \log \eta_{\s_1} + 1) + \de \,\eta_{\s_1}^2 \right]\dx \\
&\qquad  +\frac{1}{2} \int_\O \left[  \tr\left(\TT_{\s_1} - \a \,\log \TT_{\s_1}\right)
+ d( \a \log \a - \a )\right]\dx,
%\ea
\end{align*}
with $d=2$, and the initial energy $E_{0,\th}$ is the same as in \eqref{def-energy-free0-level5}.

\section{Completion of the proof}\label{sec:completion-proof}

\subsection{Passage to the limits $\s_1 \rightarrow 0$ and $\th \rightarrow 0$}
The next step is to show that the limit $(\vr, \vu,\eta,\TT)$ of the sequence $(\vr_{\s_1},\vu_{\s_1},\eta_{\s_1},\TT_{\s_1})$, as $\s_1\to 0$, is a weak solution to \eqref{01}--\eqref{04}, \eqref{05a}--\eqref{07a},
in the sense of Definition \ref{def-weaksl} with regularized initial data \eqref{ini-data-mollified}, satisfying \eqref{ini-data-mollified-pt}. We choose $\th=\s_1$ and simultaneously pass to the limits $\s_1\rightarrow 0$ and $\th\rightarrow 0$. Having done so, in the next section we shall also pass to the limit
$\alpha \rightarrow 0$ with the regularization parameter $\alpha$, with $\de>0$ held fixed, and in the final section we shall let $\mathfrak{z} \rightarrow 0$, with $L>0$ kept fixed, in order to cover the entire range of the parameter $\de \in [0,\infty)$.
We begin our considerations by noting that, thanks to \eqref{a-priori-2-0-level1}--\eqref{a-priori-tau-f5-level1}, we have the following uniform bounds as $\de >0$:
\ba\label{est-level1}
\| \vr_{\s_1} \|_{L^\infty(0,T;L^\g(\O))} + \s_1^{\frac{1}{\Gamma}} \| \vr_{\s_1} \|_{L^\infty(0,T;L^{\Gamma}(\O))} &\leq C(E_{0,\th},T),\\
%  \sqrt{\s_1} \,\| \nabla_x (\vr_{\s_1}^{\frac{\g}{2}})\|_{L^2((0,T]\times \O;\R^2)} + \sqrt{\s_1} \sqrt{\s_1}\, \| \nabla_x (\vr_{\s_1}^{\frac{\Gamma}{2}})\|_{L^2((0,T]\times \O;\R^2)} &\leq C(E_{0,\th},T),\\
\| \eta_{\s_1} \|_{L^\infty(0,T;L^2(\O))} + \|  \eta_{\s_1} \|_{L^2(0,T;W^{1,2}(\O))} + \|  \eta_{\s_1}^{\frac{1}{2}} \|_{L^2(0,T;W^{1,2}(\O))} &\leq C( E_{0,\th},T),\\
    \|\vr_{\s_1}  |\vu_{\s_1}|^2\|_{L^\infty(0,T;L^1(\O))} %&\leq C(E_{0,\th},T),\\
    + \| \vu_{\s_1}\|_{L^2(0,T;W^{1,2}_0(\O;\R^2))} &\leq C(E_{0,\th},T), \\
 \| \tr\left(\TT_{\s_1}-\a\,\log(\TT_{\s_1})\right)\|_{L^\infty(0,T;L^1(\O))}
 +  \| (\eta_{\s_1}+\a)  \,\tr\, (\TT_{\s_1}^{-1}) \|_{L^1(0,T;L^1(\O))}  &\leq C(E_{0,\th},T),\\
 %\| \tr\,\log \TT_{\s_1} \|_{L^\infty(0,T;L^1(\O))} +
 \sqrt{\a}\, \| \nabla_x \tr\left(\log \TT_{\s_1}\right) \|_{L^2(0,T;L^2(\O;\R^2))} &\leq C(E_{0,\th},T),\\
\| \TT_{\s_1} \|_{L^\infty(0,T;L^2(\O;\R^{2\times 2}))} + \|  \TT_{\s_1} \|_{L^2(0,T;W^{1,2}(\O;\R^{2\times 2}))}  &\leq C\left(E_{0,\th},T,\|\TT_{0,\th}\|_{L^2(\O;\R^{2\times 2})}^2\right).
\ea
Similarly to (\ref{s2dt}), the
uniform estimates \eqref{est-level1} and the equations for
$\eta_{\s_1}$ and $\TT_{\s_1}$ imply, for any $r \in (1,2)$, that
\ba
\|\d_t \eta_{\s_1}\|_{L^{2}(0,T;W^{-1,r}(\Omega))}
+ \| \d_t \TT_{\s_1}\|_{L^{2}(0,T;W^{-1,r}(\Omega;\R^{2\times 2}))} \leq C(E_{0,\th},T).
\label{s1dt}
\ea
Similarly as in \eqref{con-level2}, we have the following convergence results for
$\eta_{\s_1}$ and $\TT_{\s_1}$,
as
$\s_1 \rightarrow 0$:
\ba\label{con-level1}
%\begin{aligned}
% \vr_n &\to \vr_{\s_2} \quad  &&\mbox{weakly-* in} \  L^\infty(0,T;L^{\Gamma}(\O))\cap L^2(0,T;W^{1,2}(\O)),\\
 \eta_{\s_1} &\to \eta  \quad &&\mbox{weakly-* in} \  L^\infty(0,T;L^2(\O))\cap L^2(0,T;W^{1,2}(\O)),\\
 \eta_{\s_1} &\to \eta  \quad &&\mbox{strongly in} \  L^2(0,T;L^q (\O)) \qquad \forall\, q \in [1,\infty),\\
% \vu_n &\to \vu_{\s_2} \quad  &&\mbox{weakly in} \   L^2(0,T;W_0^{1,2}(\O;\R^2)),\\
 \TT_{\s_1} &\to \TT \quad  &&\mbox{weakly-* in} \  L^\infty(0,T;L^2(\O;\R^{2\times 2}))\cap L^2(0,T;W^{1,2}(\O;\R^{2\times 2})),\\
 \TT_{\s_1} &\to \TT \quad  &&\mbox{strongly in} \   L^2(0,T;L^q (\O;\R^{2\times 2})) \qquad \forall\, q \in [1,\infty),\\
 \tr\left(\log \TT_{\s_1}\right) &\to  \tr\left(\log \TT\right)
 \quad  &&\mbox{weakly in} \   L^2(0,T;W^{1,2}(\O)).
%\end{aligned}
\ea
We note that since  $\TT_{\sigma_1}$ is symmetric and positive definite a.e.
on $(0,T] \times \Omega$, it follows
from (\ref{con-level1})$_4$
that $\TT$ is symmetric and positive semidefinite a.e. on $(0,T] \times \Omega$,
and we have
$\eta \geq 0$ and
$\TT>0$ a.e. in $(0,T]\times \O$ by employing the argument from Section
\ref{sec:tau-positive-level2}. The other limit processes, associated with $\vr_{\s_1}$ and $\vu_{\s_1}$, can be performed similarly as in the study of the compressible Navier--Stokes equations, and we refer to Section 4 in \cite{FNP} for details (see also the paragraph following \eqref{est-level2} above). Thus, by observing the strong convergence of the initial data in \eqref{ini-data-mollified-pt}, we deduce that the limit $(\vr,\vu,\eta,\TT)$ is a weak solution to \eqref{01}--\eqref{04},
\eqref{05a}--\eqref{07a}, in the sense of Definition \ref{def-weaksl}, with the initial data satisfying \eqref{ini-data}. The bounds \eqref{a-priori-tau-f5} and (\ref{energy1}) follow
by letting $\s_1=\th \to 0$ in \eqref{a-priori-tau-f5-level1} and
(\ref{a-priori-2-0-level1}), respectively.
 Moreover, thanks to the definition of $E_{\s_1}(t)$, following \eqref{a-priori-tau-f5-level1} above,
   and by noting that the expression appearing on the right-hand side of \eqref{a-priori-2-level1} is
   independent of $\de$, the argument contained in Remark \ref{rem-eta2-tau} implies that the
   constant on the right-hand side of \eqref{a-priori-tau-f5} is independent of $\de$, as long as $L>0$.
%From the uniform estimates in \eqref{est-level1} and \eqref{con-level1} and the equations satisfied %by $\eta_{\s_1}$ and $\TT_{\s_1}$ we deduce that
%$$
%\|\d_t \eta_{\s_1}\|_{L^{\frac{4}{3}}(0,T;W^{-1,\frac{4}{3}}(\Omega))}+ \| \d_t \TT_{\s_2}\|_{L^{\frac{4}{3}}(0,T;W^{-1,\frac{4}{3}}(\Omega;\R^{2\times 2}))} \leq C(E_{0,\th},T).
%$$
Finally, by Lemmas \ref{lem-Cw-Cw*} and \ref{lem-Cw}, (\ref{con-level1}) and
(\ref{s1dt}), we have
$$
\eta \in C_w([0,T];L^2(\O)),\quad \TT  \in C_w([0,T];L^2(\O;\R^{2\times 2})).
$$
The proof of Theorem \ref{thm} is complete.

\subsection{The vanishing logarithmic term limit: passage to the limit $\a \rightarrow 0$}\label{sec:alpha=0}

In this section, we study the process of letting $\a\to 0$ in the problem \eqref{01}--\eqref{04},
\eqref{05a}--\eqref{07a}.
We will show that letting $\a\to 0$, with $\de>0$ held fixed, yields the existence of a global-in-time weak solution to the corresponding problem without the logarithmic term $\frac{\a}{2}
\nabla_x \tr\left(\log\TT\right)$ in (\ref{02})
and with no $\alpha$ term in (\ref{04}),
which is our original model \eqref{01a}--\eqref{07a} in the case of $\de>0$.

\subsubsection{Weak solutions and main theorem}
We re-iterate our hypotheses on the initial data, but this time we do so without requiring the positivity of the initial extra stress tensor (only its symmetry and nonnegativity are assumed):
\ba\label{ini-data-f}
&\vr(0,\cdot) = \vr_0(\cdot) \ \mbox{with}\ \vr_0 \geq 0 \ {\rm a.e.} \ \mbox{in} \ \O, \quad \vr_0 \in L^\gamma(\O),\\
&\vu(0,\cdot) = \vu_0(\cdot) \in L^r(\O;\R^d) \ \mbox{for some $r\geq 2\g'$}\ \mbox{such that}\ \vr_0|\vu_0|^2 \in L^1(\O),\\
&\eta(0,\cdot)=\eta_0 \ \mbox{with}\ \eta_0 \geq 0 \ {\rm a.e.} \ \mbox{in} \ \O, \quad \eta_0 \in L^2 (\O), \\ %\int_\Omega \eta_0(x)\,\dx = 1,\\
&\TT(0,\cdot) = \TT_0(\cdot) \ \mbox{with}\ \TT_0=\TT_0^{\rm T} \geq 0 \ \mbox{a.e. in}  \ \O ,\quad  \TT_0 \in L^2(\O;\R^{d \times d}).
\ea
The corresponding weak solution is defined similarly as in Definition \ref{def-weaksl}.

\begin{definition}\label{def-weaksl-f} Let $T>0$ and suppose that $\O\subset \R^d$ is a bounded $C^{2,\beta}$ domain, with  $0<\beta<1$. Assume further that $\ff\in L^\infty((0,T]\times \O;\R^d)$. We say that $(\vr,\vu,\eta,\TT)$ is a finite-energy
weak solution in $(0,T]\times \O$ to the system of equations \eqref{01a}--\eqref{07a},  supplemented by the initial data \eqref{ini-data-f}, if:
\begin{itemize}
\item $\vr \geq 0 \ {\rm a.e.\  in} \ (0,T] \times \Omega,\quad  \vr \in  C_w ([0,T];  L^\gamma(\Omega)),\quad \vu\in L^{2}(0,T;W_0^{1,2}(\Omega; \R^d)),\quad
\TT \mbox{ \rm is symmetric}$,
\begin{align*}
&\vr \vu \in C_w([0,T]; L^{\frac{2 \gamma}{ \gamma + 1}}(\Omega; \R^d)),\quad \vr |\vu|^2 \in L^\infty(0,T; L^{1}(\Omega)),\\
&\eta  \geq 0  \ {\rm a.e.\  in} \ (0,T] \times \Omega,\quad \eta \in C_w ([0,T];  L^2(\Omega)) \cap L^2 (0,T;  W^{1,2}(\Omega)),\\
&{ \TT \geq 0  \ {\rm a.e.\  in} \ (0,T] \times \Omega,\quad \TT \in C_w ([0,T];  L^2(\Omega;\R^{d \times d})) \cap L^2 (0,T;  W^{1,2}(\Omega;\R^{d \times d}))}.
\end{align*}
\item For any $t \in (0,T]$ and any test function $\phi \in C^\infty([0,T] \times \Ov{\Omega})$, one has
\be\label{weak-form1-f}
\int_0^t\intO{\big[ \vr \partial_t \phi + \vr \vu \cdot \Grad \phi \big]} \,\dt' =
\intO{\vr(t, \cdot) \phi (t, \cdot) } - \intO{ \vr_{0} \phi (0, \cdot) },
\ee
\be\label{weak-form2-f}
\int_0^t \intO{ \big[ \eta \partial_t \phi + \eta \vu \cdot \Grad \phi - \e \nabla_x\eta \cdot \nabla_x \phi \big]} \, \dt' =  \intO{ \eta(t, \cdot) \phi (t, \cdot) } - \intO{ \eta_{0} \phi (0, \cdot) }.
\ee
\item For any $t \in (0,T]$ and any test function $\vvarphi \in C^\infty([0,T]; \DC({\Omega};\R^d))$, one has
\ba\label{weak-form3-f}
&\int_0^t \intO{ \big[ \vr \vu \cdot \partial_t \vvarphi + (\vr \vu \otimes \vu) : \Grad \vvarphi  + p(\vr)\, \Div_x \vvarphi + \big(kL\eta+\de\,\eta^2\big)\, \Div_x \vvarphi  - \SSS(\nabla_x \vu) : \Grad \vvarphi \big] } \, \dt'\\
&= \int_0^t \intO{ \TT : \nabla_x\vvarphi - \vr\, \ff \cdot \vvarphi} \, \dt' + \intO{ \vr \vu (t, \cdot) \cdot \vvarphi (t, \cdot) } - \intO{ \vr_{0} \vu_{0} \cdot \vvarphi(0, \cdot) }.
\ea
\item For any $t \in (0,T]$ and any test function $\YY \in C^\infty([0,T] \times \Ov{\Omega};\R^{d \times d})$, one has
\ba\label{weak-form4-f}
&\int_0^t \intO{ \left[ \TT : \partial_t \YY +  (\vu \,\TT) :: \nabla_{x} \YY
+ \left(\nabla_x \vu \,\TT + \TT\, \nabla_x^{\rm T} \vu \right):\YY
- \e \nabla_x  \TT :: \nabla_x \YY  \right]}\,\dt'\\
&=\int_0^t \intO{ \left[ - \frac{k\,A_0}{2\lambda}\eta  \,\tr\left(\YY\right) + \frac{A_0}{2\lambda} \TT:\YY \right]} \, \dt' + \intO{ \TT (t, \cdot) : \YY (t, \cdot) }
- \intO{ \TT_{0} : \YY(0, \cdot) }.
\ea
\item The continuity equation holds in the sense of renormalized solutions:
\be\label{weak-renormal-f}
\d_t b(\vr) +\Div_x (b(\vr)\vu) + (b'(\vr)\vr - b(\vr))\,\Div_x \vu =0 \quad \mbox{in} \ \mathcal{D}' ((0,T)\times \O),
\ee
for any $b\in C^0[0,\infty)\cap C^1(0,\infty)$ satisfying {\rm (\ref{cond-b-renormal})}.
%
%\be\label{cond-b-renormal-f}
%|b'(s)|<C \,s^{-\lambda_0}\quad \forall \,s\in (0,1] \qquad
%\mbox{and} \qquad |b'(s)|<C \,s^{\lambda_1}\quad \forall \,s \geq 1,
%\ee
%where $\lambda_0 <1$ and $\lambda_1 \in (-1,\infty)$; see (6.2.9) and (6.2.10) in
%\cite{N-book}.
%
\item  For a.e. $t \in (0,T]$, the following \emph{energy inequality} holds:
\ba\label{energy1-f}
&\int_\O \left[ \frac{1}{2} \vr |\vu|^2 + \frac{a}{\g-1} \vr^\g + \left(k L  (\eta \log \eta + 1) + \de \,\eta^2\right)+ \frac{1}{2}\tr \left(\TT \right)\right]\dx \\
&\quad+ 2\e\int_0^t \int_\O 2k L   |\nabla_x \eta^{\frac12}|^2 +  \de\,  |\nabla_x \eta|^2 \,\dx\,\dt'  + \frac{A_0}{4\l} \int_0^t  \int_\O \tr\left(\TT\right) \dx \,\dt'  \\
&\quad +  \int_0^t  \int_\O \mu^S \left| \frac{\nabla_x \vu + \nabla^{\rm T}_x \vu}{2} - \frac{1}{d} (\Div_x \vu) \II \right|^2 + \mu^B |\Div_x \vu|^2\,\dx \,\dt'  \\
& \leq  \int_\O \left[ \frac{1}{2} \vr_0 |\vu_0|^2 + \frac{a}{\g-1} \vr_0^\g + \left(k L  (\eta_0 \log \eta_0 + 1) + \de \,\eta_0^2\right)+ \frac{1}{2}\tr \left(\TT_0\right) \right]\dx + \int_\O \vr\,\ff \cdot  \vu \,\dx + \frac{k\,A_0\,d}{4\lambda} \int_\O \eta  \,\dx .
\ea
\end{itemize}
\end{definition}

We state the associated result concerning the existence of large data global-in-time finite-energy weak solutions.

\begin{theorem}\label{thm-f}
Let $d=2$, $\g>1$ and $\de>0$. Then, there exists a finite-energy global-in-time weak solution $(\vr,\vu,\eta,\TT)$ to the compressible Oldroyd-B model \eqref{01a}--\eqref{07a} in the sense of Definition \ref{def-weaksl-f} with initial data \eqref{ini-data-f}. Moreover, the extra stress tensor $\TT $ in such a weak solution satisfies the bound \eqref{a-priori-tau-f5}.
\end{theorem}

In the rest of this section we briefly prove Theorem \ref{thm-f}.   The first step is the regularization of the initial stress tensor in order to make it strictly positive definite.  This allows us to apply Theorem \ref{thm} to construct a family of approximating solutions.

\subsubsection{Proof of Theorem \ref{thm-f}} Let $\TT_0$ be as in  \eqref{ini-data-f} and $\a\in (0,1)$. We define:
\be\label{reg-tau-alpha}
\TT_{0,\a} = \TT_0 + \a \,\II.
\ee
Direct calculations give
\be\label{reg-tau-alpha-2}
\TT_{0,\a} \geq \a \,\II >0, \ \mbox{a.e. in } \ \O,\quad
|\tr\left(\log \TT_{0,\a}\right) | \leq d|\log \a| + \tr\left(\TT_0\right) + d\a \in L^1(\O)\quad \mbox{(with $d=2$ in our case here)}.
\ee

We consider the problem \eqref{01}--\eqref{04},
\eqref{05a}--\eqref{07a}, where $\a$ is chosen to be the same as in \eqref{reg-tau-alpha}. The  initial data are as in \eqref{ini-data-f} except that the initial stress tensor is taken to be the regularized one in \eqref{reg-tau-alpha}. When $d=2$, as is assumed to be the case here, by Theorem \ref{thm} and its proof, for any $\a\in (0,1)$, there exists a weak solution $(\vr_\a, \vu_\a, \eta_\a, \TT_\a)$ in the sense of Definition \ref{def-weaksl} satisfying \eqref{a-priori-tau-f5}.

By the energy inequality \eqref{energy1} and Gronwall's inequality we deduce, for a.e. $t\in (0,T]$, that, with $d=2$,
\ba\label{a-priori-2-level0}
&E_{\a}(t)  +  2\e\int_0^t \int_\O \left(2 k L |\nabla_x \eta_{\a}^{\frac 12}|^2  +  \de |\nabla_x \eta_{\a}|^2  \right)   \dx \,\dt' \\
&\quad + \int_0^t \int_\O \mu^S \left| \frac{\nabla_x \vu_{\a} + \nabla^{\rm T}_x \vu_{\a} }{2} - \frac{1}{d} (\Div_x \vu_{\a}) \II \right|^2 + \mu^B |\Div_x \vu_{\a}|^2\,\dx \,\dt' \\
&\quad + \frac{A_0}{4\l} \int_0^t \int_\O \tr\left(\TT_{\a}\right) \dx \,\dt'
+ \frac{\a\,k\,A_0}{4\lambda} \int_0^t \int_\O  (\eta_{\a}+\a)  \,\tr\left(\TT_{\a}^{-1}\right)\dx\,\dt'  + \frac{\a\,\e}{2d} \int_0^t   \int_\O \left|\nabla_x \tr\left( \log    \TT_{\a}  \right)\right|^2\dx\,\dt' \\
& \leq  ( E_{0,\a}+C\,t)\, {\rm e}^{Ct}.
\ea
Here $E_{\a}(t)$ is the same as $E(t)$, as defined in \eqref{def-energy-free}, but with
$(\vr,\vu,\eta,\TT)$ replaced by $(\vr_{\a},\vu_{\a},\eta_{\a},\TT_{\a})$.
Similarly, $E_{0,\a}$ is the same as $E_0$, as defined in \eqref{def-energy-free0},
but with $\TT_0$ replaced by $\TT_{0,\a}$.
We explore the behavior of $E_{0,\a}$ as $\a\to 0$. Thanks to the property registered in \eqref{reg-tau-alpha-2}, the quantity $E_{0,\a}$ is uniformly bounded as $\a\to 0$, and  we have the following convergence result, as $\a\to 0$:
$$
%\be\label{cov-E0-al}
E_{0,\a} \to \int_\O \left[ \frac{1}{2} \vr_0 |\vu_0|^2 + \frac{a}{\g-1} \vr_0^\g + \left(k L  (\eta_0 \log \eta_0 + 1) + \de \,\eta_0^2\right)+ \frac{1}{2}\tr \left(\TT_0\right) \right]\dx.
%\ee
$$
Thus, from \eqref{a-priori-2-level0} and \eqref{a-priori-tau-f5}, we derive analogous uniform bounds to those in \eqref{est-level1}.
Time derivative bounds, similar to (\ref{s1dt}),
obtained from %(\ref{03-l3}) and (\ref{04-l3})
the equations for $\eta_\a$ and $\TT_\a$, and
the application of the Aubin--Lions--Simon compactness theorem
then yield (strong) convergence of the sequences $(\eta_\a)_{\alpha > 0}$ and $(\TT_\a)_{\alpha >0}$.

Letting $\a\to 0$ in \eqref{weak-form1}--\eqref{energy1} we deduce \eqref{weak-form1-f}--\eqref{energy1-f}; here we only deal with the terms associated with $\a$, as all other terms can be handled similarly as in \cite{FNP} in the case of the compressible Navier--Stokes equations, together with the strong convergence we have obtained for the sequences $(\eta_\a)_{\alpha > 0}$ and $(\TT_\a)_{\alpha >0}$.

A partial result of \eqref{a-priori-2-level0} is the following uniform bound:
$$
 \sqrt{\a}\, \| \nabla_x \tr\left(\log \TT_{\a}\right) \|_{L^2(0,T;L^2(\O;\R^2))} \leq C(E_{0,\a},T).
$$
Then, for any $\vvarphi \in C^\infty([0,T]; \DC({\Omega};\R^d)$, as $\a\to 0$,
$$
\left|\frac{\a}{2}\int_0^t \intO{  \tr\left(\log \TT_\a\right) \,\Div_{ x}\vvarphi } \, \dt'\right| \leq \sqrt{\a}\sqrt{\a}\, \| \nabla_x \tr\left(\log \TT_{\a}\right) \|_{L^2(0,T;L^2(\O;\R^2))}  \|  \vvarphi \|_{L^2(0,T;L^2(\O;\R^2))}\to 0.
$$

\medskip

 The energy inequality \eqref{energy1-f} can be deduced  by letting $\a\to 0$ in \eqref{energy1}. Indeed,  thanks to the fact that $s-\log s -1\geq 0$ for any $s>0$, we have $-\log s\geq - s + 1$. Thus,
$$
\tr \left(\TT_\a - \a\, \log \TT_\a \right) + d\left( \a \log \a - \a \right)  \geq
\tr \left(\TT_\a\right) + \a \left(- \tr\left(\TT_\a\right)  + d \right) + d\left( \a \log \a - \a \right) \to \tr\,(\TT).
$$
All of the other terms can be handled directly. The additional bound \eqref{a-priori-tau-f5} follows similarly. The proof of Theorem \ref{thm-f} is thereby complete. \hfill $\Box$

\section{Passing to the limit $\de \to 0$}\label{sec:de-to-0}
Inspired by the conclusions of  \cite{BS2016}, in this  section we shall study the limit process for the problem in Section \ref{sec:alpha=0} as $\de\to 0$, so as to be able to cover the entire parameter range $\de \in [0,\infty)$. We will show that the limiting problem is one that arises by formally setting $\de=0$. To this end we shall assume henceforth that $L>0$ is kept fixed. The initial data are as follows (note, in particular, that the initial polymer number density $\eta_0$ is now only assumed to have
$L \log L(\Omega)$ integrability instead of the, stronger, $L^2(\Omega)$ integrability assumed hitherto (cf. \eqref{ini-data-f}): % +++++
\ba\label{ini-data-ff}
&\vr(0,\cdot) = \vr_0(\cdot) \ \mbox{with}\ \vr_0 \geq 0 \ {\rm a.e.} \ \mbox{in} \ \O, \quad \vr_0 \in L^\gamma(\O),\\
&\vu(0,\cdot) = \vu_0(\cdot) \in L^r(\O;\R^d) \ \mbox{for some $r\geq 2\g'$}\ \mbox{such that}\ \vr_0|\vu_0|^2 \in L^1(\O),\\
& \eta(0,\cdot)=\eta_0 \ \mbox{with}\ \eta_0 \geq 0 \ {\rm a.e.} \ \mbox{in} \ \O, \quad \eta_0\log\eta_0 \in L^1 (\O),\\
&\TT(0,\cdot) = \TT_0(\cdot) \ \mbox{with}\ \TT_0 =\TT_0^{\rm T}\geq 0 \ \mbox{a.e. in}  \ \O ,\quad  \TT_0 \in L^2(\O;\R^{d \times d}).
\ea

We first regularize the initial polymer number density $\eta_0$ given in $\eqref{ini-data-ff}_3$ to obtain a square-integrable function:
\be\label{reg-eta-ff}
\eta_{0,\de} = \frac{\eta_0}{1+\de^{\frac{1}{4}}\eta_0^{\frac 12}}.
\ee
Hence,
\be\label{reg-eta-ff2}
\de \int_\O \eta_{0,\de}^2\,\dx \leq \de^{\frac 12} \int_\O \eta_{0}\,\dx \to 0, \ \mbox{as $\de \to 0$}.
\ee
Furthermore, by direct computations and the Dominated Convergence Theorem we deduce that
\ba\label{reg-eta-ff3}
 \int_\O \eta_{0,\de}\log \eta_{0,\de}\,\dx & =  \int_\O \frac{\eta_0}{1+\de^{\frac{1}{4}}\eta_0^{\frac 12}} \left(\log \eta_{0} - \log \big(1+\de^{\frac{1}{4}}\eta_0^{\frac 12}\big)   \right) \dx \\
 &= \int_\O  \left(\eta_0 \log \eta_{0} - \frac{\de^{\frac{1}{4}} \eta_0^{\frac{3}{2}}}{1+\de^{\frac{1}{4}}\eta_0^{\frac 12}}   \log \eta_{0} - \frac{\eta_0}{1+\de^{\frac{1}{4}}\eta_0^{\frac 12}} \log (1+\de^{\frac{1}{4}}\eta_0^{\frac 12})   \right) \dx\\
 & \to  \int_\O  \left(\eta_0 \log \eta_{0} \right) \dx, \  \ \mbox{as $\de \to 0$}.
\ea
The results in \eqref{reg-eta-ff2} and \eqref{reg-eta-ff3} then allow us to pass to the limit $\de\to 0$ on the right-hand side of the energy inequality \eqref{energy1-f}. We are now ready to state our third main theorem.

\begin{theorem}\label{thm-ff}
Suppose that $L>0$ is held fixed.
For any $\de>0$, let $(\vr_\de,\vu_\de,\eta_\de,\TT_\de) $ be a weak solution in the sense of Definition \ref{def-weaksl-f} with initial data as in \eqref{ini-data-ff}, except that the initial polymer number density is taken as the regularized one, as in \eqref{reg-eta-ff}, satisfying \eqref{reg-eta-ff2} and \eqref{reg-eta-ff3}. We then have that
$$
(\vr_\de,\vu_\de,\eta_\de,\TT_\de) \to (\vr,\vu,\eta,\TT)  \ \mbox{in}\ {\mathcal D}'((0,T)\times \O),\quad \mbox{as} \ \de\to 0,
$$
and the limit $(\vr,\vu,\eta,\TT)$ is a weak solution to the problem \eqref{01a}--\eqref{07a},
%\eqref{05}--\eqref{07},
with initial data \eqref{ini-data-ff}, in the same sense as in Definition \ref{def-weaksl-f} except that $\de$ is taken to be $0$ and $\eta$ is taken in the set of all $\eta$ such that
$$
%\be\label{space-eta-de}
\eta  \geq 0  \ {\rm a.e.\  in} \ (0,T] \times \Omega,\  \eta \in C_w ([0,T];  L^1(\Omega)), \ \eta\log\eta \in L^\infty (0,T; L^1(\O)),\ \eta^{\frac 12} \in L^2 (0,T; W^{1,2}(\O)).
%\ee
$$
\end{theorem}

\begin{proof}
By the energy inequality \eqref{energy1-f} and Gronwall's inequality we deduce the following uniform bounds:
\ba\label{est-level0-f}
\| \vr_{\de} \|_{L^\infty(0,T;L^\g(\O))}   &\leq C ,\\
%  \sqrt{\de} \,\| \nabla_x (\vr_{\de}^{\frac{\g}{2}})\|_{L^2((0,T)\times \O;\R^2)} + \sqrt{\de} \sqrt{\de}\, \| \nabla_x (\vr_{\de}^{\frac{\Gamma}{2}})\|_{L^2((0,T)\times \O;\R^2)} &\leq C(E_{0,\th},T),\\
\| \eta_{\de} \log \eta_{\de}\|_{L^\infty(0,T;L^1(\O))} + \|  \eta_{\de}^{\frac{1}{2}} \|_{L^2(0,T;W^{1,2}(\O))}  &\leq C,\\
 \de \|   \eta_{\de}\|^2_{L^\infty(0,T;L^2(\O))} + \de \|  \eta_{\de} \|^2_{L^2(0,T;W^{1,2}(\O))}  &\leq C,\\
    \|\vr_{\de}  | \vu_{\de}|^2\|_{L^\infty(0,T;L^1(\O))} %&\leq C,\\
+ \| \vu_{\de}\|_{L^2(0,T;W^{1,2}_0(\O;\R^2))} &\leq C, \\
% \| \tr\,\TT_{\de} \|_{L^\infty(0,T;L^1(\O))} +  \|  \eta_{n}  \,\tr\,   (\TT_{n}^{-1}) \|_{L^1(0,T;L^1(\O))}  &\leq C(E_{0,\th},T),\\
% \| \tr\,\log \TT_{\de} \|_{L^\infty(0,T;L^1(\O))} +  \sqrt{\a}\, \| \nabla_x \tr\,\log \TT_{\de} \|_{L^2(0,T;L^2(\O;\R^2))} &\leq C(E_{0,\th},T),\\
\| \TT_{\de} \|_{L^\infty(0,T;L^2(\O;\R^{2\times 2}))} + \|  \TT_{\de} \|_{L^2(0,T;W^{1,2}(\O;\R^{2\times 2}))}  &\leq C,
\ea
where $C$ only depends on $T$ and the initial data; in particular it is independent of $\de$ as $\de\to 0$.
We draw the reader's attention here to the alternatives \eqref{kL1}
and \eqref{kL2}, corresponding to $L > 0$, $\de\geq 0$ and $L\geq 0$ and $\de>0$, respectively, and emphasize that
we are now operating in the first of these two regimes, corresponding to \eqref{kL1}, which guarantees the independence of the constant in the energy inequality \eqref{energy1-f} on $\de$, provided that $L>0$ is held fixed (as has been assumed in the statement of the theorem).
The independence of the constant on $\de$, as $\de\to 0$, in the bounds
on $\eta_\de$ stated in $\eqref{est-level0-f}_2$ and in the bounds on $\TT_\de$
stated in $\eqref{est-level0-f}_5$ can be shown by the same argument as in Remark \ref{rem-eta2-tau}, thanks to $L>0$ being held fixed.

The uniform bounds for $\vr_\de, \ \vu_\de, \ \TT_\de$ are the same as in the previous section, Section \ref{sec:alpha=0}. To understand the limit as $\de\to 0$, we only focus on $\eta_\de$ and the terms in the equations related to $\eta_\de$. The passage to the limit for the other terms can be dealt with similarly as in the previous sections.

We apply Dubinski{\u\i}'s compactness theorem (Lemma \ref{lem-Dubinskii}) to show strong convergence of $\eta_\de$. Let
$$
%\ba\label{dubinskii-spaces}
X:= L^1 (\O ),\quad X_0:=\{ \vp\in X : \vp\geq 0,\ \sqrt{\vp} \in W^{1,2} (\O) \},\quad X_1: =
 W^{-2,2}(\O) = [W^{2,2}_0(\O)]',
%\ea
$$
where $X_0$ is a seminomed space in the sense of Dubinski{\u\i}, with seminorm defined by
$$
[\vp]_{X_0}: = \|\vp\|_{L^1(\O)} + \int_{\O }  \left|\nabla_x \sqrt{\vp}\right|^2 \dx.
$$

We shall now verify that $X$, $X_0$ and $X_1$ thus defined do indeed satisfy the requirements of Lemma \ref{lem-Dubinskii}.
By the bounds in $\eqref{est-level0-f}_2$ and Sobolev embedding we obtain
\be\label{est-eta-de}
\| \eta_{\de} \log \eta_{\de}\|_{L^\infty(0,T;L^1(\O))}
+ \|  \eta_{\de} \|_{L^1(0,T; L^{\frac{1}{\delta}}(\O))}  \leq C,\quad \mbox{for any $\delta \in (0,1)$}.
\ee
From function space interpolation, we deduce that
\be \label{est-eta-int-de}
\| \eta_{\de}\|_{L^{2-\delta}(0,T;L^2(\O))} + \|  \eta_{\de} \|_{L^{2+\frac{\delta^2}{2}}(0,T;
L^{2-\delta}(\O))}  \leq C,\quad \mbox{for any $\delta \in (0,1)$}.
\ee
Together with (\ref{est-level0-f})$_{2,4}$ and equation \eqref{03a},  we have that
\be\label{psin-est1}
\mbox{$(\eta_\de)_{\de>0}$ is bounded in $L^1(0,T;X_0)$ and $(\d_t \eta_\de)_{\de>0}$ is bounded in $L^{1}(0,T;X_1)$}.
\ee

The continuity of the embedding $X\hookrightarrow X_1$ is immediate by Sobolev embedding.
We now verify the compactness of the embedding $X_0 \hookrightarrow X$. Let $(\vp_n)_{n\in N}$ be a bounded sequence in $X_0$. Thus, the sequence $\big(\sqrt{\vp_n}\big)_{n\in N}$ is bounded in $W^{1,2}(\O)$, which is compactly embedded into $L^2(\O)$. This means that
$$
\sqrt{\vp_n} \to \rho \quad \mbox{strongly in}\ L^2(\O), \quad \mbox{as }n \rightarrow \infty.
$$
Define $\vp:= \rho^2$. Then,
\begin{align*}
%\ba\label{Dub-comp}
\|\vp_n - \vp\|_{L^1(\O)} &= \| \left(\sqrt{\vp_n} + \sqrt{\vp}\right)\left(\sqrt{\vp_n} - \sqrt{\vp}\right)\|_{L^1(\O)} \\
&\leq \|\sqrt{\vp_n} + \sqrt{\vp}\|_{L^2(\O)} \|\sqrt{\vp_n} - \sqrt{\vp}\|_{L^2(\O)} \to 0,\quad \mbox{as} \ n \to \infty.
%\ea
\end{align*}
This implies that the embedding $X_0 \hookrightarrow X$ is compact.

By \eqref{psin-est1} and Dubinski{\u\i}'s compactness theorem, we obtain
$$
%\be\label{eta-lim1}
\eta_\de \to \eta \quad \mbox{strongly in} \ L^1((0,T)\times \Omega),
\quad \mbox{as }\de \rightarrow 0.
%\ee
$$
This implies that
\be\label{eta-lim20}
\eta_\de \to \eta \quad \mbox{a.e. in} \ (0,T)\times\Omega,
\quad \mbox{as }\de \rightarrow 0.
\ee

\medskip

Again by the argument in Remark \ref{rem-eta2-tau}, we have the following uniform bound on the $L^2$ norm of $\eta_\de$:
\ba\label{est-eta2-de-0}
\| \eta_\de \|_{L^2(0,T;L^{2}(\O))}   \leq C,
\ea
where $C$ is independent of $\de$. Thus, by the almost everywhere convergence in \eqref{eta-lim20} and Vitali's theorem, we
have, for any $\delta\in (0,1)$, that
\be\label{conv-eta-de}
\eta_\de \to \eta \quad \mbox{strongly in} \ L^{2-\delta}((0,T)\times \Omega),
\quad \mbox{as }\de \rightarrow 0.
\ee

By \eqref{est-eta2-de-0}, \eqref{conv-eta-de} and writing $\nabla_x \eta_{\de}$ as $2\eta_{\de}^{\frac{1}{2}}
\,\nabla_x \eta_{\de}^{\frac{1}{2}}$, we can pass to the limit $\de\to 0$ in the nonlinear terms associated with $\eta_\de$ in the weak formulations \eqref{weak-form1-f}--\eqref{weak-form4-f} to deduce that the equations are satisfied by the
limiting quadruple of functions, $(\vr,\vu,\eta,\TT)$.

The energy inequality for $(\vr,\vu,\eta,\TT)$ can be obtained by letting $\de\to 0$ in \eqref{energy1-f}, using \eqref{reg-eta-ff2}--\eqref{reg-eta-ff3}, and omitting the nonnegative terms $\de \, \eta_\de^2$ and $\de  |\nabla_x \eta_\de|^2$ on the left-hand side of \eqref{energy1-f}.

It is immediate to deduce that the solution $\eta$ satisfies $\eta \geq 0$
a.e.\ in $(0,T] \times \Omega$ and
$$
%\be\label{eta-Cw-L1-0}
\eta \log \eta \in L^\infty(0,T;L^1(\O)), %\log L^1(\O))
\ \quad \eta^{\frac{1}{2}} \in L^2(0,T;W^{1,2}(\O)), \quad \eta \in L^2(0,T;L^2(\O)), \quad  \d_t \eta \in L^1(0,T;W^{-2,2}(\O)).
%\ee
$$
Thus,  by using Lemma \ref{lem-Cw-Cw*} (ii), we have that
$$\eta \in C_w([0,T];L^1(\O)).$$
The proof of this assertion proceeds as follows. Let $\mathcal{F}(s):= s(\log s - 1) + 1$ for $s>0$, and define $\mathcal{F}(0):=1$. Clearly,
$\mathcal{F}(s) \geq 0$ for all $s \in [0,\infty)$, $\mathcal{F}(1)=0$, $\mathcal{F}$ is strictly convex with superlinear growth
as $s \rightarrow \infty$. We take $X:=L^\Phi(\Omega)$, the Orlicz space with Young's function $\Phi(s)=\mathcal{F}(1+|s|)$
(cf. Kufner, John \& Fu\v{c}\'{\i}k \cite{KJF}, Sec. 3.6) whose separable predual $E:=E^\Psi(\Omega)$ has Young's function,
$\Psi(s) = {\rm exp}\,|s| - |s| - 1$, the Fenchel conjugate of $\Phi$ (see, Section 3.12 in \cite{KJF} for the definition
of $E^\Psi(\Omega)$, Theorem 3.12.9 in \cite{KJF} for the separability of $E^\Psi(\Omega)$, and Section 3.13.8, eq. (1)
in \cite{KJF} for the duality $[E^\psi(\Omega)]' = L^\Phi(\Omega)$).
Further, we choose $Y:=W^{-2,2}(\Omega)$, whose predual $F=W^{2,2}_0(\Omega)$ is, clearly,
continuously embedded in $L^\infty(\Omega)$ by the Sobolev embedding theorem, and $L^\infty(\Omega)$ is, in turn, continuously embedded in $E=E^\Psi(\Omega)$ thanks to Theorem 3.17.7 in \cite{KJF}. It then follows from Lemma \ref{lem-Cw-Cw*} (ii) that
$\eta \in C_{w*}([0,T]; L^\Phi(\Omega))$.
However, as $L^\infty(\O)=[L^1(\O)]'$, $C_{w*}([0,T]; L^\Phi(\O))$ is contained in $C_{w}([0,T]; L^1(\O))$, whereby $\eta \in C_w([0,T];L^1(\O))$, as has been asserted.
\end{proof}

We conclude with a further result, which shows that if the initial polymer number density has stronger integrability than
$L \log L(\Omega)$, say $\eta_0 \in L^{q}(\O)$, $q>1$, then the regularity and the integrability properties of $\eta(t,\cdot)$
for $t \in (0,T]$ are also improved; the proof is based on function space interpolation and repeated application of Lemma \ref{lem-parabolic-2}.

\begin{proposition}\label{rem-eta-better}

Suppose that $L>0$ is held fixed. Assume further
that $\eta_0\in L^q(\O)$, $q>1$; then, for any $\delta \in (0,1)$ such that $\delta<q-1$, we have that
\be\label{est-eta-para2-0}
\eta \in L^{\infty}(0,T; L^q(\O)) \cap   L^{2-\delta}(0,T; W^{1,q}(\O)) \cap L^2(0,T; W^{1,q-\delta}(\O)).
\ee
\end{proposition}

\begin{proof}
By $\eqref{est-level0-f}$, we have that the limit $\vu \in L^2(0,T;W^{1,2}(\O,\R^2))$
and $\eta \in L^\infty(0,T;L^1(\O)) \cap L^2(0,T;L^2(\O))$. By Sobolev embedding $W^{1,2}(\O)\hookrightarrow L^{\frac{1}{\delta}}(\O)$, for any $\delta\in (0,1)$, we have that
\be\label{est-eta-u-de1}
 \eta  \vu \in L^{2-\frac{\delta}{2}}(0,T; L^1(\O;\R^2)) \cap  L^1(0,T; L^{2-\frac{\delta}{2}}(\O;\R^2))\hookrightarrow L^{1+c(\delta)}(0,T; L^{2-\delta}(\O;\R^2)),
\ee
for any $\delta\in (0,1)$ and some $c(\delta)>0$.

We first consider the case $1<q<2$.
Using (\ref{est-eta-u-de1}),  we can apply Lemma \ref{lem-parabolic-2} to deduce that
$$
%\be\label{est-eta-para1-1}
\eta \in L^{\infty}(0,T; L^{q}(\O)) \cap   L^{1+c(\delta)}(0,T; W^{1,q}(\O))\hookrightarrow L^{2+2c(\delta)}(0,T; L^{\frac{4q}{4-q}}(\O)),
%\ee
$$
where we have used the Sobolev embedding
$$
W^{1,q}(\O)\hookrightarrow  L^{\frac{2q}{2-q}}(\O), \quad \mbox{if $1\leq q < 2$}.
$$
Again by Sobolev embedding and function space interpolation, with $\delta \in (0,1)$ such that $\delta<q-1$, we deduce that
$$
%\be\label{est-eta-u-de2}
 \eta  \vu \in L^{2}(0,T; L^{q-\delta}(\O;\R^2)) \cap L^{1+c(\delta)}(0,T; L^{\frac{4q}{4-q}}(\O;\R^2)) , \ \mbox{for some $c(\delta)>0$}.
%\ee
$$
This implies that $\eta\vu\in   L^{2-\delta}(0,T; L^q(\O;\R^2))$, for any $\delta\in (0,1)$ such that $\delta<q-1$ and some $c(\delta)>0$.  By using Lemma \ref{lem-parabolic-2} again we arrive at \eqref{est-eta-para2-0}.

\medskip

Let us now consider the case $q\geq2$.  By \eqref{est-eta-u-de1} and Lemma \ref{lem-parabolic-2} we deduce that
$$
%\be\label{est-eta-para1-2}
\eta \in L^{\infty}(0,T; L^{2-\delta}(\O)) \cap   L^{1+c(\delta)}(0,T; W^{1,2-\delta}(\O)),\quad \mbox{for any $\delta\in (0,1)$ and some $c(\delta)>0$}.
%\ee
$$
This implies furthermore that
$$
%\be\label{est-eta-u-de2-1}
\eta \vu \in L^{2-\delta}(0,T; L^2(\O;\R^2)) \cap  L^2(0,T; L^{2-\delta}(\O;\R^2)), \quad \mbox{for any $\delta\in (0,1)$}.
%\ee
$$
By applying Lemma \ref{lem-parabolic-2} again we deduce that
\be\label{est-eta-para2}
\eta \in L^{\infty}(0,T; L^2(\O)) \cap   L^{2-\delta}(0,T; W^{1,2}(\O)) \cap  L^2(0,T; W^{1,2-\delta}(\O)),\quad \mbox{for any $\delta\in (0,1)$},
\ee
which proves \eqref{est-eta-para2-0} with $q=2$. It remains to consider the case when $q>2$. By \eqref{est-eta-para2} we have that
$$
%\be\label{est-eta-u-de3-1}
\eta \vu \in   L^{2}(0,T; L^{2-\delta}(\O;\R^2)) \cap  L^{1+c(\delta)}(0,T; L^{\frac{1}{\delta}}(\O;\R^2)), \quad \mbox{for any $\delta\in (0,1)$ and some $c(\delta)>0$}.
%\ee
$$
Together with Lemma \ref{lem-parabolic-2} we then deduce that
$$
%\be\label{est-eta-para3-1}
\eta \in L^{\infty}(0,T; L^{q}(\O)) \cap   L^{1+c(q)}(0,T; W^{1,q}(\O)),\quad \mbox{for some $c(q)>0$}.
%\ee
$$
Again by the bound on $\vu$ and Sobolev embedding we have that
$$
%\be\label{est-eta-u-de3-2}
\eta\vu \in L^{2}(0,T; L^{q-\delta}(\O;\R^2)) \cap  L^{1+c(\delta)}(0,T; L^{2q-\delta}(\O;\R^2)),\quad \mbox{for any $\delta\in (0,1)$ and some $c(\delta)>0$}.
%\ee
$$
This gives $\eta\vu\in L^{2-\delta}(0,T; L^{q}(\O;\R^2))$ for any $\delta\in (0,1)$. Hence, the application of Lemma \ref{lem-parabolic-2} implies  \eqref{est-eta-para2-0}. That completes the proof of the proposition.
\end{proof}

\section*{Acknowledgements}
\thispagestyle{empty}

The authors acknowledge the support of the project LL1202 in the programme ERC-CZ funded by the Ministry of Education, Youth and Sports of the Czech Republic. ES is grateful to members of the Ne\v{c}as Center for Mathematical Modeling and the
Faculty of Mathematics and Physics of the Charles University in Prague for their kind hospitality during his sabbatical leave from the University of Oxford.

%%%%%%%%%%%%%%%%%%%%%%%%%%%%%%%%%%%%%%%%%%%%%%%%%%%%%%%%%%%%%%%%%%%%%%%%%%%%%%%%%%%%%%%%%%%%%%%%%%%%%%%%%

%\newpage

\begin{appendix}

\section{Proof of Lemma \ref{lem-lap-tau-inverse}}\label{appendixa}

We complete our analysis of the compressible Oldroyd-B model by
providing the proofs of Lemmas \ref{lem-lap-tau-inverse},
\ref{lem-pt-mf} and \ref{lem-lap-tau-inverses3}.
Here, we give the proof of Lemma \ref{lem-lap-tau-inverse};
the proofs of Lemmas \ref{lem-pt-mf} and \ref{lem-lap-tau-inverses3} are contained in Appendix \ref{appendixb}
and Appendix \ref{appendixc}, respectively.
%
%\begin{lemma}\label{lem-lap-tau-inverse-new} Let $\TT:\O \to \R^{d\times d}$ be a smooth symmetric positive definite matrix %function satisfying the boundary condition \eqref{07}; then,
%%
%\ba\label{est-log-lap-tau-new}
%\int_\O \Delta_x \TT : \TT^{-1}\,\dx &= \sum_{j=1}^d  \int_\O \tr \left(\left((\d_{x_j} \TT)(\TT^{-1})\right)^2\right) \dx  \\
%& \geq  \sum_{j=1}^d  \int_\O \sum_{i=1}^d \left| \d_{x_j} \log \l_{i}  \right|^2\,\dx \\
% & \geq \frac{1}{d}  \int_\O \left|\nabla_x \left( \tr\,\log \TT \right)\right|^2\,\dx,
%\ea
%%
%where $\l_i(\cdot),\ i=1,\ldots,d$, are the eigenvalues of $\TT(\cdot)$.
%\end{lemma}
%
%

\begin{proof}[Proof of Lemma \ref{lem-lap-tau-inverse}]
We begin by proving the inequality stated in Lemma \ref{lem-lap-tau-inverse}.
As $$s \in [1,\infty) \mapsto g(s):=s^2 - 2s\log s -1 \in \mathbb{R}_{\geq 0}$$ is a convex function, with a unique stationary point located at $s=1$, where $g$ attains its minimum value on $[1,\infty)$, it follows that $g(s) \geq g(1)=0$ for all $s \in [1,\infty)$. Assuming that $a,b \in \mathbb{R}_{>0}$, and rearranging the expression $g(s_*) \geq 0$, where $$s_*:=\sqrt{\frac{\max\{a,b\}}{\min\{a,b\}}},$$ we deduce that
\begin{align}\label{eq-ab}
- (a-b)\left(\frac{1}{a}- \frac{1}{b}\right) \geq (\log a  - \log b)^2\qquad \forall\, a,b \in \mathbb{R}_{>0}.
\end{align}

In order to extend this inequality to symmetric positive definite matrices,
we adapt an argument from \cite{Barrett-Boyaval}.
Suppose that $\AAA, \BB \in \mathbb{R}^{d \times d}$ are symmetric positive definite matrices, with respective diagonalizations
\[ \AAA = \OO_{\AAA} \DD_{\AAA} \OO^{\rm T}_{\AAA}\quad \mbox{and}\quad \BB = \OO_{\BB} \DD_{\BB} \OO^{\rm T}_{\BB},\]
where $\DD_{\AAA}, \DD_{\BB} \in \mathbb{R}^{d\times d}$ are diagonal, with positive diagonal entries, and $\OO_{\AAA}$ and $\OO_{\BB}$ are orthogonal. By defining the matrix $\CC:= \log \AAA - \log \BB$, applying the Cauchy--Schwarz inequality, and noting that $\CC = \CC^{\rm T}$, we have that
\begin{align}\label{eq-AAA}
|\tr(\log \AAA) - \tr(\log \BB)|^2 &= |\tr(\CC)|^2
\leq d \sum_{i=1}^d (\CC_{ii})^2 \leq d \sum_{i,k=1}^d (\CC_{ik})^2
= d \sum_{i,k=1}^d \CC_{ik} \CC_{ki} = d \sum_{i=1}^d (\CC^2)_{ii}
= d\, \tr(\CC^2) \nonumber\\
&= d\, \tr((\log \AAA - \log \BB)^2) = d\, \tr((\log \AAA - \log \BB)(\log \AAA - \log \BB))\nonumber\\
&= d\, \tr((\OO_{\AAA} (\log \DD_{\AAA}) \OO^{\rm T}_{\AAA} - \OO_{\BB} (\log\DD_{\BB}) \OO^{\rm T}_{\BB})(\OO_{\AAA} (\log\DD_{\AAA}) \OO^{\rm T}_{\AAA} - \OO_{\BB} (\log\DD_{\BB}) \OO^{\rm T}_{\BB})).
\end{align}
Since $\OO_{\AAA}$ is orthogonal and the trace of a product of matrices is invariant under cyclic permutations of the factors
appearing in the product, we have, with $\OO:=\OO^{\rm T}_{\AAA}\OO_{\BB}$, that
\begin{align*}
&\tr((\OO_{\AAA} (\log \DD_{\AAA}) \OO^{\rm T}_{\AAA} - \OO_{\BB} (\log\DD_{\BB}) \OO^{\rm T}_{\BB})(\OO_{\AAA} (\log\DD_{\AAA}) \OO^{\rm T}_{\AAA})) \\
&\quad = \tr(\OO_{\AAA} (\log \DD_{\AAA})^2\, \OO^{\rm T}_{\AAA}) - \tr(\OO^{\rm T}_{\AAA}\OO_{\BB} (\log\DD_{\BB}) (\OO^{\rm T}_{\AAA}\OO_{\BB})^{\rm T} (\log\DD_{\AAA}))\\
%&\quad = \tr(\OO_{\AAA} (\log \DD_{\AAA})^2\, \OO^{\rm T}_{\AAA}) - \tr(\OO (\log\DD_{\BB}) \OO^{\rm T} (\log\DD_{\AAA}))\\
&\quad = \tr((\log \DD_{\AAA})^2) - \tr(\OO (\log\DD_{\BB}) \OO^{\rm T} (\log\DD_{\AAA}))\\
&\quad = \tr((\log \DD_{\AAA})^2)  - \sum_{i,j=1}^d (\OO_{ij})(\log\DD_{\BB})_{jj} (\OO^{\rm T})_{ji} (\log\DD_{\AAA})_{ii} \\
&\quad = \sum_{i,j=1}^d (\OO_{ij})^2(\log\DD_{\AAA})_{ii} (\log\DD_{\AAA})_{ii} - \sum_{i,j=1}^d (\OO_{ij})^2(\log\DD_{\BB})_{jj} (\log\DD_{\AAA})_{ii},
\end{align*}
where in the transition to the last line we have used that $\sum_{j=1}^d (\OO_{ij})^2 = 1$ for all $i \in \{1,\dots, d\}$, which is a direct consequence of the fact that $\OO \OO^{\rm T} = \II$, because $\OO_\AAA$ and $\OO_\BB$ are orthogonal matrices. Thus we have shown that
\begin{align}\label{eq-BBB}
&\tr((\OO_{\AAA} (\log \DD_{\AAA}) \OO^{\rm T}_{\AAA} - \OO_{\BB} (\log\DD_{\BB}) \OO^{\rm T}_{\BB})(\OO_{\AAA} (\log\DD_{\AAA}) \OO^{\rm T}_{\AAA})) = \sum_{i,j=1}^d (\OO_{ij})^2[(\log\DD_{\AAA})_{ii} - (\log\DD_{\BB})_{jj}] (\log\DD_{\AAA})_{ii}.
\end{align}
By swapping $\AAA$ and $\BB$ in this identity, and noting that the matrix $\OO_\BB^{\rm T} \OO_{\AAA}$, resulting
from swapping $\AAA$ and $\BB$ in the definition of $\OO=\OO_\AAA^{\rm T} \OO_\BB$, is equal to the transpose of $\OO$, we have that
\begin{align*}
&\tr((\OO_{\BB} (\log\DD_{\BB}) \OO^{\rm T}_{\BB}- \OO_{\AAA}(\log \DD_{\AAA}) \OO^{\rm T}_{\AAA})(\OO_{\BB} (\log\DD_{\BB}) \OO^{\rm T}_{\BB})) = \sum_{i,j=1}^d (\OO_{ji})^2[(\log\DD_{\BB})_{ii} - (\log\DD_{\AAA})_{jj}] (\log\DD_{\BB})_{ii}.
\end{align*}
After %extracting a factor of $(-1)$ from both sides of this equality, and
renaming $i$ into $j$ and $j$ into $i$ under the double summation sign appearing on the right-hand side, we have that
\begin{align}\label{eq-CCC}
&- \tr((\OO_{\AAA} (\log\DD_{\AAA}) \OO^{\rm T}_{\AAA}- \OO_{\BB}(\log \DD_{\BB}) \OO^{\rm T}_{\BB})(\OO_{\BB} (\log\DD_{\BB}) \OO^{\rm T}_{\BB})) = - \sum_{i,j=1}^d (\OO_{ij})^2[(\log\DD_{\AAA})_{ii} - (\log\DD_{\BB})_{jj}] (\log\DD_{\BB})_{jj}.
\end{align}
By summing \eqref{eq-BBB} and \eqref{eq-CCC} and recalling the inequality \eqref{eq-ab}, we deduce that
\begin{align}\label{eq-DDD}
&\tr((\OO_{\AAA} (\log \DD_{\AAA}) \OO^{\rm T}_{\AAA} - \OO_{\BB} (\log\DD_{\BB}) \OO^{\rm T}_{\BB})(\OO_{\AAA} (\log\DD_{\AAA}) \OO^{\rm T}_{\AAA}- \OO_{\BB} (\log\DD_{\BB}) \OO^{\rm T}_{\BB}))\nonumber\\
&\qquad= \sum_{i,j=1}^d (\OO_{ij})^2[(\log\DD_{\AAA})_{ii} - (\log\DD_{\BB})_{jj}] [(\log\DD_{\AAA})_{ii} -
(\log\DD_{\BB})_{jj}]\nonumber\\
&\qquad = \sum_{i,j=1}^d (\OO_{ij})^2[(\log\DD_{\AAA})_{ii} - (\log\DD_{\BB})_{jj}]^2\nonumber\\
&\qquad \leq - \sum_{i,j=1}^d (\OO_{ij})^2\left[\left((\DD_{\AAA})_{ii} - (\DD_{\BB})_{jj}\right)\left(\frac{1}{(\DD_{\AAA})_{ii}} - \frac{1}{(\DD_{\BB})_{jj}}\right)\right] \nonumber\\
&\qquad= - \sum_{i,j=1}^d (\OO_{ij})^2\left[\left((\DD_{\AAA})_{ii} - (\DD_{\BB})_{jj}\right)\left((\DD_{\AAA}^{-1})_{ii} - (\DD_{\BB}^{-1})_{jj}\right)\right].
\end{align}
Now, by an analogous calculation to the one that led to the first equality in \eqref{eq-DDD} above, we have that
\begin{align}\label{eq-EEE}
\tr\left(\left( \OO_{\AAA} \DD_{\AAA} \OO^{\rm T}_{\AAA} -  \OO_{\BB} \DD_{\BB} \OO^{\rm T}_{\BB}
\right)\left( \OO_{\AAA} (\DD_{\AAA})^{-1} \OO^{\rm T}_{\AAA} -  \OO_{\BB} (\DD_{\BB})^{-1} \OO^{\rm T}_{\BB}
\right)\right) =  \sum_{i,j=1}^d (\OO_{ij})^2\left[\left((\DD_{\AAA})_{ii} - (\DD_{\BB})_{jj}\right)\left((\DD_{\AAA}^{-1})_{ii} - (\DD_{\BB}^{-1})_{jj}\right)\right].
\end{align}
By comparing the right-hand sides of \eqref{eq-DDD} and \eqref{eq-EEE}, we deduce that
\begin{align*}
\tr((\OO_{\AAA} (\log \DD_{\AAA}) \OO^{\rm T}_{\AAA} - \OO_{\BB} (\log\DD_{\BB}) \OO^{\rm T}_{\BB})(\OO_{\AAA} (\log\DD_{\AAA}) \OO^{\rm T}_{\AAA}- \OO_{\BB} (\log\DD_{\BB}) \OO^{\rm T}_{\BB}))\\
\leq -\tr\left(\left( \OO_{\AAA} \DD_{\AAA} \OO^{\rm T}_{\AAA} -  \OO_{\BB} \DD_{\BB} \OO^{\rm T}_{\BB}
\right)\left( \OO_{\AAA} (\DD_{\AAA})^{-1} \OO^{\rm T}_{\AAA} -  \OO_{\BB} (\DD_{\BB})^{-1} \OO^{\rm T}_{\BB}
\right)\right);
\end{align*}
equivalently,
\begin{align}\label{eq-FFF}
\tr((\log \AAA - \log\BB)(\log \AAA - \log \BB)) \leq -\tr\left(\left(\AAA -\BB \right)\left({\AAA}^{-1} - {\BB}^{-1} \right)\right).
\end{align}
Substitution of \eqref{eq-FFF} into the penultimate line of \eqref{eq-AAA} then implies that,
for any two symmetric positive definite matrices $\AAA$, $\BB \in \mathbb{R}^{d\times d}$, the following inequality holds:
\begin{align}\label{eq-AiBi}
&|\tr(\log \AAA) - \tr(\log \BB)|^2 \leq - d\, \tr\left(\left(\AAA -\BB \right)\left({\AAA}^{-1} - {\BB}^{-1} \right)\right).
\end{align}

Let $e_j$ denote the unit vector pointing in the positive $Ox_j$ direction, $j=1,\dots, d$. Then, for each $x \in \Omega$ and
each $j \in \{1,\dots, d\}$, there exists a bounded closed interval $I_{x,j} \subset \mathbb{R}$, with $0$ contained in the interior of $I_{x,j}$, such that $x + he_j \in \overline{\Omega}$ for all $h \in I_{x,j}$. As, by hypothesis, $\PP$ is symmetric and positive definite, uniformly on $\overline\Omega$, there exists a $c_0 \in \mathbb{R}_{>0}$ such that $\PP(x) \geq c_0 \II$ for all $x \in \overline\Omega$. Thus, $\AAA = \PP(x + h e_j)$ and $\BB = \PP(x)$ are legitimate
choices in \eqref{eq-AiBi} for all $h \in I_{x,j}$ and $j \in \{1,\dots,d\}$. Dividing the resulting inequality by $h^2d$, and passing to the limit $h \rightarrow 0$, thanks to the assumed regularity $\PP \in C^1(\overline\Omega;\mathbb{R}^{d \times d})$,
we deduce that
\begin{align}\label{eq-AjBj}
&\frac{1}{d}\,|\partial_{x_j}\tr(\log \PP(x))|^2 \leq - \tr\left((\partial_{x_j}\PP(x))\,(\partial_{x_j}({\PP}^{-1}(x)))\right)
\qquad \forall\, x \in \Omega, \quad \forall\,j \in \{1,\dots,d\}.
\end{align}
Here, to obtain the expression on the left-hand side of the last inequality,
we have made use of the fact that by %made use of a result of Rellich (cf. Theorem, p.44 in Ch.1, {\S}5 of \cite{Rellich}), according %to which symmetric $C^1$ matrices have $C^1$ eigenvalues, which implies that $x \in \Omega \mapsto \tr(\log\PP(x)) %\in \mathbb{R}$ is a $C^1$ function, because $\tr(\log\PP)$ is equal to the sum of the logarithms of %the eigenvalues of $\PP$,
%each of which is a $C^1$ function. Alternatively,
Jacobi's identity, $\tr(\log \PP(x)) = \log \det \PP(x)$, and $x \in \Omega \mapsto \log \det \PP(x)\in \mathbb{R}$ is a $C^1$ function, whereby
the same is true of $x \in \Omega \mapsto \tr(\log\PP(x)) \in \mathbb{R}$.

As $\PP \PP^{-1} = \II$, it follows from the product rule that $\partial_{x_j}(\PP^{-1}) = - \PP^{-1} (\partial_{x_j} \PP) \PP^{-1}$, and therefore \eqref{eq-AjBj} yields
\begin{align*}
&\frac{1}{d}\,|\partial_{x_j}\tr(\log \PP(x))|^2 \leq \tr\left(((\partial_{x_j}\PP(x))\PP^{-1}(x))^2\right)\qquad \forall\, x \in \Omega, \quad \forall\,j \in \{1,\dots,d\}.
\end{align*}
Because $\PP \in C^1(\overline\Omega;\mathbb{R}^{d\times d})$, the expression $x \in \Omega \mapsto \tr\left(((\partial_{x_j}\PP(x))\PP^{-1}(x))^2\right)$ appearing on the right-hand side of this inequality is a
bounded continuous function on $\Omega$.
%because a $C^0$ matrix has $C^0$ eigenvalues (cf. Theorem 5.1 in Ch.2, {\S}5 of \cite{Kato})
%and the trace of a matrix is equal to the sum of its eigenvalues.
%\footnote{Why do we need to discuss eigenvalues for this RHS term ?}
Also, thanks to the discussion in the previous paragraph, the expression appearing on the left hand side of this inequality is a continuous (and therefore, thanks to the upper bound furnished by the inequality, a bounded continuous) function on $\Omega$. By integrating the inequality over $\Omega$ and summing over $j=1,\dots, d$ we thereby deduce that
\begin{align*}
&\frac{1}{d}\int_\Omega |\nabla_x\tr(\log \PP)|^2\, \dx \leq \sum_{j=1}^d \int_\Omega \tr\left(\left((\partial_{x_j}\PP)\PP^{-1}\right)^2\right) \dx,
\end{align*}
thus completing the proof of the inequality stated in the lemma.

It remains to prove the equality stated in Lemma \ref{lem-lap-tau-inverse}. By partial integration (cf. Corollary 2.6 in Ch.1 of \cite{GR}), recalling that, by hypothesis, the symmetric and uniformly positive definite matrix function $\PP
\in W^{2,2}(\Omega;\mathbb{R}^{d \times d})\cap C^1(\overline\Omega;\mathbb{R}^{d \times d})$ satisfies a homogeneous Neumann boundary condition on $\partial\Omega$, and noting that $\PP^{-1}
\in C^1(\overline\Omega;\mathbb{R}^{d \times d}) \subset W^{1,2}(\Omega;\mathbb{R}^{d \times d})$, we have that
\begin{align*}
\int_\O \Delta_x \PP : \PP^{-1}\,\dx &= - \sum_{j=1}^d \int_\Omega \partial_{x_j} \PP : \partial_{x_j}(\PP^{-1})\,\dx\\
&= - \sum_{j=1}^d \int_\Omega \tr \left((\partial_{x_j} \PP) (\partial_{x_j}(\PP^{-1}))\right)\dx
= \sum_{j=1}^d  \int_\O \tr \left(\left((\d_{x_j} \PP)(\PP^{-1})\right)^2\right) \dx,
\end{align*}
where we have, once again, made use of the identity $\partial_{x_j}(\PP^{-1}) = - \PP^{-1} (\partial_{x_j} \PP) \PP^{-1}$.
%
%That completes the proof of Lemma \ref{lem-lap-tau-inverse}.
\end{proof}

\section{Proof of Lemma \ref{lem-pt-mf}}\label{appendixb}

\begin{proof}[Proof of Lemma \ref{lem-pt-mf}]
According to (2.15) in \cite{Barrett-Boyaval}, for any concave function $g \in C^1(\mathbb{R})$, and any pair
of symmetric matrices $\AAA, \BB \in \mathbb{R}^{d \times d}$, one has that
\begin{align}\label{eq:eq2-1}
(\AAA - \BB):g'(\BB) \geq \tr(g(\AAA)-g(\BB)) \geq (\AAA - \BB):g'(\AAA).
\end{align}
For a convex function $g \in C^1(\mathbb{R})$, the inequalities in \eqref{eq:eq2-1} are reversed, yielding
\begin{align}\label{eq:eq2-2}
(\AAA - \BB):g'(\BB) \leq \tr(g(\AAA)-g(\BB)) \leq (\AAA - \BB):g'(\AAA)
\end{align}
for any pair of symmetric matrices $\AAA, \BB \in \mathbb{R}^{d \times d}$.

Let us suppose that $g \in C^{1,\gamma}(\mathbb{R})$, with $0<\gamma \leq 1$,
is concave. As the univariate symmetric matrix function  $\PP \in W^{1,2}((0,T);\mathbb{R}^{d\times d})$ is absolutely continuous on $[0,T]$,
it is differentiable a.e. on $(0,T)$. Let $t_* \in (0,T)$ be such that $\PP$ is differentiable at $t_*$. Hence, by choosing $\AAA=\PP(t_*+h)$ and $\BB=\PP(t_*)$ in \eqref{eq:eq2-1}, where $0<|h|<\min(t_*,T-t_*)$, we have that
\begin{align}\label{eq:eq2-2a}
\frac{\PP(t_*+h) - \PP(t_*)}{h}:g'(\PP(t_*)) \geq \frac{\tr(g(\PP(t_*+h))-g(\PP(t_*)))}{h}
\geq \frac{\PP(t_*+h) - \PP(t_*)}{h}:g'(\PP(t_*+h)),
\end{align}
for $h>0$, and with the $\geq$ signs replaced by $\leq$ for $h<0$.

Now, the function $t \in [0,T] \mapsto g'(\PP(t))$ is continuous on $[0,T]$. Indeed, as $g' \in C^{0,\gamma}(\mathbb{R})$,
the matrix function $\QQ \mapsto g'(\QQ)$, defined on the space of symmetric matrices $\QQ \in \mathbb{R}^{d\times d}$,
is also H\"older continuous, with the same H\"older exponent $\gamma$ (cf. Theorem 1.1 in \cite{Wihler}), and
\[ |g'(\PP(t_*+h)) - g'(\PP(t_*))|\leq \|g'\|_{C^{0,\gamma}(\mathbb{R})} d^{\frac{1-\gamma}{2}} |\PP(t_*+h) - \PP(t_*)|^\gamma. \]
Thanks to the (absolute) continuity of $t \in [0,T] \mapsto \PP(t) \in \mathbb{R}^{d \times d}$, the right-hand side of this inequality converges to $0$ as $h \rightarrow 0$; therefore the same is true of the left-hand side of the inequality. Hence,
\[ \lim_{h \rightarrow 0} g'(\PP(t_*+h)) = g'(\PP(t_*)).\]
Since $\PP$ is differentiable at $t_* \in (0,T)$, we can now pass to the limit $h \rightarrow 0_+$ in \eqref{eq:eq2-2a} to deduce that
\begin{align*}
\partial_t\PP(t_*):g'(\PP(t_*)) \geq
\lim_{h \rightarrow 0_+}\frac{\tr(g(\PP(t_*+h))-g(\PP(t_*)))}{h} \geq
\partial_t\PP(t_*): g'(\PP(t_*)),
\end{align*}
for $h>0$, with the $\geq$ signs replaced by $\leq$ and $\lim_{h \rightarrow 0_+}$ replaced by $\lim_{h \rightarrow 0_-}$ for $h<0$.

Hence, and thanks to the linearity of the trace operator $\tr$,
\begin{align*}
\lim_{h \rightarrow 0}\frac{\tr(g(\PP(t_*+h)))-\tr(g(\PP(t_*)))}{h} =
\partial_t\PP(t_*): g'(\PP(t_*)).
\end{align*}
Consequently, for a concave function $g \in C^{1,\gamma}(\mathbb{R})$, $0<\gamma \leq 1$,
\[ \d_t \tr\left(g(\PP(t_*))\right) =  g'(\PP(t_*)): (\d_t \PP(t_*)) = \tr\left(  g'(\PP(t_*)) \,(\d_t \PP(t_*)) \right)\]
at each point $t_* \in (0,T)$ at which $\PP$ is differentiable. In the case when $g \in C^{1,\gamma}(\mathbb{R})$, $0<\gamma \leq 1$, is a convex function the same pair of equalities is arrived at by an analogous argument, but now starting from \eqref{eq:eq2-2}.

Thus we have shown that, for any symmetric matrix $\PP \in W^{1,2}(0,T; \mathbb{R}^{d\times d})$,
\be\label{pt-mf}
\d_t \tr\left(g(\PP(t))\right) =  g'(\PP(t)): (\d_t \PP(t)) = \tr\left(  g'(\PP(t)) \,(\d_t \PP(t)) \right)
\qquad \mbox{for a.e. $t \in (0,T)$},
\ee
under the assumption that $g \in C^{1,\gamma}(\mathbb{R})$, with $0<\gamma \leq 1$, is concave or convex.
\end{proof}

\section{Proof of Lemma \ref{lem-lap-tau-inverses3}}\label{appendixc}

\begin{proof}[Proof of Lemma \ref{lem-lap-tau-inverses3}]
By hypothesis,  the symmetric matrix function $\PP \in C([0,T_{\sigma_3}];W^{1,2}(\O;\R^{d\times d}))$
 and $\Delta_x \PP \in L^2(0,T_{\sigma_3};L^2(\O;\R^{d\times d}))$. We shall first show that this implies that $\chi_{\sigma_3}(\PP)^{-1} \in L^\infty(0,T_{\sigma_3}; W^{1,2}(\Omega;\mathbb{R}^{d \times d}))$.
To this end, we first note that, as $\PP \in C([0,T_{\sigma_3}];W^{1,2}(\Omega;\mathbb{R}^{d\times d}))$, it follows that, for any $\Omega' \Subset \Omega$ and any $h\neq 0$ such that
$\mbox{dist}(\Omega',\partial\Omega)>|h|$, we have a bounded difference quotient
\[ \int_{\Omega'} \left|\frac{\PP(t,x+he_j) - \PP(t,x)}{h}\right|^2 \dx \leq \int_\Omega |\partial_{x_j} \PP(t,x)|^2 \dx,
\quad j=1,\dots,d, \quad \forall\,t \in (0,T_{\sigma_3}].\]
On the other hand, as $s \in \mathbb{R} \mapsto \chi_{\sigma_3}(s) \in \mathbb{R}_{>0}$ is globally Lipschitz continuous with Lipschitz constant equal to 1, by Theorem 1.1 in \cite{Wihler}, $\PP\in \mathbb{R}^{d\times d} \mapsto \chi_{\sigma_3}(\PP) \in \mathbb{R}^{d\times d}$ is globally Lipschitz with Lipschitz constant equal to 1; hence,
\[ |\chi_{\sigma_3}(\PP(t,x+he_j)) - \chi_{\sigma_3}(\PP(t,x))| \leq |\PP(t,x+h e_j) - \PP(t,x)|,\]
which then implies that
\[ \int_{\Omega'} \left|\frac{\chi_{\sigma_3}(\PP(t,x+he_j)) - \chi_{\sigma_3}(\PP(t,x))}{h}\right|^2 \dx \leq \int_\Omega |\partial_{x_j} \PP(t,x)|^2\, \dx,
\quad j=1,\dots,d,\quad \forall\,t \in (0,T_{\sigma_3}],\]
for any $\Omega' \Subset \Omega$ and any $h \neq 0$ such that $\mbox{dist}(\Omega',\partial\Omega)>|h|$.
Hence $x \in \Omega \mapsto \chi_{\sigma_3}(\PP(t,x)) \in \mathbb{R}^{d \times d}$ is weakly differentiable on $\Omega$
for all $t \in (0,T_{\sigma_3}]$, with $\partial_{x_j} \chi_{\sigma_3}(\PP(t,\cdot)) \in L^2(\Omega;\mathbb{R}^{d\times d})$ for all $j \in \{1,\dots, d\}$ and all $t \in (0,T_{\sigma_3}]$; furthermore,
\begin{align}\label{eq-ap3-1}
 \mbox{ess.sup}_{t \in (0,T_{\sigma_3}]}\int_\Omega |\partial_{x_j}(\chi_{\sigma_3}(\PP(t,x)))|^2 \dx \leq
\mbox{ess.sup}_{t \in (0,T_{\sigma_3}]}\int_\Omega |\partial_{x_j}\PP(t,x)|^2\, \dx.
\end{align}

Now,
\[ \mbox{ess.sup}_{t \in (0,T_{\sigma_3}]} \int_\Omega |\chi_{\sigma_3}(\PP(t,x))^{-1}|^2\, \dx =
\mbox{ess.sup}_{t \in (0,T_{\sigma_3}]} \int_\Omega |G'_{\sigma_3}(\PP(t,x))|^2\, \dx \leq \frac{|\Omega|d}{\sigma_3^2},\]
and, similarly, for any $j \in \{1,\dots,d\}$,
\begin{align*}
\mbox{ess.sup}_{t \in (0,T_{\sigma_3}]}\int_\Omega |\partial_{x_j}(\chi_{\sigma_3}(\PP(t,x))^{-1})|^2\, \dx&
= \mbox{ess.sup}_{t \in (0,T_{\sigma_3}]} \int_\Omega |(\chi_{\sigma_3}(\PP(t,x))^{-1}) (\partial_{x_j}(\chi_{\sigma_3}(\PP(t,x)))) (\chi_{\sigma_3}(\PP(t,x))^{-1})|^2\, \dx \nonumber\\
&\leq \mbox{ess.sup}_{t \in (0,T_{\sigma_3}]} \int_\Omega |\chi_{\sigma_3}(\PP(t,x))^{-1}|^2\,
|\partial_{x_j}(\chi_{\sigma_3}(\PP(t,x)))|^2\,
|\chi_{\sigma_3}(\PP(t,x))^{-1}|^2\, \dx
\nonumber\\
&\leq \frac{d^2}{\sigma_3^4} \mbox{ess.sup}_{t \in (0,T_{\sigma_3}]} \int_\Omega |\partial_{x_j}(\chi_{\sigma_3}(\PP(t,x)))|^2\, \dx,
\end{align*}
whereby, thanks to \eqref{eq-ap3-1},
\begin{align*}
\mbox{ess.sup}_{t \in (0,T_{\sigma_3}]}\int_\Omega |\partial_{x_j}(\chi_{\sigma_3}(\PP(t,x))^{-1})|^2\, \dx
\leq \frac{d^2}{\sigma_3^4}\, \mbox{ess.sup}_{t \in (0,T_{\sigma_3}]}\int_\Omega |\partial_{x_j}\PP(t,x)|^2\, \dx.
\end{align*}
Thus we have shown that $\chi_{\sigma_3}(\PP)^{-1} \in L^\infty(0,T_{\sigma_3}; W^{1,2}(\Omega;\mathbb{R}^{d\times d}))$, as has been asserted above.

As $\PP \in C([0,T_{\sigma_3}];W^{1,2}(\O;\R^{d\times d}))$, with $\Delta_x \PP \in L^2(0,T_{\sigma_3};L^2(\O;\R^{d\times d}))$, satisfies a homogeneous Neumann boundary condition on $\partial \Omega$, and, as was shown above, $\chi_{\sigma_3}(\PP)^{-1} \in L^\infty(0,T_{\sigma_3}; W^{1,2}(\Omega;\mathbb{R}^{d \times d}))$, we can apply
Corollary 2.6 in Ch.1 of \cite{GR} to integrate by parts:
\begin{align*}
\int_\O \Delta_x \PP : \chi_{\sigma_3}(\PP)^{-1}\,\dx &= - \sum_{j=1}^d \int_\Omega \partial_{x_j} \PP : \partial_{x_j}(\chi_{\sigma_3}(\PP)^{-1})\,\dx  = - %\sum_{j=1}^d
\int_\Omega \nabla_{x} \PP :: \nabla_{x}(\chi_{\sigma_3}(\PP)^{-1})\,\dx, \qquad \mbox{a.e. on $(0,T_{\sigma_3}]$,}
\end{align*}
thus proving the equality stated in Lemma \ref{lem-lap-tau-inverses3}.

\bigskip

Next we will show that, for any $j\in \{1,\ldots, d\}$,  the following inequality holds:
\ba
-\int_\O \d_{x_j} \PP(t,x) : \d_{x_j} (\chi_{\s_3}(\PP(t,x))^{-1})\,\dx
& \geq
\frac{1}{d}  \int_\O \left|\d_{x_j} \tr\left(\log
\chi_{\s_3}(\PP(t,x)) \right)\right|^2\,\dx,\qquad \mbox{for a.e. $t \in (0,T_{\sigma_3}]$}.
\label{C1}
\ea
As $\chi_{\sigma_3}$ is nondecreasing and globally Lipschitz continuous with Lipschitz constant equal to 1, it follows that
\[ -(a-b) \left(\frac{1}{\chi_{\sigma_3}(a)} - \frac{1}{\chi_{\sigma_3}(b)}\right)
    \geq - (\chi_{\sigma_3}(a) - \chi_{\sigma_3}(b))\left(\frac{1}{\chi_{\sigma_3}(a)} - \frac{1}{\chi_{\sigma_3}(b)}\right)
    \qquad \forall\, a,b \in \mathbb{R},\quad \forall \sigma_3 \in \mathbb{R}_{>0}.\]
In conjunction with \eqref{eq-ab}, with $a$ and $b$ in \eqref{eq-ab} replaced by $\chi_{\sigma_3}(a)$ and $\chi_{\sigma_3}(b)$,
respectively, this then yields that
\begin{align}\label{eq:ap3-2}
 -(a-b) \left(\frac{1}{\chi_{\sigma_3}(a)} - \frac{1}{\chi_{\sigma_3}(b)}\right)  \geq
(\log \chi_{\sigma_3}(a) - \log \chi_{\sigma_3(b)})^2\qquad \forall\, a,b \in
 \mathbb{R},\quad \forall \sigma_3 \in \mathbb{R}_{>0}.
\end{align}
By an identical argument to the one above that resulted in \eqref{eq-AiBi},
we then have, for all symmetric $\AAA,\, \BB \in \R^{d\times d}$, that
\begin{align}\label{eq-ap3-3}
&|\tr(\log \chi_{\sigma_3}(\AAA)) - \tr(\log \chi_{\sigma_3}(\BB))|^2 \leq - d\,
\tr\left(\left(\AAA -\BB \right)\left(\chi_{\sigma_3}({\AAA})^{-1} - \chi_{\sigma_3}({\BB})^{-1} \right)\right).
\end{align}
Let, again, $e_j$ denote the unit vector pointing in the positive $Ox_j$ direction, $j=1,\dots, d$. Thus, for any $\Omega'\Subset\Omega$ and any $h \neq 0$ such that $0<|h|<\mbox{dist}(\Omega',\partial\Omega)$,
and any $j\in \{1,\dots, d\}$, we deduce from
\eqref{eq-ap3-3} that
\begin{align}\label{eq-ap3-4}
&\frac{1}{d}\,\int_{\Omega'}
\left|\frac{\tr(\log \chi_{\sigma_3}(\PP(t,x + h e_j))) - \tr(\log \chi_{\sigma_3}(\PP(t,x)))}{h}\right|^2 \dx\nonumber\\
&\qquad \leq -
\int_{\Omega'}\tr\left(\frac{\PP(t,x + h e_j) -\PP(t,x)}{h}\;\frac{\chi_{\sigma_3}(\PP(t,x + h e_j))^{-1} - \chi_{\sigma_3}(\PP(t,x))^{-1}}{h} \right)\dx\qquad \forall\,t \in (0,T_{\sigma_3}].
\end{align}
By the Cauchy--Schwarz inequality, the right-hand side of \eqref{eq-ap3-4} can be bounded, for each $t \in (0,T_{\sigma_3}]$,
as follows:
\begin{align*}
0 &\leq -
\int_{\Omega'}\tr\left(\frac{\PP(t,x + h e_j) -\PP(t,x)}{h}\;\frac{\chi_{\sigma_3}(\PP(t,x + he_j))^{-1} - \chi_{\sigma_3}(\PP(t,x))^{-1}}{h} \right)\dx\\
 &\leq \mbox{ess.sup}_{t \in (0,T_{\sigma_3}]} \left\|\frac{\PP(t,\cdot + h e_j) -\PP(t,\cdot)}{h}\right\|_{L^2(\Omega')}
\mbox{ess.sup}_{t \in (0,T_{\sigma_3}]}\left\| \frac{\chi_{\sigma_3}(\PP(t,\cdot + he_j))^{-1} - \chi_{\sigma_3}(\PP(t,\cdot))^{-1}}{h}\right\|_{L^2(\Omega')} \\
 &\leq \mbox{ess.sup}_{t \in (0,T_{\sigma_3}]}\|\partial_{x_j} \PP(t,\cdot)\|_{L^2(\Omega)}\,\,
\mbox{ess.sup}_{t \in (0,T_{\sigma_3}]}\|\partial_{x_j}(\chi_{\sigma_3}(\PP(t,\cdot))^{-1})\|_{L^2(\Omega)}.
\end{align*}
Thus we deduce from \eqref{eq-ap3-4} that $\partial_{x_j}\tr(\log \chi_{\sigma_3}(\PP(t,\cdot)) \in L^2(\Omega;\mathbb{R}^{d\times d})$, for all $j \in \{1, \dots, d\}$ and all $t \in (0,T_{\sigma_3}]$, and by letting $h \rightarrow 0$,
\begin{align}\label{eq-ap3-5}
&\frac{1}{d}\,
\|\partial_{x_j}\tr(\log \chi_{\sigma_3}(\PP(t,\cdot)))\|^2_{L^2(\Omega)}
=\frac{1}{d}\,\lim_{h \rightarrow 0} \int_{\Omega_{h,j}}
\left|\frac{\tr(\log \chi_{\sigma_3}(\PP(t,x + he_j))) - \tr(\log \chi_{\sigma_3}(\PP(t,x)))}{h}\right|^2 \dx\nonumber\\
&\qquad \leq -\lim_{h \rightarrow 0}
\int_{\Omega_{h,j}}\tr\left(\frac{\PP(x + he_j) -\PP(x)}{h}\;\frac{\chi_{\sigma_3}(\PP(x + he_j))^{-1} - \chi_{\sigma_3}(\PP(x))^{-1}}{h} \right)\dx  \qquad \forall\,t \in (0,T_{\sigma_3}],
\end{align}
where $\Omega_{h,j}:= \{x \in \Omega\,:\,x+he_j \in \Omega\}$, $j=1,\dots, d$. The limit of the sequence of integrals
appearing in the second line of \eqref{eq-ap3-5} exists, possibly upon extraction of a subsequence, thanks to the boundedness
of the sequence. Hence, for all
$t \in (0,T_{\sigma_3}]$, we have
\begin{align}\label{eq-ap3-6}
&\frac{1}{d}\,
\|\partial_{x_j}\tr(\log \chi_{\sigma_3}(\PP(t,\cdot)))\|^2_{L^2(\Omega)} \nonumber\\
%=\lim_{h \rightarrow 0} \int_{\Omega_{h,j}}
%\left|\frac{\tr(\log \chi_{\sigma_3}(\PP(t,x + he_j))) - \tr(\log \chi_{\sigma_3}(\PP(t,x)))}{h}\right|^2 \dx\nonumber\\
&\qquad \leq -\lim_{h \rightarrow 0}
\int_{\Omega} \chi_{\Omega_{h,j}}(x)\,\frac{\PP(t,x + he_j) -\PP(x)}{h}\,:\,\frac{\chi_{\sigma_3}(\PP(t,x + he_j))^{-1} - \chi_{\sigma_3}(\PP(t,x))^{-1}}{h}\, \dx,\qquad j=1,\dots,d.
\end{align}
Here $\chi_{\Omega_{h,j}}$ denotes the characteristic function of the set
$\Omega_{h,j}$ (not to be confused with the cut-off function $\chi_{\s_3}$).
As $\PP \in C([0,T_{\sigma_3}];W^{1,2}(\Omega;\mathbb{R}^{d \times d}))$ and $\chi_{\sigma_3}(\PP)^{-1} \in L^\infty(0,T_{\sigma_3};W^{1,2}(\Omega;\mathbb{R}^{d\times d}))$, it follows that, as $h \rightarrow 0$,
\begin{alignat*}{2}
\frac{\PP(t,\cdot + he_j) -\PP(t,\cdot)}{h} &\rightarrow  \partial_{x_j} \PP(t,\cdot) \quad &&\mbox{weakly in $L^2(\Omega;\mathbb{R}^{d\times d})$, $\; j=1,\dots,d$, $\;\forall\,t \in (0,T_{\sigma_3}]$},\\
\frac{\chi_{\sigma_3}(\PP(t,\cdot + he_j))^{-1} - \chi_{\sigma_3}(\PP(t,\cdot))^{-1}}{h}
&\rightarrow  \partial_{x_j} \chi_{\sigma_3}(\PP(t,\cdot))^{-1} \; &&\mbox{weakly in $L^2(\Omega;\mathbb{R}^{d\times d})$, $\; j=1,\dots,d$, $\;$ for a.e. $t \in (0,T_{\sigma_3}]$}.
\end{alignat*}

Furthermore, since $\Delta_x \PP(t,\cdot) \in L^2(\Omega;\mathbb{R}^{d\times d})$ for a.e. $t \in (0,T_{\sigma_3}]$, $\PP$ satisfies a homogeneous Neumann boundary condition on $\partial \Omega$, and $\Omega$ is a $C^{2,\beta}$ domain, with $\beta\in(0,1)$, it follows by elliptic regularity theory (cf. Lemma 4.27 in \cite{N-book}) that $\PP(t,\cdot) \in W^{2,2}(\Omega;\mathbb{R}^{d\times d})$ for a.e. $t \in (0,T_{\sigma_3}]$. Hence,
\begin{alignat*}{2}
\frac{\partial_{x_i}\PP(t,\cdot + he_j) -\partial_{x_i}\PP(t,\cdot)}{h} &\rightarrow  \partial_{x_i}\partial_{x_j} \PP(t,\cdot) \qquad &&\mbox{weakly in $L^2(\Omega;\mathbb{R}^{d\times d})$, $\quad i,j=1,\dots,d,\quad $ for a.e. $t \in (0,T_{\sigma_3}]$}.
\end{alignat*}
Consequently,
\begin{alignat*}{2}
\frac{\PP(t,\cdot + he_j) -\PP(t,\cdot)}{h} &\rightarrow  \partial_{x_j} \PP(t,\cdot) \qquad &&\mbox{weakly in $W^{1,2}(\Omega;\mathbb{R}^{d\times d})$, $\quad j=1,\dots,d,\quad$ for a.e. $t \in (0,T_{\sigma_3}]$},
\end{alignat*}
and therefore
\begin{alignat*}{2}
\frac{\PP(t,\cdot + he_j) -\PP(t,\cdot)}{h} &\rightarrow  \partial_{x_j} \PP(t,\cdot) \qquad &&\mbox{strongly in $L^{r}(\Omega;\mathbb{R}^{d\times d})$, $\quad j=1,\dots,d,\quad$ for a.e. $t \in (0,T_{\sigma_3}]$},
\end{alignat*}
where $r \in [1,\infty)$ when $d=2$ and $r \in [1,6)$ when $d=3$; also, $\chi_{\Omega_{h,j}} \rightarrow 1$, strongly in $L^s(\Omega)$, as $h \rightarrow 0$, for all $s \in [1,\infty)$. Thus we deduce, with $r=s=4$, that
\begin{alignat*}{2}
\chi_{\Omega_{h,j}} \frac{\PP(t,\cdot + he_j) -\PP(t,\cdot)}{h} &\rightarrow  \partial_{x_j} \PP(t,\cdot) \qquad &&\mbox{strongly in $L^{2}(\Omega;\mathbb{R}^{d\times d})$, $\quad j=1,\dots,d,\quad$ for a.e. $t \in (0,T_{\sigma_3}]$},
\end{alignat*}
as $h \rightarrow 0$. Hence, as $h \rightarrow 0$,
\begin{align*}
\chi_{\Omega_{h,j}}(\cdot)\,\frac{\PP(t,\cdot + he_j) -\PP(t,\cdot)}{h}\,:\,\frac{\chi_{\sigma_3}(\PP(t,\cdot + he_j))^{-1} - \chi_{\sigma_3}(\PP(t,\cdot))^{-1}}{h} \rightarrow  (\partial_{x_j} \PP) : (\partial_{x_j} \chi_{\sigma_3}(\PP(t,\cdot))^{-1})\\ \qquad \mbox{weakly in $L^{1}(\Omega),\quad j=1,\dots,d,\quad$ for a.e. $t \in (0,T_{\sigma_3}]$}.
\end{align*}
Hence, by passing to the limit in \eqref{eq-ap3-4}, we have that
\begin{align*}
&\frac{1}{d}\,
\|\partial_{x_j}\tr(\log \chi_{\sigma_3}(\PP(t,\cdot)))\|^2_{L^2(\Omega)}
\leq - \int_\Omega (\partial_{x_j} \PP(t,x)):(\partial_{x_j} \chi_{\sigma_3}(\PP(t,x))^{-1})\, \dx, \quad
j=1,\dots,d,\quad \mbox{for a.e. $t \in (0,T_{\sigma_3}]$}.
\end{align*}
Finally, by summing over $j=1,\dots, d$, we deduce that
\begin{align*}
&\frac{1}{d}\,
\|\nabla_x\tr(\log \chi_{\sigma_3}(\PP(t,\cdot)))\|^2_{L^2(\Omega)}
\leq - \int_\Omega \nabla_x \PP(t,x) :: \nabla_x \chi_{\sigma_3}(\PP(t,x))^{-1}\, \dx \quad \mbox{for a.e. $t \in (0,T_{\sigma_3}]$}.
\end{align*}
That completes the proof of Lemma \ref{lem-lap-tau-inverses3}.
\end{proof}

\end{appendix}

\end{document}